\documentclass[reqno,11pt]{amsart}

\usepackage{xcolor}
\usepackage{graphicx}
\usepackage[hidelinks]{hyperref}

\usepackage{amscd}

\usepackage{amsmath}
\usepackage{amsthm}
\usepackage{amssymb}
\usepackage{latexsym,array}
\usepackage{amsfonts}
\usepackage{shadow}
\usepackage{amsbsy}
\usepackage{amssymb}
\usepackage{dsfont}
\usepackage{mathtools}
\usepackage{enumitem}
\usepackage{doi}

\usepackage{tikz-cd}
\usepackage{color}
\usepackage{graphicx}
\usepackage[hidelinks]{hyperref}
\usepackage{cleveref}
\usepackage{fullpage}
\usepackage{comment}

\usepackage{dsfont}

\usepackage{cleveref}

\numberwithin{equation}{section}

\newtheorem{theorem}{Theorem}[section]
\newtheorem{lemma}[theorem]{Lemma}
\newtheorem{corollary}[theorem]{Corollary}
\newtheorem{proposition}[theorem]{Proposition}

\theoremstyle{definition}
\newtheorem{remark}[theorem]{{\bf Remark}}
\newtheorem{definition}[theorem]{Definition}

\newcommand{\cc}{\mathbb{C}}

 \newcommand{\hn}{h_n}

\newcommand{\rr}{\mathbb{R}}

\newcommand{\unx}{\underline{x}}

\usepackage{amscd}
\usepackage{tikz-cd}
\usepackage{tikz}

\crefname{enumi}{}{}
\crefname{enumii}{}{}

\title[]{A unified approach to the Dirac fine structures on the $S$-spectrum and a connection with Gegenbauer polynomials}
\author[F. Colombo]{Fabrizio Colombo}
\address{(FC)
	Politecnico di Milano\\Dipartimento di Matematica\\Via E. Bonardi, 9\\20133
	Milano, Italy}
\email{fabrizio.colombo@polimi.it}

\author[A. De Martino]{Antonino De Martino}
\address{(ADM)
	Politecnico di Milano\\Dipartimento di Matematica\\Via E. Bonardi, 9\\20133
	Milano, Italy
} \email{antonino.demartino@polimi.it}

\author[S. Pinton]{Stefano Pinton}
\address{(SP)
	Politecnico di Milano\\Dipartimento di Matematica\\Via E. Bonardi, 9\\20133
	Milano, Italy
} \email{stefano.pinton@polimi.it}

\date{}

\begin{document}
	\maketitle
	\begin{abstract}
This paper contributes to the recently introduced theory of fine structures on the $S$-spectrum.
We study, in a unified way, the functional calculi for axially Poly-Analytic-Harmonic functions on the $S$-spectrum.
Axially Poly-Analytic-Harmonic functions of type $(\beta, m)$, for $\beta, m \in \mathbb{N}_0$ belong to the kernel of the
Dirac-Laplace operators $D^\beta\Delta^m$ of type $(\beta, m)$
and contain as particular cases axially Poly-Analytic and axially Poly-Harmonic functions.
By applying these operators to the Cauchy kernels $S^{-1}_L(s,x)$ of (left) slice hyperholomorphic functions,
we obtain an integral representation for axially Poly-Analytic-Harmonic functions.
We point out that the kernels $D^\beta\Delta^mS^{-1}_L(s,x)$ have a remarkable connection with Gegenbauer polynomials.
By replacing the paravector operator $T$ with commuting components in the kernels $D^\beta\Delta^m S^{-1}_L(s,x)$,
 we obtain the associated resolvent operators.
 With these resolvent operators, denoted by $S^{-1}_{L, D^\beta\Delta^m}(s,T)$, we define the associated functional calculi based on the $S$-spectrum and study their properties.
	\end{abstract}

	

\tableofcontents

\section{Introduction}

The spectral theory on the $S$-spectrum originated to provide quaternionic quantum mechanics, see \cite{adler,BF} with a precise mathematical foundation. It soon turned out to be the natural spectral theory for linear operators in vector analysis, such as the gradient operator and its variations, see \cite{BARACCOColomPelPint2022, Gradient}. More recently, it has found applications in differential geometry for the spectral theory of the Dirac operator on manifolds beyond the self-adjoint case, see \cite{DIRACHYPSPHE,DiracHarm} . For comprehensive discussions,
the reader is referred to the books  \cite{CGK, FJBOOK,ColomboSabadiniStruppa2011}.
This theory has proven to be significantly more general than initially anticipated. It naturally extends to full Clifford operators and reveals connections with the spectral theory based on the monogenic spectrum, which was developed by A. McIntosh and collaborators, see \cite{JBOOK,JM}.

\medskip
 A new branch of this spectral theory, called \textit{fine structures on the $S$-spectrum}, has recently emerged. This development involves function spaces linked to the Fueter-Sce mapping theorem, also called Fueter-Sce extension theorem, see \cite{Fueter, TaoQian1, Sce} and \cite{ColSabStrupSce,DDG, DDG1}, which in the Clifford algebra setting $\mathbb{R}_n$ with $n$ imaginary units $e_i$ for $i=1,\ldots,n$, bridges left slice hyperholomorphic functions
 $\mathcal{SH}_L(U)$, on a suitable open set $U\subseteq \mathbb{R}^{n+1}$, and (left) axially  monogenic functions $\mathcal{AM}_L(U)$, using the operator $\Delta^{\frac{n-1}{2}}$, where
 $\Delta$ is the Laplacian in $n+1$ dimensions.
Recalling that the Dirac operator $D$ and its conjugate $\overline{D}$ are defined by
	\begin{equation}\label{DIRACeBARDNN}
		D= \frac{\partial}{\partial x_0}+ \sum_{i=1}^{n} e_i \frac{\partial}{\partial x_i}, \ \ {\rm and} \ \
		\overline{D}= \frac{\partial}{\partial x_0}- \sum_{i=1}^{n} e_i \frac{\partial}{\partial x_i},
	\end{equation}
as a consequence of the  Fueter-Sce extension theorem we have
$$
D\Delta^{\frac{n-1}{2}}f(x)=0,\ \ \ {\rm for \ all }\ f\in \mathcal{SH}_L(U).
$$

This means that the map $x\mapsto \Delta^{\frac{n-1}{2}}f(x)$ belongs to the space $\mathcal{AM}_L(U)$.
Similar relation holds for right slice hyperholomorphic functions $\mathcal{SH}_R(U)$ and
 right axially monogenic functions $\mathcal{AM}_R(U)$.
For odd $n$, the operator $\Delta^{\frac{n-1}{2}}$ acts as a pointwise differential operator,
	while for even $n$, fractional powers of the Laplace operator are involved. In this
paper we focus on integer powers, for fractional powers completely different techniques are needed, see \cite{CDD}.
The operator $\Delta^{\frac{n-1}{2}}$ can be factorized in terms of the
Dirac operator $ D$ and its conjugate
 $\overline{D}$ since
   $$
   D\overline{D}=\overline{D}D=\Delta.
   $$
   So it is possible to repeatedly apply to a slice hyperholomorphic function $f(x)$
   the Dirac operator and its conjugate, until we reach the maximum power of the Laplacian, i.e., the Sce exponent $\frac{n-1}{2}$.
 This implies the possibility to build different sets of functions which lie between the set of slice hyperholomorphic functions and the set of axially monogenic functions.

\medskip

We will call {\em Dirac fine structures of the spectral theory on the $S$-spectrum} the collection
 of function spaces and
 the associated functional calculi
induced by all possible factorizations of the operator $\Delta^{\frac{n-1}{2}}$, in terms of $D$ and of its conjugate $\overline{D}$.
There are several possible factorizations such as
$$\Delta^{\frac{n-1}{2}}
= \underbrace{(D \overline{D}) \cdots (D  \overline{D})}_{(n-1)/2-\text{times}}
\ \ \ \ \ \ {\rm or} \ \ \ \ \ \ \
	\Delta^{\frac{n-1}{2}}
	= \underbrace{D \cdots D}_{(n-1)/2- \text{times}} \cdot
	\underbrace{\overline{D} \cdots \overline{D}}_{(n-1)/2-\text{times}}
$$
or recalling that $D\overline{D}=\Delta$, we can also get
\begin{equation}
	\label{fact}
	\Delta^{\frac{n-1}{2}}
	= D^{\frac{n-1}{2}-\gamma} \cdot \Delta^\gamma \cdot\overline{D}^{\frac{n-1}{2}-\gamma},\  \ \ {\rm for} \ \gamma\in \mathbb{N}_0\ \ {\rm such\ that}\ \ 0\leq \gamma\leq \frac{n-1}{2}.
\end{equation}

From the above considerations we observe that the building blocks of any Dirac-factorization of $\Delta^{\frac{n-1}{2}}$  in terms of $D$ and $\overline{D}$ are the operators
\begin{equation}\label{oper-D-Delta}
D^\beta \Delta^m \ \ \ {\rm and}\ \ \ \ \ \overline{D}^\beta \Delta^m
\end{equation}
for suitable $\beta$ and $m\in \mathbb{N}_0$.

\medskip
The operators in (\ref{oper-D-Delta}) lead to the definition of the spaces of
(left) axially Analytic-Harmonic functions $\mathcal{AAH}_{\beta,m}^L(U)$ of type $(\beta,m)$ for  $\beta$, $m\in \mathbb{N}_0$ (see Definition \ref{axpolC}).
Roughly speaking $\mathcal{AAH}_{\beta,m}^L$ consists of those functions $f: U\subseteq\mathbb{R}^{n+1} \to \mathbb{R}_n$ of class $ \mathcal{C}^{2m+\beta}(U)$
such that	
$$
D^\beta \Delta^m f(x)=0,\ \  \ \ {\rm for\ all}\ \ \  x\in U.
$$
Similarly for functions such that $\overline{D}^\beta \Delta^m f(x)=0$ for all $x\in U$ we define the spaces of
(left) axially Anti-Analytic-Harmonic functions
$\overline{\mathcal{AAH}}_{\beta,m}^L(U)$ of type $(\beta, m)$ (see in Definition \ref{antiax}).
Similar definitions hold for the right case $\mathcal{AAH}_{\beta,m}^R(U)$.

\medskip
The function spaces $ \mathcal{AAH}_{\beta,m}^L(U)$ of order $(\beta, m)$ contain as particular cases the set of (left) axially Poly-Harmonic functions $\mathcal{APH}^L_m(U)$ (see Definition \ref{polyharm})
	  and the (left) axially polyanalytic functions $\mathcal{APA}_\beta^L(U)$ (see Definition
	\ref{polyharm}). In Definition \ref{polyharm} we also give a complete description of all the function spaces contained in $ \mathcal{AAH}_{\beta,m}^L(U)$ as particular cases.

\medskip

To introduce integral representation of the function spaces of the fine Dirac fine structures on the $S$-spectrum
it is crucial to compute explicitly the
action of the differential operators
$ D^\beta \Delta^m$  and $\overline{D}^\beta \Delta^m $ give in (\ref{oper-D-Delta})
when applied to the left (resp. right) slice hyperholomorphic Cauchy kernel and considering the
 $S_L^{-1}(s,x)$ written in form II, for  $s,x\in \mathbb{R}^{n+1}$ with $x\not\in [s]$, i.e.,
\[
			S_L^{-1}(s,x):=(s-\bar x)\mathcal{Q}_{c,s}^{-1}(x),\ \  {\rm where}\ \ \
			\mathcal{Q}_{c,s}^{-1}(x):=(s^2-2{\rm Re}(x) s+|x|^2)^{-1}.
\]	
To obtain the explicit expressions of the maps
$$
(s,x)\mapsto D^k\Delta^m S_L^{-1}(s,x)
\ \ \
{\rm and}
\ \ \
(s,x)\mapsto  \overline{D}^k\Delta^m S_L^{-1}(s,x)
$$
is non trivial, but they can be written in closed form. In this paper we will mainly concentrate on the
left  slice hyperholomorphic Cauchy kernel
 $S_L^{-1}(s,x)$, but similar results hold for the right  slice hyperholomorphic Cauchy kernel
 $S_R^{-1}(s,x)$, see \cite{DPMilano}.

\medskip
Assume that $n$ is an odd number and denote by $h_n=\frac{n-1}{2}$ the Sce exponent.
For $s, x \in \mathbb{R}^{n+1}$ with $x \notin [s]$ and $\beta$, $m\in \mathbb{N}_0$
the computation of $D^\beta\Delta^m S_L^{-1}(s,x)$
and $\overline{D}^\beta\Delta^m S_L^{-1}(s,x)$ is based on the explicit expression of $\Delta^m S_L^{-1}(s,x)$ and then on
the actions of ${D}^\beta$ and $\overline{D}^\beta$, respectively, on $\Delta^m S_L^{-1}(s,x)$.

The first remarkable  fact
is the following result, proved in \cite{CSS}, that claims that for  $n$ odd, $h_n=\frac{n-1}{2}$, and $m\in \mathbb{N}$  we have
		\begin{equation}\label{DELTAEMME}
			\Delta^m S^{-1}_L(s,x)=4^m  m! (-h_n)_{m}
(s-\overline x)\mathcal{Q}_{c,s}^{-m-1}(x),
		\end{equation}
where $\Delta$ is considered with respect to the variable $x$,
and $(-h_n)_{m}$ is the Pochhammer symbol as in Definition \ref{Pochhsym}.
Using (\ref{DELTAEMME}) we obtain
\begin{equation}\label{DBETADELTA}
D^\beta\Delta^m(S^{-1}_L(s,x))
=4^m  m! (-h_n)_{m} D^\beta\big((s-\overline x)\mathcal{Q}_{c,s}^{-m-1}(x)\big)
\end{equation}
and  of $\overline{D}^\beta$ on $\Delta^m S_L^{-1}(s,x)$, i.e.
\begin{equation}\label{ANTIDBETADELTA}
\overline{D}^\beta\Delta^mS^{-1}_L(s,x)
=4^m  \ell! (-h_n)_{m} \overline{D}^\beta\big((s-\overline x)\mathcal{Q}_{c,s}^{-m-1}(x)\big).
\end{equation}
Secondly,  the action of the operators $D$,  $\overline{D}$ in the second hand side of the relations (\ref{DBETADELTA}) and (\ref{ANTIDBETADELTA}) can be computed iteratively by the crucial equalities (I) and (II), proved in \cite{CDP25,DPMilano},  given by:
\begin{itemize}
 \item[(I)]
 For  $n$ odd number, $h_n=\frac{n-1}{2}$ and $m\in \mathbb{N}$ we have
	\begin{align}
	D\left( (s-\bar x) \mathcal{Q}_{c,s}^{-m}(x) \right) & =-2(h_n-m+1) \mathcal{Q}_{c,s}^{-m}(x), \label{d_cauchy_ker}
\\
	D\left( Q_{c,s}^{-m}(x) \right) & =4m (s-x_0)\mathcal{Q}_{c,s}^{-m-1}(x)(x)-2m(s-\bar x) \mathcal{Q}^{-m-1}_{c,s}(x).
 \label{d2_cauchy_ker}
\end{align}
\item[(II)]
For  $n$ odd number, $h_n=\frac{n-1}{2}$
		  and $m \in \mathbb{N}$ we have
\begin{align}
			\overline D\left( (s-\bar x) \mathcal{Q}_{c,s}^{-m}(x) \right) & =2(h_n-\ell) \mathcal{Q}_{c,s}^{-m}(x)(x) +4\ell(s-\bar x )(s-x_0) \mathcal{Q}_{c,s}^{-m-1}(x),
\\
			\overline D\left( \mathcal{Q}_{c,s}^{-m}(x) \right) & =2\ell (s-\bar x)\mathcal{Q}_{c,s}^{-m-1}(x).
		\end{align}
\end{itemize}
In order to obtain $D^k\big((s-\overline x)\mathcal{Q}_{c,s}^{-m-1}(x)$ we reason by induction
using step (I). Similarly for $\overline{D}^k\big((s-\overline x)\mathcal{Q}_{c,s}^{-m-1}(x)$ we consider point (II).

\medskip
The explicit expression of $D^\beta\Delta^m S_L^{-1}(s,x)$
and $\overline{D}^\beta\Delta^m S_L^{-1}(s,x)$ obtained by (\ref{DBETADELTA}) and (\ref{ANTIDBETADELTA})
 are linear combinations of the kernels:
\begin{equation}
	K_{\nu,\ell}^{1,L}(s,x):=(s-\bar{x}) (s-x_0)^{\nu} \mathcal{Q}_{c,s}^{-\ell}(x), \qquad \ell \leq h_n, \quad \nu \in \mathbb{N},
\end{equation}
and
\begin{equation}
	K_{\nu,\ell}^{2}(s,x):= (s-x_0)^{\nu} \mathcal{Q}_{c,s}^{-\ell}(x), \qquad \ell \leq h_n, \quad \nu \in \mathbb{N},
\end{equation}
for $n \in \mathbb{N}$ being odd and $h_n=\frac{n-1}{2}$ that are the building blocks for all the kernels of the integral representations of the functions of the Dirac fine structures.

Consequently, all the resolvent operators for the related functional calculi are also based on these functions, where instead of $x$, we replace a paravector operator.
Theorem \ref{p111}, proved in \cite{DPMilano}, clarifies this fact.
We recall that $\beta$, $m\in \mathbb{N}_0$ are such that $m + \beta\leq h_n$.
Assuming that $s$, $x \in \mathbb{R}^{n+1}$ are such that $s \notin[x]$ and
 if $\beta=2k_1+1$, where $k_1\in\mathbb N_0$, then we have
	\begin{equation}
			D^\beta \Delta^m S^{-1}_{L}(s,x)= c_{\beta, m, h_n}\left(  \sum_{j=0}^{k_1-1} a^1_{j,k_1,m} K^{1,L}_{2j+1, m+j+2+k_1}(s,x)  -\sum_{j=0}^{k_1} b^1_{j,k_1,m}	K_{2j,m+1+k_1+j}^{2}(s,x)  \right),
	\end{equation}
	where the coefficients are $c_{\beta, m, h_n}$ are defined in (\ref{cbmhn}) and $a^1_{j,k_1,m} $ and $b^1_{j,k_1,m}$ are explicitly computed in (\ref{coeff}).
Similar expression holds for  $\beta=2k_2$, where $k_2\in\mathbb N$. Moreover, in Theorem \ref{barp} it is stated the analogous results for
$\overline{D}^\beta \Delta^m S^{-1}_{L}(s,x)$.

\medskip
Thanks to the explicit expression of the kernels
$D^\beta \Delta^m S^{-1}_{L}(s,x)$  and $\overline{D}^\beta \Delta^m
 S^{-1}_{L}(s,x)$,
the Cauchy formula  for slice hyperholomorphic functions recalled in (\ref{cauchynuovo})
and the Fueter-Sce mapping theorem 	%$D\Delta^{\frac{n-1}{2}}_{n+1}f(x)=0$
 we can give an integral representation of
  (left) axially Analytic-Harmonic functions $\mathcal{AAH}_{\beta,m}^L(U)$ of type $(\beta,m)$ for  $\beta$, $m\in \mathbb{N}_0$, see Definition  \ref{axpolC} and
of (left) axially  Anti-Analytic-Harmonic functions $\overline{\mathcal{AAH}}_{\beta,m}^L(U)$ of type $(\beta,m)$ for  $\beta$, $m\in \mathbb{N}_0$, see Definition  \ref{antiax}.
In fact, applying $D^\beta \Delta^m$ to a left slice hyperholomorphic function $f(x)$ we get
\begin{equation}\label{INTEGRALREP}
		D^\beta \Delta^mf(x)=\frac{1}{2 \pi}\int_{\partial (U\cap \mathbb{C}_I)} D^\beta \Delta^mS_L^{-1}(s,x)\, ds_I\,  f(s).
	\end{equation}
Such integral representations, and the analogue for $\overline{D}^\beta \Delta^m
f(x)$, allow to define the functional calculi of all possible
functions coming from the factorization of the $ \Delta^{\frac{n-1}{2}}_{n+1}$ in terms of the Dirac operator $D$ and its conjugate $\overline{D}$.
The explicit expressions of the kernels $D^\beta\Delta^m S_L^{-1}(s,x)$ and $\overline{D}^\beta\Delta^m S_L^{-1}(s,x)$ recalled above were obtained in \cite{DPMilano}, and the identity \eqref{INTEGRALREP} follows by simply exchanging the integral with the differential operator in the Cauchy formula. The new point is that the right hand side of \eqref{INTEGRALREP} is \emph{independent of the representative}: the function $f_{\beta,m}=D^\beta\Delta^m f$ does not determine $f$, since $f$ may be modified by any element of $\ker(D^\beta\Delta^m)\cap\mathcal{SH}_L(U)$, and one has to show that the integral does not change. This is the content of point (4) of Theorem \ref{Dinte}, and it is precisely for this reason that we prove Lemma \ref{aux} and Lemma \ref{kernel}, the latter characterizing $\ker(D^\beta\Delta^m)\cap\mathcal{SH}_L(U)$ as the set of polynomials of degree at most $2m+\beta-1$.

\medskip

A very important fact is the connection of the series expansion of the kernel
$\mathcal{Q}_{c,s}^{-\ell}(x)$ with Gegenbauer polynomials proved in Theorem \ref{exppser}.
In fact,
denoting by $h_n:= \frac{n-1}{2}$ the Sce exponent and let $n$ be an odd number, we show that
 for $s$, $x \in \mathbb{R}^{n+1}$ such that $|x|<|s|$ the function $\mathcal{Q}^{-\ell}_{c,s}(x)$, with $1 \leq \ell \leq h_n$, admits the series expansion:
	\begin{equation}
		\label{series11}
		\mathcal{Q}_{c,s}^{-\ell}(x)= \sum_{k=2 \ell-1}^{\infty} \mathcal{H}_\ell^k(x)s^{-k-1}
	\end{equation}
where the coefficients $\mathcal{H}_{\ell}^k(x)$ are defined as
$$
	\mathcal{H}_{\ell}^k(x):=\sum_{j=0}^{k-2\ell+1} \mathcal{T}^k_j(\ell) x^{k-2\ell+1-j} \bar{x}^{j},\quad \mathcal{T}^k_j(\ell):=\binom{\ell+j-1}{\ell-1} \binom{k-\ell-j}{\ell-1}
$$
and these polynomials are connected with the Gegenbauer polynomials  $C_n^{(\alpha)}$, see \cite{Fer, S39} and \eqref{Gen}, by the relation
$$
	\mathcal{H}_{\ell}^k(x)=|x|^{k-2\ell+1}C_{k-2\ell+1}^{\ell}\left(\frac{x_0}{|x|}\right),\qquad x\in\mathbb{R}^{n+1}\setminus\{0\},
$$
where $k\geq 2\ell-1$. The series expansion \eqref{series11} is the building block for obtaining series expansions of $D^\beta \Delta^m S_L^{-1}(s,x)$ and $\overline{D}^\beta \Delta^m S_L^{-1}(s,x)$; see Propositions~\ref{expser} and~\ref{expser1}.

\medskip

In order to explain the functional calculi based on the integral representations formulas (\ref{INTEGRALREP}) of the functions of the Dirac
fine structures we consider  bounded paravector operators
$T= \sum_{\mu=0}^{n}e_{\mu} T_{\mu}$, where $T_\mu \in \mathcal{B}(V)$ for $\mu=0,1,...,n$. The conjugate of the operator $T$ is given by $\overline{T}=T_0-\sum_{\mu=1}^{n} e_{\mu}T_{\mu}$. The subset of paravector
operators whose components commute among themselves will be denoted by
$\mathcal{BC}^{0,1}(V_n)$. If let $T \in \mathcal{BC}^{0,1}(V_n)$, where $V_n=V \otimes \mathbb{R}_n$,
the operator $T \bar{T}$ is well defined and is given by
$$
T\overline{T}=\overline{T}T= \sum_{\mu=0}^{n} T_\mu^2\ \ \ {\rm and}\ \ \ T+\bar{T}=2T_0.
$$
For operators in $\mathcal{BC}^{0,1}(V_n)$ the $S$-resolvent set of $T$ is defined as
$$ \rho_S(T)= \{s \in \mathbb{R}^{n+1} \, : \, \mathcal{Q}_{c,s}^{-1}(T):=(s^2-2sT_0+T\bar{T})^{-1} \in \mathcal{BC}(V_n)\},$$
and the $S$-spectrum of $T$ is
$$ \sigma_S(T)=\mathbb{R}^{n+1}\setminus \rho_S(T).$$

 We assume that $T \in \mathcal{BC}^{0,1}(V_n)$ and $s \in \mathbb{R}^{n+1}$ being such that $\|T\| < |s|$. Then, in Proposition \ref{Qint}, we proved that, for $n$ odd and $h_n:=\frac{n-1}{2}$, with $1\leq \ell\leq h_n$, we have
\begin{equation}
\mathcal{Q}_{c,s}^{-\ell}(T)= \sum_{k=2\ell-1}^{\infty} \mathcal{H}_{\ell}^k(T) s^{-k-1}.
\end{equation}
For $s \in \rho_S(T)$ and $T \in \mathcal{BC}^{0,1}(V_n)$ we define the $K_{\nu,\ell}^{1,L}(s,T)$ and $K_{\nu,\ell}^{2}(s,T)$ resolvent operators as
 \begin{equation*}\nonumber
 	K_{\nu,\ell}^{1,L}(s,T)=(s\mathcal{I}-\bar{T}) (s\mathcal{I}-T_0)^{\nu} \mathcal{Q}_{c,s}^{-\ell}(T),
\ \ \
 	K_{\nu,\ell}^{2}(s,T)= (s \mathcal{I}-T_0)^{\nu} \mathcal{Q}_{c,s}^{-\ell}(T).
 \end{equation*}

From the integral representations of the functions of the Dirac fine structures  (see Theorems \ref{Dinte} and \ref{Dbarinte}, respectively), we introduce the notions of the corresponding resolvent operators.
These can be expressed as linear combinations of the $K_{\nu,\ell}^{1,L}(s,T)$ and $K_{\nu,\ell}^{2}(s,T)$ resolvent operators and
setting for simplicity
$$
D^{\beta} \Delta^mS^{-1}_{L}(s,T):=S^{-1}_{L,D^{\beta} \Delta^m}(s,T),
$$
for example, if $\beta=2k_1+1$ where $k_1\in\mathbb N$, the resolvent $S^{-1}_{L,D^{\beta} \Delta^m}(s,T)$ is given by
	\begin{equation*}
	S^{-1}_{L,D^{\beta} \Delta^m}(s,T)= c_{\beta, m, h_n}\left(  \sum_{j=0}^{k_1-1} a^1_{j,k_1,m} K^{1,L}_{2j+1, m+j+2+k_1}(s,T)  -\sum_{j=0}^{k_1} b^1_{j,k_1,m}	K_{2j,m+1+k_1+j}^{2}(s,T)  \right),
	\end{equation*}
Observe that  all the functional calculi associated with the function spaces
 of the Dirac fine structures discussed above  are all based on the $S$-spectrum, i.e., all these resolvent operators  are explicit functions of the map
$$
s\mapsto (s^2\mathcal{I}-s(T+\overline{T})+T\overline{T})^{-1}, \ \ \ \text{for}\ \ \  s\in\rho_S(T).
$$

So if $m$ and $\beta \in \mathbb{N}$ are such that $\beta+m\leq h_n$ and
 $T \in \mathcal{BC}^{0,1}(V_n)$ be such that its components  $T_i$, $i=0,...,n-1$ have real spectra
  when we assume that $U$ be a bounded slice Cauchy domain and set $ds_I=ds(-I)$ for $I\in \mathbb{S}$ then
   for any $f \in \mathcal{SH}^L_{\sigma_S(T)}(U)$ we set
  $$
  f_{\beta, m}(x):=D^\beta \Delta^{m}f(x).
  $$
   We we can now define functional calculus for $f_{\beta, m}$ and for $T\in \mathcal{BC}^{0,1}(V_n)$ as
	\begin{equation}
		\label{RF1}
	f_{\beta, m}(T):=\frac{1}{2\pi} \int_{\partial(U \cap \mathbb{C}_I)}  S^{-1}_{L,D^\beta\Delta^m} (s,T)ds_I f(s).
\end{equation}
In Section \ref{FUNC-CAL_DIRAC}, we prove that the integral is independent of both the bounded slice Cauchy domain $U$ and the imaginary unit $I$ (see Theorem \ref{well1}), and that it is also independent of the representative $f$ (see Proposition \ref{indep_repr}). The latter point is the more delicate one, as it requires proving that the functional calculus \eqref{RF1} yields the zero operator whenever $f$ is a polynomial of degree at most $\beta+2m-1$. This result is established in Lemma \ref{blockzero} and Proposition \ref{zerothm}, with the monogenic functional calculus of McIntosh and collaborators playing a fundamental role. In Section \ref{Jacobi-series-expan} we prove the series expansions of the kernels $\mathcal{Q}^{-\ell}_{c,s}(x)$ and the connection with Gegenbauer polynomials.

\medskip 

The paper is organized as follows. Section 1 contains the introduction. Section \ref{Prel} recalls preliminary results on Clifford algebras and function theories. Section \ref{DIRAC-FINE_STRU} describes the Dirac fine structures on the $S$-spectrum, while Section \ref{Integral-representation} establishes the crucial integral representations for the functions arising from these fine structures. In Section \ref{Jacobi-series-expan} we prove the series expansions of the kernels $\mathcal{Q}^{-\ell}_{c,s}(x)$ and the connection with Gegenbauer polynomials. Finally, in Section \ref{FUNC-CAL_DIRAC}, we use these integral representations to define the associated functional calculi.

\section{Preliminary results on Clifford algebras and function theories}\label{Prel}
In this paper we will work within the framework of the real Clifford algebra $\rr_n$, constructed over $n$ imaginary units denoted as $e_1,\dots,e_n$,
satisfying the relations $e_ie_j+e_je_i=-2\delta_{ij}$. An element in the Clifford algebra $\mathbb{R}_n$ is represented as $\sum_A e_Ax_A$, where $r=0,...,n$, $A=\{ i_1\ldots i_r\}$, $i_\ell\in \{1,2,\ldots, n\}$, $i_1<\ldots <i_r$ forms a multi-index, $e_A=e_{i_1} e_{i_2}\ldots e_{i_r}$, and $e_\emptyset =1$. Notably, when $n=1$, $\rr_1$ corresponds to the commutative algebra of complex numbers $\mathbb{C}$, and for $n=2$, it yields the division algebra of real quaternions $\mathbb{H}$. However, for $n>2$, the Clifford algebras $\rr_n$ contain zero divisors.

Within the Clifford algebra $\rr_n$, certain elements can be identified with vectors in Euclidean space $\rr^n$: a vector $(x_1,x_2,\ldots,x_n)\in\rr^n$ maps to a "1-vector" in the Clifford algebra via $\unx=x_1e_1+\ldots+x_ne_n$. Additionally, an element $(x_0,x_1,\ldots,x_n)\in \rr^{n+1}$ is identified as a paravector, represented as $x=x_0+\unx=x_0+ \sum_{i=1}^nx_ie_i$, the conjugate of $x$ is defined as $\overline{x}:=x_0-\unx$.
The norm of $x\in\rr^{n+1}$ is given by $|x|^2=x_0^2+x_1^2+\ldots +x_n^2$, and the real part $x_0$ of $x$ is denoted also as ${\rm Re} (x)$.
We introduce the notation $\mathbb{S}$ for the sphere of unit 1-vectors in $\mathbb{R}^n$, defined as
$$
\mathbb{S}=\{ \unx=e_1x_1+\ldots +e_nx_n\ :\ x_1^2+\ldots +x_n^2=1\}.
$$
Note that $\mathbb{S}$ is an $(n-1)$-dimensional sphere in $\rr^{n+1}$.
The vector space $\mathbb{R}+I\mathbb{R}$, passing through $1$ and $I\in \mathbb{S}$, is denoted by $\mathbb C_I$, that is defined as
$$ \mathbb{C}_I=\{u+Iv \, | \,  u,v \in \mathbb{R}\}.$$
Notably, $\mathbb C_I$ is a 2-dimensional real subspace of $\rr^{n+1}$ isomorphic to the complex plane, with the isomorphism being an algebraic isomorphism.

%\noindent Given a paravector $x=x_0+\unx\in\rr^{n+1}$ let us set
%$$
%I_x=\left\{\begin{array}{l}
	%\displaystyle\frac{\unx}{|\unx|}\quad{\rm if}\ \unx\not=0,\\
%	{\rm any\ element\ of\ } \mathbb{S}{\rm\ otherwise,}\\
%\end{array}
%\right.
%$$
%so, by definition we have $x\in \mathbb C_{I_x}$.

\begin{definition}\label{sphere}
	Given an element $x\in\rr^{n+1}$, we define
	$$
	[x]=\{y\in\rr^{n+1}\ :\ y={\rm Re}(x)+I |\underline{x}|,\, I\in \mathbb{S}\}.
	$$
\end{definition}

\begin{remark}{\rm The set $[x]$ is a $(n-1)$-dimensional sphere in $\rr^{n+1}$. When $x\in\rr$, then $[x]$ contains $x$ only. In this case, the
		$(n-1)$-dimensional sphere has radius equal to zero. }\end{remark}

In this part of the section, we review the main notions of slice hyperholomorphic functions; see~\cite{CGK,ColomboSabadiniStruppa2011, CSS2016}. These types of functions are defined on the following kind of sets.

\begin{definition}
	Let  $U \subseteq \mathbb{R}^{n+1}$.
	\begin{itemize}
		\item We say that $U$ is {\em axially symmetric} if $[x]\subset U$  for any $x \in U$.
		\item We say that $U$ is a {\em slice domain} if $U\cap\rr\neq\emptyset$ and if $U\cap\cc_{I}$ is a domain in $\cc_{I}$ for any $I\in\mathbb{S}$.
	\end{itemize}
\end{definition}
\begin{definition}\label{AXIAL}
	Let $U\subseteq \mathbb R^{n+1}$ be an axially symmetric domain. We set
	\begin{equation}\label{twente}
		\mathcal U:=\{ (u,v)\in\mathbb R^2: \, u+Iv\in U \quad\forall I\in\mathbb S\}.
	\end{equation}
	We say that a function $f:U\to\mathbb R_n$ is a left (resp. right) axial function (or slice function) if there exist two functions $A,\, B:\mathcal U\to\mathbb R_n$ that satisfy the  even-odd conditions
	\begin{equation}
		\label{EO}
		A(u,v)=A(u,-v), \qquad B(u,v)=- B(u,-v), \qquad \forall  (u,v) \in \mathcal{U}.
	\end{equation}
	Moreover, the function $f$ has the following decomposition:
\begin{equation}
	\label{axial1}
	f(x)=A(u,v)+I\,B(u,v),
	\quad \left(\hbox{resp.} \ f(x)=A(u,v)+B(u,v)\,I\right), \quad \hbox{for } x=u+Iv, \ \ I= \frac{\underline{x}}{|\underline{x}|}.
\end{equation}
\end{definition}

\begin{remark}
In the literature the terminology axial function or slice function depend on the context they are used.
\end{remark}

\begin{definition}[Slice hyperholomorphic functions (or slice monogenic functions)]\label{sh}
	Let $U\subseteq \mathbb{R}^{n+1}$ be an axially symmetric open set and let $\mathcal U$ defined as in \eqref{twente}. A function $f:U\to \mathbb{R}^{n+1}$ is called a left (resp. right)
	slice hyperholomorphic (or slice monogenic) function, if it is slice (see Definition \ref{AXIAL}), so it can be written as
	\begin{equation}
		\label{form}
		f(x) = \alpha(u,v) + I\beta(u,v)\qquad \text{(resp. $f(x) = \alpha(u,v) + \beta(u,v)I$) \, for } x = u + I v\in U,
	\end{equation}
	and the functions $\alpha, \beta: \mathcal{U}\to \mathbb{R}_{n}$ satisfy the even-odd conditions (\ref{EO}) and the Cauchy-Riemann-equations
	\begin{equation}
		\label{CR}
\frac{\partial}{\partial u} \alpha(u,v)- \frac{\partial}{\partial v} \beta(u,v)=0, \quad \frac{\partial}{\partial v} \alpha(u,v)+ \frac{\partial}{\partial u} \beta(u,v)=0.
	\end{equation}
\end{definition}
\begin{definition}
	Let $U$ be an axially symmetric open set $\mathbb{R}^{n+1}$.
	\begin{itemize}
		\item
		We denote the sets of left
		slice hyperholomorphic functions (or left slice monogenic functions) on $U$ by $\mathcal{SH}_L(U)$
		while  right slice hyperholomorphic functions (or right slice monogenic functions) on $U$ will be denoted by $\mathcal{SH}_R(U)$. When no confusion arises, we refer to this class of functions as $\mathcal{SH}(U)$.
		
		\item
		A slice hyperholomorphic function \eqref{form} such that $\alpha$ and $\beta$ are real-valued functions is called intrinsic slice hyperholomorphic (or intrinsic slice monogenic functions) and will be denoted by $\mathcal{N}(U)$.
	\end{itemize}
	\begin{remark}
		The set of intrinsic left slice monogenic functions coincides with the set of right intrinsic slice monogenic functions, so when we mention the set $\mathcal{N}(U)$ we do not distinguish the left case from the right one.
	\end{remark}
\end{definition}
In the theory of slice monogenic functions, a crucial aspect is the representation formula,
which shows that any slice monogenic function, as in Definition \ref{sh},
on an axially symmetric set can be entirely characterized by its values on two complex planes $\mathbb{C}_I$, for $I\in \mathbb{S}$, because of the  book structure of $\mathbb{R}^{n+1}$, i.e.,
$\mathbb{R}^{n+1}=\bigcup_{I\in \mathbb{S}} \mathbb{C}_I$.

Now, we recall that also for slice hyperholomorphic functions holds a counterpart of the Cauchy integral theorem, see \cite{CGK, ColomboSabadiniStruppa2011}.
\begin{theorem}
	\label{Cif}
	Let $U \subset \mathbb{R}^{n+1}$ be an open set and $I \in \mathbb{S}$. We assume that $f$ is a left slice hyperholomorphic function and $g$ is a right slice hyperholomorphic function in $U$. Furthermore, let $D_I \subset U \cap \mathbb{C}_I$
	be an open and bounded subset of $ \mathbb{C}_I$ with $ \overline{D}_I \subset U \cap \mathbb{C}_I$
	such that $\partial D_I$ is a finite union of piecewise continuously differentiable Jordan curves. Then, for $ds_I=ds(-I) $, we have
	$$ \int_{\partial D_I}g(s)ds_I f(s)=0.$$
\end{theorem}

We now revisit the slice hyperholomorphic Cauchy formulas, which serve as the foundation for developing  the hyperholomorphic spectral theories on the $S$-spectrum. The following result will be required, see \cite{CGK, ColomboSabadiniStruppa2011}.
\begin{proposition}[Cauchy kernel series]
	\label{cauchyseries}
	Let $s$, $x \in \mathbb{R}^{n+1}$ such that $|x|<|s|$ then
	$$ \sum_{n=0}^{\infty}x^n s^{-1-n}=-(x^2 -2x {\rm Re} (s)+|s|^2)^{-1}(x-\overline s)$$
	and
	$$ \sum_{n=0}^{\infty}s^{-1-n}x^n=-(x-\overline s)(x^2 -2x {\rm Re} (s)+|s|^2)^{-1}.$$
\end{proposition}
An important fact for our considerations is that the sum of the Cauchy kernel series can be written in two equivalent ways by the following well known result.
\begin{proposition}\label{secondAA}
	Suppose that $x$ and $s\in\rr^{n+1}$ are such that $x\not\in [s]$. We have that
	\begin{equation}\label{second}
		-(x^2 -2x {\rm Re} (s)+|s|^2)^{-1}(x-\overline s)=(s-\bar x)(s^2-2{\rm
			Re}(x)s+|x|^2)^{-1},
	\end{equation}
	and
	\begin{equation}\label{third}
		(s^2-2{\rm Re}(x)s+|x|^2)^{-1}(s-\bar x)=-(x-\bar s)(x^2-2{\rm Re}(s)x+|s|^2)^{-1} .
	\end{equation}
\end{proposition}
 Proposition \ref{cauchyseries} justifies the following definitions.
\begin{definition}[Slice hyperholomorphic Cauchy kernels]
	\label{Ckernel}
	Let $x$, $s\in \rr^{n+1}$ such that $x\not\in [s]$. We say that  $S_L^{-1}(s,x)$ (resp.  $S_R^{-1}(s,x)$) is written in the form I if
	$$
	S_L^{-1}(s,x):=-(x^2 -2x {\rm Re} (s)+|s|^2)^{-1}(x-\overline s), \quad \left(	\hbox{resp.} \, \,	S_R^{-1}(s,x):=-(x-\bar s)(x^2-2{\rm Re}(s)x+|s|^2)^{-1} \right).
	$$
	We say that $S_L^{-1}(s,x)$ (resp.  $S_R^{-1}(s,x)$) is written in the form II if
	$$
	S_L^{-1}(s,x):=(s-\bar x)(s^2-2{\rm Re}(x) s+|x|^2)^{-1}, \quad \left(\hbox{resp.} \, \,S_R^{-1}(s,x):=(s^2-2{\rm Re}(x)s+|x|^2)^{-1}(s-\bar x)\right).
	$$
\end{definition}
\begin{remark}
	In the following the polynomial
	\begin{equation}\label{polynomQCSX}
		\mathcal{Q}_{c,s}(x):=s^2-2{\rm Re}(x) s+|x|^2, \ \ \ s,x\in \mathbb{R}^{n+1}
	\end{equation}
	and its inverse, defined for $s,x\in \mathbb{R}^{n+1}$ with $x\not\in [s]$,
	\begin{equation}\label{QCSX}
		\mathcal{Q}_{c,s}^{-1}(x):=(s^2-2{\rm Re}(x) s+|x|^2)^{-1}
	\end{equation}
	play a crucial role in several results. The term $\mathcal{Q}_{c,s}^{-1}(x)$, that appears in the Cauchy kernel in form II  of slice monogenic functions, is often called commutative pseudo-Cauchy kernel to distinguish it from the pseudo-Cauchy kernel that is given by
	$\mathcal{Q}_{s}^{-1}(x):=(x^2-2{\rm Re}(s)x+|s|^2)^{-1}$, defined for $s,x\in \mathbb{R}^{n+1}$ with
	$x\not\in [s]$.
\end{remark}
\begin{remark}
	Thanks to Proposition \ref{secondAA}, the Cauchy kernels can be expressed in two equivalent ways from the perspective of the Clifford Algebra.
	However, this equivalence does not hold when the paravector $x$
	is replaced by a paravector operator.
	Specifically, the Cauchy kernels in form I are well-suited for paravector
	operators $T=\sum_{\ell=0}^n T_\ell e_\ell$, where the operators $T_\ell$, for $\ell=0,\cdots,n$,
	do not commute among themselves.
	In contrast, the Cauchy kernels in  form II requires
	the operators $T_\ell$, for $\ell=0,\cdots,n$, to commute among themselves.
	This commuting restriction on the operators has the advantage of allowing the resolvent operators of the fine structure to be computed explicitly in a closed form.
\end{remark}

The Cauchy formula for slice hyperholomorphic functions is a crucial tool
for defining the $S$-functional calculus.
\begin{definition}[Slice Cauchy domain]
	A slice axially symmetric open set $U\subset \mathbb{R}^{n+1}$ is called a slice Cauchy domain,
	if $U\cap\cc_I$ is a Cauchy domain in $\cc_I$ for any $I\in\mathbb{S}$.
	More precisely, $U$ is a slice Cauchy domain if for any $I\in\mathbb{S}$
	the boundary ${\partial( U\cap\cc_I)}$ of $U\cap\cc_I$ is the union of a finite number of non-intersecting piecewise continuously differentiable Jordan curves in $\cc_{I}$.
\end{definition}
\begin{theorem}[The Cauchy formulas for slice hyperholomorphic functions, see \cite{CGK, ColomboSabadiniStruppa2011}]
	\label{Cauchy}
	Let $U\subset\mathbb{R}^{n+1}$ be a bounded slice Cauchy domain, let $I\in\mathbb{S}$ and set  $ds_I=ds (-I)$.
	If $f$ is a left (resp. right) slice hyperholomorphic function on a set that contains $\overline{U}$ then
	\begin{equation}\label{cauchynuovo}
		f(x)=\frac{1}{2 \pi}\int_{\partial (U\cap \mathbb{C}_I)} S_L^{-1}(s,x)\, ds_I\,  f(s),\quad \left(f(x)=\frac{1}{2 \pi}\int_{\partial (U\cap \mathbb{C}_I)}  f(s)\, ds_I\, S_R^{-1}(s,x)\right),
	\end{equation}
	for any $x \in U$. Moreover, the integrals in \eqref{cauchynuovo}  depend neither on $U$ nor on the imaginary unit $I\in\mathbb{S}$.
\end{theorem}

Alongside the class of slice hyperholomorphic functions, another prominent class of functions in hypercomplex analysis is that of axially monogenic functions.

\begin{definition}\label{AXIALM}
	Let $U \subseteq \mathbb{R}^{n+1}$ be an axially symmetric domain.
	A function $f:U \to \mathbb{R}^{n+1}$ of class $ \mathcal{C}^1$ is said to be left (resp. right)  axially monogenic if it is monogenic:
	\begin{equation}\label{LEFTMONFUNC}
		\! \! \! \! \!	Df(x)= \left(\frac{\partial}{\partial x_0}+ \sum_{i=1}^{n} e_i \frac{\partial}{\partial x_i} \right)f(x)=0, \quad \left(\hbox{resp.} \, \,f(x)D=\frac{\partial}{\partial x_0}f(x)+ \sum_{i=1}^{n} \frac{\partial}{\partial x_i}f(x) e_i  =0\right),
	\end{equation}
	i.e. it is in the kernel of the Dirac operator $D$, and if in addition it has a left (resp. right) axial form, see Definition \ref{AXIAL}.
	The class of left (resp. right) axially monogenic is denoted by $ \mathcal{AM}_L(U)$ (resp. $\mathcal{AM}_R(U)$). When no confusion arises, we refer to this class of functions as $ \mathcal{AM}$.
	
\end{definition}
These functions have been extensively studied in the following books \cite{red, green}. Well-known examples of axially monogenic functions are the Clifford-Appell polynomials:

\begin{equation}
	\label{ca}
	\mathcal{Q}_n^k(x)= \frac{(2 h_n)!}{k!} \sum_{k=0}^{k-2h_n} \binom{k-j-h_n}{h_n} \binom{h_n+j-1}{h_n-1}x^{k-2h_n-j} \bar{x}^j, \qquad k \geq 2h_n,
\end{equation}
where $h_n=\frac{n-1}{2}$ with $n$ being an odd number. These polynomials were introduced in \cite{CFM,CMF} and have been extensively studied in \cite{DDG, DDG1}. They have found applications in various areas, such as Schur analysis (see \cite{ACDDS}) and in connection with Fock and Hardy spaces (see \cite{AKS3}).

\begin{remark}
A Cauchy formula has also been established for the class of axially monogenic functions. This formula forms the basis of the monogenic functional calculus developed by McIntosh and collaborators; see~\cite{JBOOK, JM}.
\end{remark}

The connection between slice hyperholomorphic functions and axially monogenic functions is a crucial tool in hypercomplex analysis and is known as the Fueter–Sce mapping theorem. This result was originally proved in~\cite{Fueter} for quaternions and was later generalized to the Clifford setting; see~\cite{Sce} and~\cite{ColSabStrupSce} for an English translation.

\begin{theorem}[Fueter-Sce mapping theorem (also called Fueter-Sce extension theorem)]
	\label{FS1}
	Let $n$ be an odd number and set $h_n:=(n-1)/2$. We assume that
		$f_0(z)=A(u,v)+i B(u,v)$ is a holomorphic function of one variable defined in a domain
		$\Pi\subseteq\mathbb{C}$ which is symmetric with respect to the real axis, and we assume that
		its components $A$ and $B$ satisfy the even-odd conditions \eqref{EO}. Lets us consider
	$$ U_{\Pi}:= \{x=x_0+ \underline{x} \ \  : \ \ (x_0, | \underline{x}|) \in \Pi\},$$
	the open set induced by $\Pi$ in $ \mathbb{R}^{n+1}$. The operator $T_{FS1}$, called slice operator, and defined by
	$$
	f(x)=T_{FS1}(f_0):= A(u,v)+ I\,B(u,v), \ \ \ x=u+Iv\in  U_{\Pi}
	$$
	maps the set of holomorphic functions in the set of slice hyperholomorphic functions. Furthermore,
	the operator
	$$
	T_{FS2}:=\Delta^{h_n}
	$$
	maps the slice hyperholomorphic functions $f(x)=T_{FS1}(f_0)$ into the set of axially monogenic function, i.e.,
	the function
	$$
	\breve{f}(x):=\Delta^{h_n} \left(A(u,v)
	+ I\,B(u,v) \right), \ \ \ x=u+Iv\in  U_{\Pi}
	$$
	is an axial function and in the kernel of the Dirac operator.
\end{theorem}

\begin{remark}
Theorem \ref{FS1} was originally stated for holomorphic functions defined on a domain contained in the upper half of the complex plane. However, by considering domains symmetric with respect to the real axis and components of the holomorphic function satisfying the even-odd conditions \eqref{EO}, we obtain the formulation needed to produce slice hyperholomorphic functions on axially symmetric open sets.
\end{remark}

\section{The Dirac fine structures}\label{DIRAC-FINE_STRU}

Poly-Analytic and Poly-Harmonic functions are well-known classes of functions that have been extensively studied in \cite{Aro} and \cite{Balk}, respectively. In particular, Poly-Analytic functions have found applications in time-frequency analysis \cite{A1, FISH}, elasticity theory \cite{Russ}, operator theory \cite{ACDS}, and in the study of duality theorems \cite{CDKSS}. Moreover, in \cite{Begehr2}, the authors study boundary value problems for the inhomogeneous Poly-Analytic equation, while in \cite{B1976} Poly-Analytic functions were first investigated in the quaternionic setting.

\medskip

We now introduce more general classes of functions that include Poly-Harmonic and Poly-Analytic functions, naturally arising from the Fueter-Sce mapping theorem.

\begin{definition}[Axially Analytic-Harmonic functions $\mathcal{AAH}_{\beta,m}(U)$ of type $(\beta,m)$]
	\label{axpolC}
Let $U$ be an open set in $\mathbb{R}^{n+1}$, and let $\beta,m\geq 0$ be integers such that $\beta$ and $m$ are not simultaneously zero.
 A function $f: \mathbb{R}^{n+1} \to \mathbb{R}_n$ of class $ \mathcal{C}^{2m+\beta}(U)$ is said to be left (resp. right) axially Analytic-Harmonic of type $(\beta, m)$ if it is an axial function and if we have
	$$ D^\beta \Delta^m f(x)=0, \qquad \forall x \in U, \quad \left(f(x)\Delta^m D^\beta=0, \quad \forall x \in U \right).$$
	We denote this class of functions by $ \mathcal{AAH}_{\beta,m}^L(U)$ (resp. $ \mathcal{AAH}_{\beta,m}^R(U)$). Whenever it is clear from the context, we use the notation $ \mathcal{AAH}_{\beta,m}(U)$  or by $ \mathcal{AAH}_{\beta,m}$  to represent this class of functions.	
\end{definition}
\begin{definition}[Axially Anti-Analytic-Harmonic  $\overline{\mathcal{AAH}_{\beta,m}}(U)$ of type $(\beta,m)$]
	\label{antiax}
	Let $U$ be an open set in $\mathbb{R}^{n+1}$ and let $\beta \geq 0$, $m  \geq 1$  be integers. A function $f: \mathbb{R}^{n+1} \to \mathbb{R}_n$ of class $ \mathcal{C}^{2m+\beta}(U)$ is said to be left (resp. right) axially Anti-Analytic-Harmonic of type $(\beta, m)$ if it is  an axial function and if we have
	$$ \overline{D}^\beta \Delta^m f(x)=0, \qquad \forall x \in U, \quad \left(f(x)\Delta^m \overline{D}^\beta=0, \quad \forall x \in U \right).$$
	We denote this class of functions by $ \overline{\mathcal{AAH}}_{\beta,m}^L(U)$ (resp. $ \overline{\mathcal{AAH}}_{\beta,m}^R(U)$). Whenever it is clear from the context, we use the notation $\overline{\mathcal{AAH}}_{\beta,m}(U)$ or by $\overline{\mathcal{AAH}}_{\beta,m}$ to represent this class of functions.
\end{definition}

\begin{remark}
If we do not assume that the functions are axial functions, we can define the more general classes of Analytic-Harmonic and Anti-Analytic-Harmonic functions of type $(\beta,m)$. However, in this paper we restrict ourselves to axial functions, since we work within the framework of the Fueter-Sce mapping theorem and slice hyperholomorphic functions, which are axial functions.
\end{remark}

In the following definition we consider specific cases of $\beta$ and $m$ of Definitions \ref{axpolC} and \ref{antiax}, which give well known function spaces.

\begin{definition}
		\label{polyharm}
		Let $U\subseteq\mathbb{R}^{n+1}$ be an open set. Inside the classes of Definition \ref{axpolC} and
		Definition \ref{antiax} we single out the following classical spaces, for which we keep the symbols
		used in the literature:
		\begin{itemize}
			\item[(i)] the left (resp. right) \emph{axially Poly-Harmonic functions of order $m$},
			$\mathcal{APH}_m(U):=\mathcal{AAH}_{0,m}(U)=\overline{\mathcal{AAH}}_{0,m}(U)$, i.e. the axial
			functions of class $\mathcal{C}^{2m}$ with $\Delta^m f=0$ on $U$;
			\item[(ii)] the left (resp. right) \emph{axially Poly-Analytic functions of order $\beta$},
			$\mathcal{APA}_\beta(U):=\mathcal{AAH}_{\beta,0}(U)$, i.e. the axial functions of class
			$\mathcal{C}^{\beta}$ with $D^\beta f=0$ on $U$;
			\item[(iii)] the left (resp. right) \emph{axially monogenic functions}
			$\mathcal{AM}(U)=\mathcal{APA}_1(U)=\mathcal{AAH}_{1,0}(U)$ (see Definition \ref{AXIALM}) and the \emph{axially Harmonic
				functions} $\mathcal{AH}(U)=\mathcal{APH}_1(U)=\mathcal{AAH}_{0,1}(U)$.
		\end{itemize}
		As above, the left and right classes are denoted by the superscripts $L$ and $R$, which are omitted
		when no confusion arises. The functions in (ii) were first studied in \cite{B1976} and are also
		called poly-monogenic functions.
\end{definition}

\begin{remark}
In this paper, we focus our attention on the spaces defined above and in Definitions \ref{axpolC} and \ref{antiax}. However, we do not consider the spaces $\overline{\mathcal{AAH}}_{\beta,0}(U)$, which consist of functions in the kernel of $\overline{D}^\beta$. This is due to the fact that these function spaces have not yet been considered in the literature in connection with functional calculi, see Remark \ref{new02}. Moreover, the construction of the corresponding functional calculi would require further and more involved computations, which we leave for future work.
\end{remark}

The function spaces defiend in Definitions \ref{axpolC} and \ref{antiax} appear in various works in the literature; see, for example, \cite{L, LL, LR}. In \cite{Fivedim}, the authors established a connection between these spaces in dimension five, while in \cite{CDD} the authors uncovered a relation between axially monogenic functions and the counterparts of the above spaces in the fractional setting in dimension three. However, the relationship in general dimensions remained unexplored. In this paper, we establish a connection between all the function spaces introduced above within the context of the Dirac fine structures on the $S$-spectrum.

\medskip
Now suppose that $n$ is an odd number and $h_n := \frac{n-1}{2}$ is the Sce exponent.
In order to give the crucial definition of Dirac fine structures we recall that
all possible factorizations of the  differential operator $ \Delta^{h_n}$,
in terms of the Dirac operator $D$, its conjugate $\overline{D}$ and $\Delta$ are given by the operators
	\begin{equation}
		\label{T1}
		 D^\beta \Delta^m, \quad \beta, m \in \mathbb{N}_0, \quad \beta+m \leq h_n
	\end{equation}
	 and
	\begin{equation}
		\label{T2}
		 \overline{D}^\beta \Delta^m, \quad \beta, m \in \mathbb{N}_0, \quad \beta+m \leq h_n,
	\end{equation}
so that by the Fueter-Sce mapping theorem we have
$$
D\Delta^{h_n}f(x)=0 \ \ \ {\rm for \ all }\ f\in \mathcal{SH}_L(U)\ \ \ \  ({\rm resp.} \ \
f(x)D\Delta^{h_n}=0\ \ \ {\rm for \ all }\ f\in \mathcal{SH}_R(U)).
$$
As already announced in the introduction, we call \emph{Dirac fine structures on the $S$-spectrum} the collection of all axial functions belonging to the kernels of the operators \eqref{T1} and \eqref{T2}, together with their associated functional calculi, defined via the integral representations based on the Cauchy formula for slice hyperholomorphic functions.
\begin{definition}
Let  $h_n := \frac{n-1}{2}$ be the Sce exponent and let  $\beta, m \in \mathbb{N}_0
$ with
$\beta+m \leq h_n$.
Fine structures arising from the factorization of $ \Delta^{h_n}$ with the operators
 defined in (\ref{T1}) are call $D$-fine structure.
 Similarly the factorizations obtained via the operators defined in (\ref{T2})
are called the $\overline{D}$-fine structure.
\end{definition}
In the following result we study the action of the operators (\ref{T1}) and (\ref{T2}) when applied to an axial function, see Definition \ref{AXIAL}.

\begin{proposition}
	\label{axial}
	Let $U \subseteq \mathbb{R}^{n+1}$ be an axially symmetric open set and
$h_n:=\frac{n-1}{2}$. We assume that $f$ is a left (resp. right) axial function
on $U$ of class $\mathcal{C}^{2m+\beta}$. Then, for any $\beta$,
	$m \in \mathbb{N}_0$, we have that $\overline{D}^\beta \Delta^mf(x)$ (resp. $f(x)\overline{D}^\beta \Delta^m$) and $D^\beta \Delta^mf(x)$ (resp. $f(x)D^\beta \Delta^m$) are left (resp. right) axial functions on $U$. 
\end{proposition}

\begin{proof}
	We focus only on the case where the function $f$ is left axial, as the right axial case follows from similar arguments. By hypothesis we can write $f(x)=\alpha(u,v)+I \beta(u,v)$, for $x=u+Iv$ with $I \in \mathbb{S}$. By \cite[thm. 11.33]{GHS} and \cite{CDP25} we know that
	$$ \Delta^mf(x)=2^m(h_n-m+1)_m \left[\left(\frac{1}{v} \frac{\partial}{\partial v}\right)^m\alpha(u,v)+I \left(\frac{\partial}{\partial v} \frac{1}{v}\right)^m \beta(u,v)\right],$$
	where $(h_n-m+1)_m$ is the Pochhammer symbol. Now, we prove that the above components of the axial function satisfy the even-odd conditions, see \eqref{EO}. For this reason we set
	$$\gamma_m(u,v):=\left(\frac{1}{v} \frac{\partial}{\partial v}\right)^m\alpha(u,v), \qquad \delta_m(u,v):=\left(\frac{\partial}{\partial v} \frac{1}{v}\right)^m \beta(u,v).$$
	We prove that the function $\gamma_m(u,v)$ is even in the variable $v$ by induction on $m$. By using the fact that $\alpha(u,v)$ is even in the variable $v$ and Differentiating with respect to $v$, we obtain
	$$\partial_v\alpha(u,-v)=-\partial_v\alpha(u,v).$$
	Therefore,
	$$	\gamma_1(u,-v)=\frac{1}{-v}\partial_v\alpha(u,-v)=\frac{1}{v}\partial_v\alpha(u,v)=\gamma_1(u,v).$$
	Now, by inductive hypothesis we assume $\gamma_m(u,-v)=\gamma_m(u,v)$. Differentiating with respect to $v$, we get
	$$\partial_v\gamma_m(u,-v)=-\partial_v\gamma_m(u,v).$$
	Thus,
	$$\gamma_{m+1}(u,-v)=\frac{1}{-v}\partial_v\gamma_m(u,-v)=\frac{1}{v}\partial_v\gamma_m(u,v)=\gamma_{m+1}(u,v).$$
	We prove by induction on $m$ that the function $\delta_m(u,v)$ is odd in the variable $v$.
	Since, we are assuming that $\beta(u,v)$ is odd in $v$, we have that $h(u,v):=\frac{\beta(u,v)}{v}$ is even in $v$. Differentiating with respect to $v$, we obtain
	$$\partial_vh(u,-v)=-\partial_vh(u,v).$$
	Since $\delta_1(u,v)= \partial_v h(u,v)$ we get
	$$ \delta_1(u,-v)=\partial_v h(u,-v)=-\partial_v h(u,v)=-\delta_1(u,v).$$
	Now, by the inductive hypotehsis we have $\delta_m(u,-v)=-\delta_m(u,v)$. This means that the function $h_m(u,v)=\frac{\delta_m(u,v)}{v}$ is even in $v$. Differentiating with respect to $v$, we get
	$$\partial_vh_m(u,-v)=-\partial_vh_m(u,v).$$
	Since $\delta_{m+1}(u,v)= \partial_v h_m(u,v)$ we get
	$$\delta_{m+1}(u,-v)=(\partial_v h_m)(u,-v)=-(\partial_v h_m)(u,v)=-\delta_{m+1}(u,v).$$
	This means that $g(x):=\Delta^m f(x)$ is an axial function, see Definition \ref{AXIAL}. Thus, it remains to prove that $\overline{D}^\beta g$ is an axial function. We proceed by induction on $\beta$. For $\beta=1$, by \cite{D}, for $x=u+Iv$ we can write
	\begin{equation}
		\label{dbar}
		\overline{D}g(x)=A(u,v)+IB(u,v),
	\end{equation}
	where
	$$A(u,v)=\frac{\partial}{\partial u}\gamma_m(u,v)+\frac{\partial}{\partial v}\delta_m(u,v) +\frac{2h_n}{v}\delta_m(u,v), \quad \hbox{and} \quad B(u,v)=\frac{\partial}{\partial u}\delta_m(u,v)-\frac{\partial}{\partial v}\gamma_m(u,v).$$
	Since $\gamma_m(u,v)$ is even and $\delta_m(u,v)$ is odd in the variable $v$, we have
	$$\partial_u\gamma_m(u,-v)=\partial_u\gamma_m(u,v),\qquad \partial_v\delta_m(u,-v)=\partial_v\delta_m(u,v), \qquad \frac{\delta_m(u,-v)}{-v}=\frac{\delta_m(u,v)}{v}.$$
	Hence,
	$$A(u,-v)=A(u,v).$$
	Moreover,
	$$\partial_u\delta_m(u,-v)=-\partial_u\delta_m(u,v),\qquad\partial_v\gamma_m(u,-v)=-\partial_v\gamma_m(u,v),$$
	and therefore
	$$B(u,-v)=-B(u,v).$$
	Thus, $\overline{D}g$ is an axial function. Now suppose that the statement holds for $\beta$. Then
	$$\overline{D}^{\beta}g(x)=A_\beta(u,v)+IB_\beta(u,v),$$
	where $A_\beta$ is even and $B_\beta$ is odd in $v$. Applying \eqref{dbar} to $\overline{D}^{\beta}g$, the same argument shows that
	$$\overline{D}^{\beta+1}g(x)=A_{\beta+1}(u,v)+IB_{\beta+1}(u,v),$$
	where $A_{\beta+1}$ is even and $B_{\beta+1}$ is odd in $v$. Hence, by induction, $\overline{D}^{\beta}\Delta^m f$ is an axial function. The proof that $D^\beta\Delta^m f$ is  an axial function follows analogously.
\end{proof}

The application of the differential operators $D^\beta\Delta^m$ and $\overline{D}^\beta\Delta^m$ to a slice hyperholomorphic function $f$ yields axial functions possessing the regularity properties introduced in Definitions \ref{antiax} and \ref{axpolC}, as shown in the following two results.

\begin{theorem}
	\label{reg0}
	Let $n$ be an odd number and set $h_n= \frac{n-1}{2}$. Assume that $\beta$, $m \in \mathbb{N}$ are such that $ \beta+m\leq h_n$, then the action of the operator $D^\beta \Delta^m$, on a left (resp. right) slice hyperholomorphic function $f$ defined on an axially symmetric open set $U \subseteq \mathbb{R}^{n+1}$ lead to:
	\begin{itemize}
		\item[1)] a left (resp. right) axially poly-harmonic function $\mathcal{APH}_{h_n-m}(U)$ of order $h_n-m$ if $\beta=1$,
		\item[2)] a left (resp. right) axially Anti-Analytic-Harmonic functions $\overline{\mathcal{AAH}}_{\beta-1, h_n-m-\beta+1}(U)$  of type $(\beta-1, h_n-m-\beta+1)$.
	\end{itemize}
\end{theorem}
\begin{proof}
We concentrate on the case where the function $f $ is left slice hyperholomorphic; the right slice hyperholomorphic case can be handled using similar arguments.

	We set $g(x):=D^\beta \Delta^m f(x)$ and we observe that the function $g$ is an axial function by Proposition \ref{axial}.
Moreover, this function belongs to the kernel of the operator
	\begin{equation}
		\label{oper}
		\overline{D}^{h_n-m}D^{h_n-\beta-m+1}.
	\end{equation}
	Indeed by the Fueter-Sce mapping theorem (see Theorem \ref{FS1}) we have
	$$
	\overline{D}^{h_n-m}D^{h_n-\beta-m+1} g(x)= \overline{D}^{h_n}D^{h_n+1}f(x)=D \Delta^{h_n}f(x)=0.
	$$
	We are now in a position to prove all the points outlined in the statement.

	\begin{itemize}
		\item[1)] If $\beta=1$ the operator in \eqref{oper} becomes $ \Delta^{h_n-m}$. Thus the function $g$ is poly-harmonic of order $h_n-m$, see Definition \ref{polyharm}.
		\item[2)] We observe that
		$$
		\overline{D}^{h_n-m}D^{h_n-\beta-m+1}= \overline{D}^{h_n - m - \beta + 1 + \beta - 1} D^{h_n - \beta - m + 1} = \overline{D}^{\beta - 1} \Delta^{h_n - \beta - m + 1}.
		$$
		Therefore, the function $ g $ is axially Anti-Analytic-Harmonic functions of type $(\beta - 1,\, h_n - m - \beta + 1) $; see Definition~\ref{antiax}.
		
	\end{itemize}
\end{proof}

\begin{theorem}
	\label{reg}
	Let $n$ be an odd number and set $h_n= \frac{n-1}{2}$. Assume that  $\beta$, $m \in \mathbb{N}_0$ are such that $ \beta+m\leq h_n$, then the action of the operator $\overline{D}^\beta \Delta^m$, on a left (resp. right) slice hyperholomorphic function $f$ defined on an axially symmetric open set $U \subseteq \mathbb{R}^{n+1}$ leads to:
	\begin{itemize}
		\item[1)] a left (resp. right) axially Poly-Analytic function $\mathcal{APA}_{h_n-m+1}(U)$ of order $h_n-m+1$ if $\beta=h_n-m$.
		\item[2)] a left (resp. right) axially Analytic-Harmonic functions $\mathcal{AAH}_{\beta+1, h_n-\beta-m}(U)$ of type  $(\beta+1, h_n-\beta-m)$.
	\end{itemize}
\end{theorem}
\begin{proof}
	Our focus is on functions $f$ that are left slice hyperholomorphic, while the right slice hyperholomorphic case follows by analogous reasoning.
	We set
$$
g(x):=\overline{D}^\beta \Delta^m f(x),
$$ this function is an axial function by Proposition \ref{axial}. Similarly as proved in Theorem \ref{reg0} the function $h$ belongs to the kernel of the following operator
	\begin{equation}\label{oper2}
		\overline{D}^{h_n-\beta-m}D^{h_n-m+1}.
	\end{equation}
Indeed by applying the Fueter-Sce Theorem (see Theorem~\ref{FS1}), we obtain
	$$ \overline{D}^{h_n-\beta-m}D^{h_n-m+1}g(x)=\overline{D}^{h_n}D^{h_n+1}f(x)=D \Delta^{h_n}f(x)=0.$$
	We are now in a position to prove all the points stated in the statement.
	
	\begin{itemize}
		\item [1)] If we consider $\beta=h_n-m$ the operator in \eqref{oper2} reduces to $D^{h_n-m+1}$.
So the function $h$ is Poly-Analytic of order $h_n-m+1$, see Definition \ref{polyharm}.
		\item[2)] We observe that
		$$	\overline{D}^{h_n-\beta-m}D^{h_n-m+1}=\overline{D}^{h_n-\beta-m}D^{h_n-m+1-\beta+\beta}= D^{\beta+1}\Delta^{h_n-m-\beta}.$$
		Thus the function $g$ is a left axially Analytic-Harmonic functions  of type $(\beta+1, h_n-m-\beta)$, see Definition \ref{axpolC}.
	\end{itemize}
\end{proof}

\begin{remark}
	The first cases of Theorem \ref{reg0} and Theorem \ref{reg} were already studied
in great detail in \cite{CDP25}. The goal of this paper is to focus on the other two cases.
\end{remark}

\begin{remark}
	In the quaternionic case (i.e., when $ h_n = 1 $), there exists only one $ D$-fine structure, namely the harmonic fine structure studied in~\cite{CDPS}. This can be recovered from Theorem~\ref{reg0} by setting $m=0$ and so $\beta=1$.
	
	\noindent
	On the other hand, the $\overline{D}$-fine structures provide a generalization of the Poly-Analytic fine structures studied in~\cite{Polyf1,Polyf2}. These can be recovered from Theorem~\ref{reg} by taking $m=0$ and so $\beta=1$.
\end{remark}

The function spaces of the Dirac fine structures obtained in Theorem~\ref{reg0} and Theorem \ref{reg} are summarized in the diagram in (\ref{PPPP}):

\begin{equation}\label{PPPP}
	\begin{tikzcd}[column sep=1.8em, row sep=1.8em, ampersand replacement=\&]
		\mathcal{SH} \arrow{r}{\overline{D}} \arrow{d}{D}
		\& \mathcal{AAH}_{2,\hn-1} \arrow{r}{\overline{D}} \arrow{d}{D}
		\& \mathcal{AAH}_{3, \hn-2} \arrow{r}{\overline{D}} \arrow{d}{D}
		\& \cdots \arrow{r}{\overline{D}} \arrow{d}{D}
		\& \mathcal{APA}_{\hn+1} \arrow{d}{D} \\
		\mathcal{APH}_{\hn} \arrow{r}{\overline{D}} \arrow{d}{D}
		\& \mathcal{AAH}_{1,\hn-1} \arrow{r}{\overline{D}} \arrow{d}{D}
		\& \mathcal{AAH}_{2,\hn-2} \arrow{r}{\overline{D}} \arrow{d}{D}
		\& \cdots \arrow{r}{\overline{D}} \arrow{d}{D}
		\& \mathcal{APA}_{\hn} \arrow{d}{D} \\
		\overline{\mathcal{AAH}}_{1, \hn-1} \arrow{r}{\overline{D}} \arrow{d}{D}
		\& \mathcal{APH}_{\hn-1} \arrow{r}{\overline{D}} \arrow{d}{D}
		\& \mathcal{AAH}_{1,\hn-2 } \arrow{r}{\overline{D}} \arrow{d}{D}
		\& \cdots \arrow{r}{\overline{D}} \arrow{d}{D}
		\& \mathcal{APA}_{\hn-1} \arrow{d}{D} \\
		\overline{\mathcal{AAH}}_{2, \hn-2} \arrow{r}{\overline{D}} \arrow{d}{D}
		\& \overline{\mathcal{AAH}}_{1, \hn-2} \arrow{r}{\overline{D}} \arrow{d}{D}
		\& \mathcal{APH}_{\hn-2} \arrow{r}{\overline{D}} \arrow{d}{D}
		\& \cdots \arrow{r}{\overline{D}} \arrow{d}{D}
		\& \mathcal{APA}_{\hn-2} \arrow{d}{D} \\
		\vdots \arrow{d}{D}
		\& \vdots \arrow{d}{D}
		\& \vdots \arrow{d}{D}
		\& \ddots \arrow{d}{D}
		\& \vdots \arrow{d}{D} \\
		\overline{\mathcal{AAH}}_{\hn-1, 1} \arrow{r}{\overline{D}}
		\& \overline{\mathcal{AAH}}_{\hn-2, 1} \arrow{r}{\overline{D}}
		\& \overline{\mathcal{AAH}}_{\hn-3,1} \arrow{r}{\overline{D}}
		\& \cdots \arrow{r}{\overline{D}}
		\& \mathcal{AM}.
	\end{tikzcd}
\end{equation}
The size of the diagram is $h_n \times h_n$. The diagram represents the relationships between different  spaces of axially Analytic-Harmonic functions $\mathcal{AAH}_{\beta,m}$ of type $(\beta,m)$, the conjugate spaces $\overline{\mathcal{AAH}}_{\beta,m}$
       and where they reduce to the well known classes of axially Poly-Harmonic $\mathcal{APH}_m$ and axially Poly-Analytic $\mathcal{APA}_\beta$ functions, under the action of the operators $D$ and $\overline{D}$. The constant $h_n$ is the Sce exponent.
 %   \label{fig:my_diagram}

\begin{remark}
In the quaternionic setting there are just two fine structures of Dirac type and we have studied the function spaces and the related functional calculi on the $S$-spectrum with great details.
Precisely, we have the harmonic fine structure:
\begin{equation}
	\label{fine1} \mathcal{SH}\overset{D}{\longrightarrow} \mathcal{AH}\overset{\overline{D}}{\longrightarrow}\mathcal{AM},
\end{equation}
where $\mathcal{AH}$ is the set of axially Harmonic functions.
Their integral representation give rise to the harmonic functional calculus on the $S$-spectrum.
This fine structure also allows to obtain a product formula for the $F$-functional calculus, see \cite[Thm. 9.3]{CDPS}.
However, since $\Delta={D} {\overline{D}}= {\overline{D}}{D}$, we can interchange the order of the operators $ D$ and $ \overline{D}$ in \eqref{fine1}. This gives rise to the factorization:
\begin{equation} \mathcal{SH}\overset{\overline{D}}{\longrightarrow} \mathcal{APA}_2\overset{D}{\longrightarrow}\mathcal{AM},
\end{equation}
where $\mathcal{APA}_2$ is a space of axially Poly-Analytic functions of order 2.
\end{remark}

\subsection{Examples of specific paths in the diagram \eqref{PPPP}}

There are several particular cases of fine structures on the $S$-spectrum that are of particular interest because of the type of operators used for the factorization and because of the function spaces.

\medskip
{\em The Dirac fine structure with the Laplacian $\Delta$ and the Poly-Harmonic operator $\Delta^m$, for $m\geq 2$.}

\medskip
These are fine structures related to the axially Analytic-Harmonic functions  $\mathcal{AAH}_{1,j}$ of type $(1,j)$ for suitable $j\in \mathbb{N}$.
They contain as particular cases the Laplacian fine structure and the bi-Laplacian one as illustrated in the following description, where $\mathcal{AM}$ denotes the space of axially monogenic functions.
\begin{enumerate}
	\item
	Factorization in terms of the Laplace operator and the spaces of
	axially Analytic-Harmonic functions $\mathcal{AAH}_{1,j}$ of type $(1,j)$, for $j=\hn-1,\hn-2,\ldots,1$:
	
	\begin{equation}\label{KKK}
		\begin{tikzcd}[column sep=2em, ampersand replacement=\&]
			\mathcal{SH} \arrow{r}{\Delta}
			\& \mathcal{AAH}_{1,\hn-1} \arrow{r}{\Delta}
			\& \mathcal{AAH}_{1,\hn-2} \arrow{r}{\Delta}
			\& \mathcal{AAH}_{1,\hn-3} \arrow{r}{\Delta}
			\& \dots \arrow{r}{\Delta}
			\& \mathcal{AAH}_{1,1} \arrow{r}{\Delta}
			\& \mathcal{AM}.
		\end{tikzcd}
	\end{equation}

	\item
	Factorization in terms of the bi-Laplace operator and the spaces of
	axially Analytic-Harmonic functions $\mathcal{AAH}_{1,j}$ of type $(1,j)$, for $j=\hn-2,\hn-4,\ldots,1$:
	\begin{equation}\label{LLLL}
		\begin{tikzcd}[column sep=2em, ampersand replacement=\&]
			\mathcal{SH} \arrow{r}{\Delta^2}
			\& \mathcal{AAH}_{1,\hn-2} \arrow{r}{\Delta^2}
			\& \mathcal{AAH}_{1,\hn-4} \arrow{r}{\Delta^2}
			\& \mathcal{AAH}_{1,\hn-6} \arrow{r}{\Delta^2}
			\& \dots \arrow{r}{\Delta^2}
			\& \mathcal{AAH}_{1,2} \arrow{r}{\Delta^2}.
			\& \mathcal{AM}
		\end{tikzcd}
	\end{equation}

	\item
	Factorization in terms of the powers $m$ of the Laplacian $\Delta^m$ and spaces of  axially Analytic-Harmonic functions  $\mathcal{AAH}_{1,j}$ of type $(1,j)$, for $j=\hn-1m,\hn-2m,\ldots,1$:
	%\label{figLAPLACEemme}
	
	\begin{equation}\label{GGGGG}
		\begin{tikzcd}[column sep=2em, ampersand replacement=\&]
			\mathcal{SH} \arrow{r}{\Delta^m}
			\& \mathcal{AAH}_{1,\hn-m} \arrow{r}{\Delta^m}
			\& \mathcal{AAH}_{1,\hn-2m} \arrow{r}{ \Delta^m}
			\& \mathcal{AAH}_{1,\hn-3m}\arrow{r}{ \Delta^m}
			\& \dots  \arrow{r}{{ \Delta^m}}
			\&\mathcal{AAH}_{1,m} \arrow{r}{{\Delta^m}}
			\&\mathcal{AM}.
		\end{tikzcd}
	\end{equation}

\end{enumerate}
The number of times $p\in\mathbb{N}$ that the operator $\Delta^m$ has to be applied in the above sequence is determined by the relation $ h_n=pm$.  Thus, for a fixed power $m$ of the Laplacian, the length of the corresponding fine structure depends on the dimension $n$ of the Clifford algebra through the Sce exponent $h_n=(n-1)/2$.

\medskip
{\em
	Fine structures associated with the operator $D$ and $\Delta$ and Poly-Harmonic function spaces.}

\medskip
We now illustrate the possibility of having fine structures that do not end with the spaces of axially monogenic functions.
The Poly-Harmonic fine structure considers the first operator $D$ and all the other operators to be $\Delta_{n+1}$. The spaces of Poly-Harmonic functions $\mathcal{APH}_{m}$ for $m=0,1,2,\ldots, \hn$, give rise to the sequence of function space:

\begin{equation}\label{MMM}
	\begin{tikzcd}[column sep=2em, ampersand replacement=\&]
		\mathcal{SH} \arrow{r}{D}
		\& \mathcal{APH}_{\hn} \arrow{r}{\Delta}
		\& \mathcal{APH}_{\hn-1} \arrow{r}{ \Delta}
		\& \mathcal{APH}_{\hn-2} \arrow{r}{{ \Delta}}
		\& \dots \arrow{r}{{ \Delta}}
		\& \mathcal{AH}.
	\end{tikzcd}
\end{equation}
The last function space $\mathcal{AH}$ consists of axially Harmonic functions, which are functions in $\mathcal{APH}_{m}$ for  $m=1$, that are functions of axial type in the kernel of the Laplace operator.

\medskip
{\em
	Fine structures associated with the operator $\overline{D}$ and $\Delta$ and to the function spaces $\mathcal{AAH}_{2,j}$  for $j=\hn-1, \hn-2,\cdots, 0$.}

\medskip
This fine structure with the related function spaces is illustrated in the diagram:
\begin{equation}\label{MMM13}
	\begin{tikzcd}[column sep=2em, ampersand replacement=\&]
		\mathcal{SH} \arrow{r}{\overline{D}}
		\& \mathcal{AAH}_{2,\hn-1} \arrow{r}{\Delta}
		\& \mathcal{AAH}_{2,\hn-2} \arrow{r}{ \Delta}
		\& \mathcal{AAH}_{2,\hn-3} \arrow{r}{ \Delta}
		\& \mathcal{AAH}_{2,\hn-4} \arrow{r}{ \Delta}
		\& \dots \arrow{r}{ \Delta}
		\& \mathcal{APA}_{2}.
	\end{tikzcd}
\end{equation}

The last function space $\mathcal{APA}_{2}$ consists of axially Poly-Analytic functions of order 2, that are functions of axial type in the kernel of the $D^2$ operator. Also this fine structure does not end with axially monogenic functions.
Similarly, for $\Delta^2$, we have:
\begin{equation}\label{MMM12}
	\begin{tikzcd}[column sep=2em, ampersand replacement=\&]
		\mathcal{SH} \arrow{r}{\overline{D}}
		\& \mathcal{AAH}_{2,\hn-1} \arrow{r}{\Delta^2}
		\& \mathcal{AAH}_{2,\hn-3} \arrow{r}{ \Delta^2}
		\& \mathcal{AAH}_{2,\hn-5} \arrow{r}{ \Delta^2}
		\& \mathcal{AAH}_{2,\hn-7} \arrow{r}{ \Delta^2}
		\& \dots \arrow{r}{ \Delta^2}
		\& \mathcal{APA}_{2}.
	\end{tikzcd}
\end{equation}
Similarly to the discussion above the number of time we apply the operators depend on the Sce exponent $h_n$.

\medskip
{\em The fine structure associated with the action of the operator $D \Delta$ applied $p\in \mathbb{N}$ times.}

\medskip
This fine structure with the related function spaces is illustrated in the diagram:

\begin{equation}\label{MMM1}
	\begin{tikzcd}[column sep=2em, ampersand replacement=\&]
		\mathcal{SH} \arrow{r}{D\Delta}
		\& \mathcal{AAH}_{0,\hn-1} \arrow{r}{D\Delta}
		\& \overline{\mathcal{AAH}}_{1,\hn-3} \arrow{r}{ D\Delta}
		\& \overline{\mathcal{AAH}}_{2,\hn-5} \arrow{r}{D\Delta}
		\& \dots \arrow{r}{D\Delta}
		\& \overline{\mathcal{AAH}}_{p-1,\hn-2p+1}.
	\end{tikzcd}
\end{equation}
The sequence ends when $\hn-2p+1=0$ so for $p=(h_n+1)/2$. This means that for $p-1= (h_n+1)/2$ the space $\mathcal{AAH}_{p-1,\hn-2p+1}$ turns out to be $\mathcal{AAH}_{(h_n+1)/2,0}$ that is the space of Poly-Analytic functions $\mathcal{APA}_{(h_n+1)/2}$  of order $(h_n+1)/2$. This mens that for this fine structure we need to require $(h_n+1)/2\in \mathbb{N}$ so the dimension $n$ is subjected to the constraint $(n+1)/4\in \mathbb{N}$.

\section{Integral representation of the functions of the Dirac fine structures}\label{Integral-representation}	

In this section, our goal is to provide an integral representation for the functions that arise in the Dirac fine structures. To derive integral expressions for the Dirac fine structures, we apply the differential operators $ D^\beta \Delta^m $ and $ \overline{D}^\beta \Delta^m $, for $\beta$, $m\in \mathbb{N}_0$ with $\beta +m \leq h_n $, to a left (resp. right) slice hyperholomorphic function $f$,
represented via the Cauchy integral formula (see Theorem~\ref{Cauchy}).
Therefore, it is essential to compute explicitly the action of the differential operators
$ D^\beta \Delta^m $ and $ \overline{D}^\beta \Delta^m $
when applied to the left (resp., right) slice hyperholomorphic Cauchy kernel.

We recall for the convenience of the reader the definition of Pochhammer symbol also for negative integers:

\begin{remark}\label{Pochhsym}
		In the following we will use the Pochhammer symbol defined as
		$$
		(x)_\ell=\frac{\Gamma(x+\ell)}{\Gamma(x)}
		$$
		when we choose $x=-h_n$, where $h_n$ is the Sce exponent and
		where $\ell$ is an integer such that $0 \leq \ell \leq h_n$.
		In order to avoid the poles of the Gamma function in the points $0, -1,-2,\ldots$
		the symbol $(-h_n)_\ell$ has to be interpreted as
		$$
		(-h_n)_{\ell}=(-1)^\ell(h_n-\ell+1)_{\ell}=(-1)^\ell  \frac{\Gamma(h_n+1)}{\Gamma(h_n-\ell+1)}
		$$
		which excludes the poles of the function $\Gamma$.
	\end{remark}

First, we recall how the operator $ \Delta^m $ naturally
acts on the second form of the left (resp. right) slice hyperholomorphic Cauchy kernel (cf.~\cite{CDP25,CSS}),
giving a very elegant expression.

\begin{theorem}
	\label{app1}
	Let $n$ be an odd number and set $h_n:=\frac{n-1}{2}$.
 Assume that $s$, $x \in \mathbb{R}^{n+1}$ are such that $s \notin [x]$. Then,  for any integer $m$ such that $0 \leq m \leq h_n$ we have
	$$ \Delta^{m} S^{-1}_L(s,x)=\gamma_{n,m} (s-\bar{x}) \mathcal{Q}_{c,s}^{-m-1}(x), \quad \left(\Delta^{m} \left(S^{-1}_R(s,x)\right)=\gamma_{n,m}  \mathcal{Q}_{c,s}^{-m-1}(x)(s-\bar{x}) \right),$$
	where
\begin{equation}\label{gammanm}
\gamma_{n, m}:=4^m m! (-h_n)_{m},
\end{equation}
and $(-h_n)_{m}$ is the Pochhammer symbol.
\end{theorem}

In order to give a more compact expression of the next crucial results we introduce the following notations. For $n \in \mathbb{N}$ being odd and $h_n=\frac{n-1}{2}$ we set

\begin{equation}
	\label{k1}
	K_{\nu,\ell}^{1,L}(s,x):=(s-\bar{x}) (s-x_0)^{\nu} \mathcal{Q}_{c,s}^{-\ell}(x), \qquad \ell \leq h_n, \quad \nu \in \mathbb{N},
\end{equation}
and
\begin{equation}
	\label{k2}
	K_{\nu,\ell}^{2}(s,x):= (s-x_0)^{\nu} \mathcal{Q}_{c,s}^{-\ell}(x), \qquad \ell \leq h_n, \quad \nu \in \mathbb{N}.
\end{equation}

We now apply the operator $ D^\beta \Delta^m$ to the left slice hyperholomorphic Cauchy kernel written in second from. The proof is lengthy and technically involved; for this reason, we present it in another paper, see \cite{DPMilano}. Moreover, to simplify the notation in the remainder of the paper, we set

\begin{equation}
	\label{kk}
	S^{-1}_{L, D^\beta \Delta^m}(s,x):=D^\beta \Delta^m S^{-1}_L(s,x),  \quad \beta+m \leq h_n.
\end{equation}

\begin{theorem}\label{p111}
	Let $n$ be a fixed odd number and set $h_n:=(n-1)/2$.
Let $\beta$, $m\in \mathbb{N}_0$ be such that $m + \beta\leq h_n$.
We assume that $s$, $x \in \mathbb{R}^{n+1}$ are such that $s \notin[x]$ and set
\begin{equation}
\label{cbmhn}
c_{\beta, m, h_n}:=\frac{2^{\beta} (h_n-m) \gamma_{n,m}}{m!}.
\end{equation}
 If $\beta=2k_1+1$, where $k_1\in\mathbb N_0$, we have
	\begin{equation}
		\label{do1}
			S^{-1}_{L, D^\beta \Delta^m}(s,x)= c_{\beta, m, h_n}\left(  \sum_{j=0}^{k_1-1} a^1_{j,k_1,m} K^{1,L}_{2j+1, m+j+2+k_1}(s,x)  -\sum_{j=0}^{k_1} b^1_{j,k_1,m}	K_{2j,m+1+k_1+j}^{2}(s,x)  \right),
	\end{equation}
	where the coefficients are defined as follows
	\begin{equation}
		\label{coeff}
		a^1_{j,k_1,m} = \frac{j+1}{k_1+j+1}  b^{1}_{j+1,k_1,m}, \quad  	b^1_{j,k_1,m}:= 2^{2j} (m+k_1+j)! \binom{h_n-m-k_1-j-1}{h_n-m-2k_1-1}  \frac{(k_1+j)!}{(2j)!}.
	\end{equation}
	If $\beta=2k_2$, where $k_2\in\mathbb N$, we have
	\begin{equation}\label{de1}
			S^{-1}_{L, D^\beta \Delta^m}(s,x)= c_{\beta, m, h_n} \left(  \sum_{j=0}^{k_2-1} a^2_{j,k_2,m} K^{1,L}_{2j, m+j+1+k_2}(s,x) -\sum_{j=0}^{k_2-1} b^2_{j,k_2,m} K_{2j+1, m+1+j+k_2}(s,x)  \right),
	\end{equation}
	where the coefficients are given by
	\begin{equation}
		\label{coef}
		b^2_{j,k_2,m}= 2 \left(\frac{k_2+j}{2j+1}\right)	a^2_{j,k_2,m}, \quad a^2_{j,k_2,m}:=2^{2j} (m+k_2+j)!\frac{(k_2+j-1)!}{(2j)!} \binom{h_n-m-k_2-1-j}{h_n-m-2k_2}.
	\end{equation}
	
\end{theorem}

\begin{remark}
	If, we consider $\beta=1$ in \eqref{kk} we get that
	$$S^{-1}_{L, D \Delta^m}(s,x)=-2(h_n-m) \gamma_{n,m} \mathcal{Q}_{c,s}^{-m-1}(x),$$
	where $\gamma_{n, m}$ is defined in (\ref{gammanm}). This is exactly the kernel of the integral representation of poly harmonic functions studied in \cite{CDP25}.
\end{remark}

Now, we recall the action of the operator $\overline{D}^\beta \Delta^{m}$ to the left slice hyperholomorphic Cauchy kernel written in second form, see \cite{DPMilano}. In this case we use the following notation
\begin{equation}
	\label{kk1}
	S^{-1}_{L, \overline{D}^\beta \Delta^m}(s,x):=\overline D^{\beta}\left( \Delta^m S^{-1}_L(s,x) \right),  \quad m+\beta \leq h_n.
\end{equation}

\begin{theorem}\label{barp}
	Let $n$ be a fixed odd number and set $h_n:=(n-1)/2$. Let $ \beta$, $m \in \mathbb{N}_0$ be such that $\beta+m\leq h_n$. Assume that $s$, $x \in \mathbb{R}^{n+1}$ are such that $s \notin[x]$ and
set
\begin{equation}
\mathbf{c}_{\beta, m, h_n}:=2^\beta 4^m (-h_n)_m.
\end{equation}
If $\beta=2k_1+1$, where $k_1\in\mathbb N_0$, we have
	\begin{equation}
		\label{bar1}
		S^{-1}_{L, \overline{D}^\beta \Delta^m}(s,x)= \mathbf{c}_{\beta, m, h_n} \left(  \sum_{j=0}^{k_1} \mathbf{a}^1_{j,k_1,m} K^{1,L}_{2j+1, m+j+2+k_1}(s,x)  +\sum_{j=0}^{k_1}  \mathbf{b}^1_{j,k_1,m} K^2_{2j, m+j+1+k_1}(s,x) \right),
	\end{equation}
	where
	\begin{equation}
		\label{c1}
		\mathbf{b}_{j+1, k_1,m}^1= \frac{\mathbf{a}_{j, k_1,m}^1}{(j+1)(k_1-j)} , \quad 	\mathbf{a}^1_{j,k_1,m}:=2^{2j+1}    \frac{(m +k_1+1+j)!(k_1+j+1)!}{(2j+1)!} \binom{h_n-m-k_1-j-2}{h_n-m-2k_1-2}
	\end{equation}
	If $\beta=2k_2$, where $k_2\in\mathbb N_0$ we have
	\begin{equation}
		\label{bar2}
		 S^{-1}_{L, \overline{D}^\beta \Delta^m}(s,x)= \mathbf{c}_{\beta, m, h_n} \left(  \sum_{j=0}^{k_2} \mathbf{a}^2_{j,k_2,m} K^{1,L}_{2j, m+j+1+k_2}(s,x) +\sum_{j=0}^{k_2-1} \mathbf{b}^2_{j,k_2,m} K^2_{2j+1,m+1+k_2+j}(s,x) \right),
	\end{equation}
	where
	\begin{equation}
		\label{c2}
		\mathbf{b}^2_{j,k_2,m}= \frac{2(k_2-j)}{(2j+1)} \mathbf{a}^2_{j,k_2,m}, \quad 	\mathbf{a}^2_{j,k_2,m}:=2^{2j} (m+k_2+j) ! (k_2-j)! \binom{k_2+j}{2j} \binom{h_n-m-k_2-1-j}{h_n-m-2k_2-1}.
	\end{equation}
\end{theorem}

\begin{remark}
If $ \beta = h_n - m $, we recover the Poly-Analytic kernel considered in~\cite{CDP25}, namely:

	$$S^{-1}_{L, \overline{D}^{h_n - m} \Delta^m}(s, x) = 4^{h_n} h_n! (-h_n)_m (s - \bar{x}) (s - x_0)^{h_n - m} \mathcal{Q}_{c,s}^{-1 - h_n}(x),$$
	see also \cite{DPMilano}.
\end{remark}

\begin{remark}
Theorem~\ref{p111} and Theorem~\ref{barp} can be proved
also in the case where the operators
$D^\beta \Delta^m$ and $\overline{D}^\beta \Delta^m$ for $\beta, m \in \mathbb{N}_0$ are applied to the right slice hyperholomorphic Cauchy kernel.
 We do not state them since they can be obtained by similar arguments.
\end{remark}

\begin{remark}
\label{new02}
In the previous results, the hypothesis $m+\beta\leq h_n$ plays a fundamental role. In particular, this condition prevents spaces such as $\overline{\mathcal{AAH}}_{\beta,0}$ from appearing in the diagram \eqref{PPPP}. To clarify this point, we provide an example in the quaternionic case $n=3$. In this case, we have the factorization
$$
\mathcal{SH}\overset{D}{\longrightarrow}
\mathcal{AH}\overset{D}{\longrightarrow}
\overline{\mathcal{AAH}}_{1,0}.
$$
In order to obtain an integral representation for functions in $\overline{\mathcal{AAH}}_{1,0}$, we would need to determine the action of the operator $D^2$ on the second form of the slice hyperholomorphic Cauchy kernel. However, Theorem 4.3 cannot be applied in this case, since $\beta=2$ and $m=0$, whereas $h_n=1$, so that the condition $m+\beta\leq h_n$ is not satisfied. Indeed, in \cite[Lemma 4.13]{CDPS}, we proved that
$$
D^2 S^{-1}_L(s,x)=-4(s-x)\mathcal{Q}_{c,s}^{-2}(x),
$$
which is not covered by the formulas obtained in Theorem 4.3. We are studying this approach in another project.
\end{remark}

\begin{remark}
In future work, we plan to study factorizations of the operator $\Delta^{\frac{n-1}{2}}$ also when $n$ is even. This would lead to two possible settings for factorizations of the form $D^\beta\Delta^m$: the case where both $\beta$ and $m$ are integers, and the case where $m$ is non-integer. We have started to investigate the latter case in \cite{CDD}, and the problem turns out to be very challenging. However, when both $\beta$ and $m$ are integers, we can use the computations performed in this paper for the action of these operators on the slice hyperholomorphic Cauchy kernel; see Theorems~\ref{app1}, \ref{p111}, and~\ref{barp}. Nevertheless, this would require a significant modification of Theorems~\ref{reg0} and~\ref{reg}, since in this case we would need to prove that the functions obtained by applying these operators belong to the fractional counterparts of the spaces introduced in Definitions~\ref{axpolC} and~\ref{antiax}.
	\end{remark}

In order to study the kernel of the operators $D^\beta \Delta^m$ and $\overline{D}^\beta \Delta^m$ we need the following result.

\begin{lemma}
	\label{aux}
	Let $n$ be an odd number and set $h_n:= \frac{n-1}{2}$ with $\beta$, $m \in \mathbb{N}$ such that $\beta+m \leq h_n$. We suppose that $U$ is a slice domain and $f \in \mathcal{SH}_L(U)$ (resp. $f \in \mathcal{SH}_R(U)$). Then, for every $x_0\in U\cap\mathbb{R}$, we have
	\begin{equation}
		\label{lim1}
		\lim_{|\underline{x}| \to 0} D^\beta \Delta^m f(x)= c(\beta,m)(-1)^m \frac{\partial^{2m+\beta} f(x_0)}{\partial x_0^{2m+\beta}},
	\end{equation}
	where the constant $c(\beta,m)$ is negative. Moreover
	\begin{equation}
		\label{lim2}
		\lim_{|\underline{x}| \to 0} \overline{D}^\beta \Delta^m f(x)= \tilde{c}(\beta,m) (-1)^m \frac{\partial^{2m+\beta} f(x_0)}{\partial x_0^{2m+\beta}},
	\end{equation}
	where the constant $\tilde{c}(\beta,m)$ is positive.
\end{lemma}
\begin{proof}
	We prove the result only for left slice hyperholomorphic functions, since for right slice hyperholomorphic function the arguments are similar. Both statements are local: it is enough to prove them in a neighbourhood of an arbitrary point $x_0\in U\cap\mathbb{R}$. We fix a ball $B_\varepsilon(x_0)$ with $x_0\in U\cap\mathbb{R}$ and $\varepsilon>0$ small enough such that $\overline{B_\varepsilon(x_0)}\subset U$. Since $f\in\mathcal{SH}_L(U)$, the Cauchy formula of Theorem \ref{Cauchy} applies on $B_\varepsilon(x_0)$. We start proving \eqref{lim1}. Since by hypothesis the function $f$ is left slice hyperholomorphic by the Cauchy formula, see Theorem \ref{Cauchy}, we have
	$$ D^\beta \Delta^m f(x)= \frac{1}{2 \pi} \int_{\partial(B_\varepsilon(x_0) \cap \mathbb{C}_I)} [D^\beta \Delta^m S^{-1}_L(s,x)] ds_I f(s), \qquad x \in B_\varepsilon(x_0).$$
	We start suppose that $\beta$ is odd, i.e. $\beta=2k_1+1$, with $k_1 \in \mathbb{N}_0$. By Theorem \ref{p111} and making the restriction to $\underline{x}=0$ we get
	\begingroup\allowdisplaybreaks
	\begin{eqnarray*}
		\lim_{|\underline{x}| \to 0}
		D^\beta \Delta^m f(x)&=& \frac{c_{\beta, m, h_n}}{2 \pi} \left[\sum_{j=0}^{k_1-1}  \int_{\partial(B_\varepsilon(x_0) \cap \mathbb{C}_I)} a_{j, k_1,m}^1 K^{1,L}_{2j+1,m+j+2+k_1}(s,x_0)ds_I f(s) \right.\\
		&&- \left. \sum_{j=0}^{k_1}  \int_{\partial(B_\varepsilon(x_0) \cap \mathbb{C}_I)} b_{j, k_1,m}^1 K^{2}_{2j,m+j+1+k_1}(s,x_0)ds_I f(s) \right].
	\end{eqnarray*}
	\endgroup
	Now, since $ \mathcal{Q}_{c,s}(x_0)=(s-x_0)^{2}$, by \eqref{k1} and \eqref{k2} we have
	\begingroup\allowdisplaybreaks
	\begin{equation}
		\label{res}
		K^{1,L}_{\nu,\ell}(s,x_0)=(s-x_0)^{\nu+1-2 \ell}, \qquad K^2_{\nu,\ell}(s, x_0)=(s-x_0)^{\nu-2\ell}.
	\end{equation}
	\endgroup
	Thus we can write
	$$
	K^{1,L}_{2j+1,m+j+2+k_1}(s,x_0)=K^{2}_{2j,m+j+1+k_1}(s,x_0)=(s-x_0)^{-2m-2k_1-2}=(s-x_0)^{-2m-\beta-1}.
	$$
	By the derivative Cauchy formula, see \cite[Thm. 2.8.5]{ColomboSabadiniStruppa2011}, and the definition of the constant $c_{\beta,m , h_n}$ we get
	\begingroup\allowdisplaybreaks
	\begin{eqnarray*}
		\nonumber
		\lim_{|\underline{x}| \to 0}
		D^\beta \Delta^m f(x)&=&  \left(\sum_{j=0}^{k_1-1} a_{j, k_1,m}^1  - \sum_{j=0}^{k_1}  b_{j, k_1,m}^1  \right) \frac{c_{\beta, m, h_n}}{2 \pi}\int_{\partial(B_\varepsilon(x_0) \cap \mathbb{C}_I)}  (s-x_0)^{-2m-\beta-1}ds_I f(s)\\
		&=&\left(\sum_{j=0}^{k_1-1} a_{j, k_1,m}^1  - \sum_{j=0}^{k_1}  b_{j, k_1,m}^1  \right) \frac{2^{\beta+2m} h_n! (-1)^m}{(h_n-m-1)!(2m+\beta)!}  \frac{\partial^{2m+\beta} }{\partial x_0^{2m+\beta}}f(x_0).
	\end{eqnarray*}
	\endgroup
	Now, we focus on the coefficients. By \eqref{coeff} we get
	\begin{eqnarray*}
c(\beta,m)&=&\left(\sum_{j=0}^{k_1-1} a_{j, k_1,m}^1  - \sum_{j=0}^{k_1}  b_{j, k_1,m}^1  \right) \frac{2^{\beta+2m} h_n!}{(h_n-m-1)!(2m+\beta)!} \\
&=& \left(\sum_{j=0}^{k_1-1} a_{j, k_1,m}^1  - \sum_{j=0}^{k_1-1}  b_{j+1, k_1,m}^1-b_{0,k_1,m}^1  \right) \frac{2^{\beta+2m} h_n!}{(h_n-m-1)!(2m+\beta)!} \\
&=&-\left(\sum_{j=0}^{k_1-1} \frac{k_1}{k_1+j+1}b_{j+1, k_1,m}^1+b_{0, k_1,m}^1\right)\frac{2^{\beta+2m} h_n!}{(h_n-m-1)!(2m+\beta)!} .
	\end{eqnarray*}
	We observe that $c(\beta,m)$ is negative. Now, we suppose that $\beta$ is even i.e. $\beta=2k_2$, with $k_2 \in \mathbb{N}$. By Theorem \ref{p111}, formula in
	\eqref{res} and the derivative of the Cauchy formula we get
	
	\begingroup\allowdisplaybreaks
	\begin{eqnarray*}
		\nonumber
		\lim_{|\underline{x}| \to 0}
		D^\beta \Delta^m f(x)&=& \frac{c_{\beta, m, h_n}}{2 \pi} \left[\sum_{j=0}^{k_2-1}  \int_{\partial(B_\varepsilon(x_0) \cap \mathbb{C}_I)} a_{j, k_2,m}^2 K^{1,L}_{2j,m+j+1+k_2}(s,x_0)ds_I f(s) \right.\\
		\nonumber
		&&- \left. \sum_{j=0}^{k_2-1}  \int_{\partial(B_\varepsilon(x_0) \cap \mathbb{C}_I)} b_{j, k_2,m}^2 K^{2}_{2j+1,m+j+1+k_2}(s,x_0)ds_I f(s) \right]\\
		\nonumber
		&=& \frac{c_{\beta, m, h_n}}{2 \pi} \left[\sum_{j=0}^{k_2-1}(a_{j, k_2,m}^2-b_{j, k_2,m}^2) \right]\int_{\partial(B_\varepsilon(x_0) \cap \mathbb{C}_I)} (s-x_0)^{-2m-\beta-1}ds_I f(s)\\
		&=& \frac{2^{\beta+2m} h_n! (-1)^m}{(h_n-m-1)!(2m+\beta)!}\left[\sum_{j=0}^{k_2-1}(a_{j, k_2,m}^2-b_{j, k_2,m}^2) \right] \frac{\partial^{2m+\beta}}{\partial x_{0}^{2m+\beta}}f(x_0).
	\end{eqnarray*}
	\endgroup
	We focus on the coefficients. By \eqref{coef} we obtain
	$$
	c(\beta,m):=\frac{2^{\beta+2m} h_n! }{(h_n-m-1)!(2m+\beta)!}\sum_{j=0}^{k_1-1}(a_{j, k_2,m}^2-b_{j, k_2,m}^2)
$$
$$
=-\frac{2^{\beta+2m} h_n! }{(h_n-m-1)!(2m+\beta)!}\sum_{j=0}^{k_2-1} \frac{2k_2-1}{2j+1}a_{j, k_2,m}^2.
	$$
	Since $k_2 \in \mathbb{N}$ we get that the above constant has to be negative.
	\\Now, we prove \eqref{lim2}. We begin by assuming that $\beta$ is odd, i.e $\beta:=2k_1+1$. By Theorem \ref{barp}, formula \eqref{res} and the derivative of the Cauchy formula we get
	\begin{eqnarray*}
		\lim_{|\underline{x}| \to 0} \overline{D}^\beta \Delta^m f(x)&=& \mathbf{c}_{\beta, m, h_n}\left( \sum_{j=0}^{k_1}\mathbf{a}_{j,k_1,m}^1+\mathbf{b}_{j,k_1,m}^1\right) \int_{\partial (B_\varepsilon(x_0)\cap \mathbb{C}_I)} (s-x_0)^{-2m-\beta-1} ds_I f(s)\\
		&=&  \left( \sum_{j=0}^{k_1}\mathbf{a}_{j,k_1,m}^1+\mathbf{b}_{j,k_1,m}^1\right) \frac{2^{\beta+2m} h_n!(-1)^m}{(2m+\beta)!(h_n-m)!}\frac{\partial^{2m+\beta} }{\partial x_0^{2m+\beta}}f(x_0).
	\end{eqnarray*}
	We denote $\textbf{c}(\beta,m)$ by
	$$ \textbf{c}(\beta,m):=\left( \sum_{j=0}^{k_1}\mathbf{a}_{j,k_1,m}^1+\mathbf{b}_{j,k_1,m}^1\right) \frac{2^{\beta+2m}}{(2m+\beta)! h_n!(h_n-m)!}.$$
	It is clear that this constant is positive.
	\\ Finally we suppose that $\beta$ is even, i.e. $\beta=2k_1$, with $k_1 \in \mathbb{N}$. By using similar arguments used for the case $\beta$ odd we get
	$$\lim_{|\underline{x}| \to 0} \overline{D}^\beta \Delta^m f(x)=\tilde{c}(\beta,m)(-1)^m\frac{\partial^{2m+\beta} }{\partial x_0^{2m+\beta}}f(x_0),$$
	where
	$$ \tilde{c}(\beta,m):= \left(\sum_{j=0}^{k_2}\mathbf{a}^2_{j,k_2,m}+\sum_{j=0}^{k_2-1}\mathbf{b}^2_{j,k_2,m}\right) \frac{2^{\beta+2m}}{(2m+\beta)! h_n!(h_n-m)!},$$
	which is obviously a positive constant. This proves the result.
	
\end{proof}

\begin{lemma}
	\label{kernel}
	Let $n$ be an odd number and set $h_n:= \frac{n-1}{2}$ with $m$, $\beta \in \mathbb{N}$ such that $\beta+m \leq h_n$. We assume that $U$ is a connected slice domain and $f \in \mathcal{SH}_L(U)$ (resp. $f \in \mathcal{SH}_R(U)$). Then we have
	$f \in \hbox{ker}(D^\beta \Delta^m)$ if and only if $f(x)= \sum_{\nu=0}^{2m+\beta-1} x^\nu \alpha_\nu$ (resp. $f(x)= \sum_{\nu=0}^{2m+\beta-1} \alpha_\nu x^\nu $), with $ \{\alpha_{\nu}\}_{1 \leq \nu \leq 2m+\beta-1} \subseteq \mathbb{R}_n$.
	Moreover, we have
	$$\hbox{ker}(D^\beta \Delta^m) \cap \mathcal {SH}_L(U) =\hbox{ker}(\overline{D}^\beta \Delta^m) \cap\mathcal{SH}_L(U)$$
$$
({\rm resp.}\ \  \hbox{ker}(D^\beta \Delta^m) \cap \mathcal {SH}_R(U) =\hbox{ker}(\overline{D}^\beta \Delta^m) \cap\mathcal{SH}_R(U)).
$$
\end{lemma}
\begin{proof}
	We prove only the case $f \in \mathcal{SH}_L(U)$, since the case $f \in \mathcal{SH}_R(U)$ can be proved by similar arguments. We proceed by double inclusion. We suppose that $ f(x)= \sum_{\nu=0}^{2m+\beta-1}  x^\nu \alpha_\nu$, then it is clear that $f \in \hbox{ker}(D^\beta \Delta^m)$. Now, we suppose that $f \in \hbox{ker}(D^\beta \Delta^m)$. Since $f$ is slice hyperholomorphic on $B_r(0)$ for $r$  sufficiently small, possibly after a change of variables along the real line, the function $ f$ can be expanded into a convergent series at the origin as follows:
	$$ f(x)= \sum_{k=0}^{\infty} x^k \alpha_k, \qquad \{\alpha_k\}_{k \in \mathbb{N}_0} \subset \mathbb{R}_n,$$
	for any $x \in B_r(0)$, see \cite[Prop. 2.3.6]{ColomboSabadiniStruppa2011}. We apply the differential operator $D^\beta \Delta^m$ and by using the hypothesis we obtain
	\begin{equation}
		\label{res1}
		0= D^\beta \Delta^m f(x)= \sum_{k=2m+\beta}^{\infty} D^\beta \Delta^m x^k \alpha_k, \qquad \forall x \in B_r(0).
	\end{equation}
	We restrict $D^\beta \Delta^m x^k$ to a neighbourhood of the real axis $\Omega$. By \eqref{lim1} we have
	$$ \lim_{|\underline{x}| \to 0} D^\beta \Delta^m x^k= c(\beta,m) \frac{(-1)^m k!}{(k-2m-\beta)!} x_0^{k-2m-\beta}, \qquad k \geq 2m+\beta.$$
	Thus by restricting \eqref{res1} to neighbourhood of the real axis we have
	$$ 0= c(\beta,m)(-1)^m\sum_{k=2m+\beta}^{\infty} \frac{k!}{(k-2m-\beta)!} x_0^{k-2m-\beta}\alpha_k, \qquad \forall x_0 \in B_r(0) \cap \mathbb{R}.$$
	Since by Proposition \ref{aux} the constant $c(\beta,m)$ is always positive, and by \cite[Cor. 1.2.6]{KP} we get:
	$$\alpha_k=0, \qquad k \geq 2m+\beta.$$
	Therefore we have $f(x)= \sum_{k=0}^{2m+\beta-1} x^k \alpha_k$ in $\Omega$, and since the set $U$ is connected we get that $f(x)= \sum_{k=0}^{2m+\beta-1} x^k \alpha_k$ with $x \in U$.
	\\ By applying a similar argument, it can be shown that $ f \in \ker(\overline{D}^\beta \Delta^m) $ if and only if
	$
	f(x) = \sum_{k=0}^{2m+\beta-1} x^k \alpha_k.
	$
	Hence, the equality between the kernels of $ D^\beta \Delta^m $ and $ \overline{D}^\beta \Delta^m $, intersected with the set of left slice hyperholomorphic functions, follows immediately.
	
\end{proof}

We now have all the necessary tools to provide the integral representation of the functions that arise from the Dirac fine structures.

\begin{theorem}[Integral representation of $D$-fine structures]\label{Dinte}
	Let $n$ be an odd number, set $h_n:=\frac{n-1}{2}$ and $s$, $x \in \mathbb{R}^{n+1}$ be such that $s \notin [x]$. Let $W\subseteq\mathbb{R}^{n+1}$ be a connected slice domain and let $f\in\mathcal{SH}_L(W)$
	(resp. $f\in\mathcal{SH}_R(W)$). Let $U$ be a connected bounded slice Cauchy domain with $\overline{U}\subset W$.
We set $ds_I=ds(-I)$ with $I \in \mathbb{S}$.
Then, for
$
f_{\beta,m}(x) := D^\beta \Delta^m f(x) 
\quad 
\big(\text{resp. } f_{\beta,m}(x) := f(x) D^\beta \Delta^m \big),
$
with $\beta, m \in \mathbb{N}_0$ satisfying $\beta + m \le h_n$, we have that

	\begin{itemize}
		\item[(1)] $f_{\beta,m}(x)$ is left (resp. right) axially poly-harmonic function $\mathcal{APH}_{h_n-m}(U)$ of order $h_n-m$ if $\beta=1$
		\item[(2)] $f_{\beta,m}$ is left (resp. right) axially Anti-Analytic-Harmonic  $\overline{\mathcal{AAH}}_{\beta-1, h_n-m-\beta+1}(U)$ type $(\beta-1, h_n-m-\beta+1)$.
	\end{itemize}
	and in all these cases it admits the integral representation
	\begin{equation}
		\label{inte1}
		f_{\beta,m}(x)=\frac{1}{2 \pi} \int_{\partial (U \cap \mathbb{C}_I)}S^{-1}_{L,D^\beta \Delta^m}(s,x) ds_I f(s), \qquad \hbox{for any} \qquad x \in U,
	\end{equation}
	where the kernel $S^{-1}_{L,D^\beta \Delta^m}(s,x)$ is given in \eqref{do1} and \eqref{de1}, according to the parity of $\beta$. Moreover:
		\begin{itemize}
			\item[(3)] the integral \eqref{inte1} does not depend on the imaginary unit $I\in\mathbb{S}$ nor $U$.
			\item[(4)] {\rm (independence of the representative)} the integral \eqref{inte1} does not depend on the particular $f\in\mathcal{SH}_L(W)$ chosen to represent it: if
			$f_1\in\mathcal{SH}_L(W)$ is such that $D^\beta \Delta^m f_1=D^\beta\Delta^m f$ on $W$, then
			$$
			\frac{1}{2\pi}\int_{\partial(U\cap\mathbb{C}_I)}S^{-1}_{L,D^\beta\Delta^m}(s,x)\,ds_I\,f_1(s)
			=\frac{1}{2\pi}\int_{\partial(U\cap\mathbb{C}_I)}S^{-1}_{L,D^\beta\Delta^m}(s,x)\,ds_I\,f(s).
			$$
	\end{itemize}

\end{theorem}
\begin{proof}
	The statements in points (1) and (2) have already been proved in Theorem~\ref{reg}. We consider just the left case since the right one is similar.
We prove the integral expression \eqref{inte1}. By the Cauchy formula, see Theorem \ref{Cauchy}, and  the application of the operator $D^\beta \Delta^m$ to the left slice hyperholomorphic Cauchy kernel, see Theorem \ref{p111}, we obtain
\begin{eqnarray*}
	f_{\beta,m}(x)\! \!  \! \!&=& \! \!  \! \!D^\beta \Delta^m f(x)=\frac{1}{2 \pi} \int_{\partial(U \cap \mathbb{C}_I)} D^\beta \Delta^m S^{-1}_L(s,x)\, ds_I\, f(s)
	\\
	&=&\! \!  \! \!\frac{1}{2 \pi} \int_{\partial (U \cap \mathbb{C}_I)}S^{-1}_{L,D^\beta \Delta^m}(s,x)\, ds_I\, f(s).
\end{eqnarray*}

Point (3), i.e. the independence of the integral \eqref{inte1} from the imaginary unit and from the bounded slice Cauchy domain, follows from Theorem \ref{Cauchy}.
We now prove point (4), which is the main content of this theorem. We assume that $f_1$ is a function such that $f_{\beta,m}(x)= D^\beta \Delta^m f_1(x)$. Therefore we have that $f_1(x)-f(x) \in \hbox{ker}(D^\beta \Delta^m)$
$\cap\,\mathcal{SH}_L(W)$, so that Lemma \ref{kernel} applies. By Lemma \ref{kernel} we have that
$$f_1(x)=f(x)+\sum_{\nu=0}^{2m+\beta-1} x^\nu \alpha_\nu, \qquad \{\alpha_{\nu}\} \subseteq \mathbb{R}_n.$$
Hence by the integral expression \eqref{inte1} we have
\begin{eqnarray*}
	&& \frac{1}{2 \pi} \int_{\partial(U \cap \mathbb{C}_I)}S^{-1}_{L,D^\beta \Delta^m}(s,x) ds_I f_1(s) \\
	&&= \frac{1}{2 \pi} \int_{\partial(U \cap \mathbb{C}_I)}S^{-1}_{L,D^\beta \Delta^m}(s,x)ds_I f(s)+\sum_{\nu=0}^{2m+\beta-1} \int_{\partial(U \cap \mathbb{C}_I)}S^{-1}_{L,D^\beta \Delta^m}(s,x)ds_I s^\nu \alpha_\nu\\
	&&= \frac{1}{2 \pi} \int_{\partial(U \cap \mathbb{C}_I)}S^{-1}_{L,D^\beta \Delta^m}(s,x)ds_I f(s)+\sum_{\nu=0}^{2m+\beta-1} (D^\beta \Delta^m x^\nu) \alpha_{\nu}\\
	&&=\frac{1}{2 \pi} \int_{\partial(U \cap \mathbb{C}_I)}S^{-1}_{L,D^\beta \Delta^m}(s,x)ds_I f(s).
\end{eqnarray*}
This proves the result.
\end{proof}

By using a similar arguments we can prove the following result.

\begin{theorem}[Integral representation of $\overline{D}$-fine structures]\label{Dbarinte}
	Let $n$ be an odd number, set $h_n:=\frac{n-1}{2}$ and $s$, $x \in \mathbb{R}^{n+1}$ be such that $s \notin [x]$.
Let $W\subseteq\mathbb{R}^{n+1}$ be a connected slice domain and let $f\in\mathcal{SH}_L(W)$
	(resp. $f\in\mathcal{SH}_R(W)$). Let $U$ be a connected bounded slice Cauchy domain with $\overline{U}\subset W$.
We set $ds_I=ds(-I)$ with $I \in \mathbb{S}$.
Then, for
$
f_{\overline{\beta},m}(x) := \overline{D}^\beta \Delta^m f(x) 
\quad 
\big(\text{resp. } f_{\overline{\beta},m}(x) := f(x) \overline{D}^\beta \Delta^m \big),
$
with $\beta, m \in \mathbb{N}_0$ satisfying $\beta + m \le h_n$, we have that
	\begin{itemize}
		\item[1)] $f_{\overline{\beta},m}(x)$ is left  (resp. right)  axially Poly-Analytic function $\mathcal{APA}_{h_n-m+1}(U)$ of order $h_n-m+1$ if $\beta=h_n-m$,
		\item[2)] $f_{\overline{\beta},m}(x)$ is left  (resp. right)  axially
Anti-Analytic-Harmonic  $\overline{\mathcal{AAH}}_{\beta+1, h_n-\beta-m}(U)$ of type $(\beta+1, h_n-\beta-m)$
	\end{itemize}
	and in all these cases we have the following integral representation
	
	\begin{equation}
		\label{inte2}
		f_{\overline{\beta},m}(x)=\frac{1}{2 \pi} \int_{\partial (U \cap \mathbb{C}_I)}S^{-1}_{L,\overline{D}^\beta \Delta^m}(s,x) ds_I f(s), \qquad \hbox{for any} \qquad x \in U,
	\end{equation}
	where the kernel \( S^{-1}_{L,\overline{D}^\beta \Delta^m}(s,x) \) is defined by equations \eqref{bar1} and \eqref{bar2}, depending on whether \( \beta \) is even or odd.
Moreover:
	\begin{itemize}
		\item[(3)] the integral \eqref{inte1} does not depend on the imaginary unit $I\in\mathbb{S}$ nor $U$.
		\item[(4)] {\rm (independence of the representative)} the integral \eqref{inte1} does not depend on the particular $f\in\mathcal{SH}_L(W)$ chosen to represent it: if
		$f_1\in\mathcal{SH}_L(W)$ is such that $D^\beta \Delta^m f_1=D^\beta\Delta^m f$ on $W$, then
		$$
		\frac{1}{2\pi}\int_{\partial(U\cap\mathbb{C}_I)}S^{-1}_{L,\overline{D}^\beta\Delta^m}(s,x)\,ds_I\,f_1(s)
		=\frac{1}{2\pi}\int_{\partial(U\cap\mathbb{C}_I)}S^{-1}_{L,\overline{D}^\beta\Delta^m}(s,x)\,ds_I\,f(s).
		$$
	\end{itemize}

\end{theorem}

\section{Series expansions of the kernels $\mathcal{Q}^{-\ell}_{c,s}(x)$ and Gegenbauer  polynomials}\label{Jacobi-series-expan}

The goal of this section is to provide a series expansion of the kernels of the integral representations $ S^{-1}_{L, D^\beta \Delta^m}(s,x)$ and $S^{-1}_{L, \overline{D}^\beta \Delta^m}(s,x) $; see~\eqref{kk} and~\eqref{kk1}, in terms of the Gegenbauer polynomials, which are defined as follows:

\begin{equation}
	\label{Gen}C_n^{(\alpha)}(z)=\sum_{k=0}^{\left\lfloor \frac{n}{2}\right\rfloor}(-1)^k
	\frac{\Gamma(n-k+\alpha)}{\Gamma(\alpha)k!(n-2k)!}(2z)^{n-2k},\qquad z\in\mathbb R,
\end{equation}
where $\alpha>0$ and $n\in\mathbb N_0$; see \cite{S39}.
In this section, we will use the following fundamental generating function
for the Gegenbauer polynomials:
\begin{equation}
	\label{genera}
	(1-2zs+s^2)^{-\ell}=\sum_{d=0}^{\infty}C_d^{(\ell)}(z)s^d,\qquad |s|<1,
\end{equation}
where $\ell>0$ and $z\in[-1,1]$.

To this end, we proceed by deriving the series expansions of the functions $ K^{1,L}_{\nu,\ell}(s,x) $ and $ K^{2}_{\nu,\ell}(s,x) $, defined in~\eqref{k1} and~\eqref{k2}, which serve as their building blocks. From the structure of $ K^{1,L}_{\nu,\ell}(s,x)$ and $ K^{2}_{\nu,\ell}(s,x) $, it is evident that the crucial step is to obtain the series expansion of $ \mathcal{Q}_{c,s}^{-\ell}(x) $. For this purpose, we introduce the following new family of polynomials.

\begin{definition}
	\label{harmpoly1}
	Let $n$ be an odd number and set $h_n=\frac{n-1}{2}$. For $k \geq 2 \ell-1$ and $1 \leq \ell \leq h_n$ we define the polynomials $\mathcal{H}_{\ell}^k(x)$ as

	\begin{equation}
		\label{harmpoly}
		\mathcal{H}_{\ell}^k(x):=\sum_{j=0}^{k-2\ell+1} \mathcal{T}^k_j(\ell) x^{k-2\ell+1-j} \bar{x}^{j},\quad \mathcal{T}^k_j(\ell):=\binom{\ell+j-1}{\ell-1} \binom{k-\ell-j}{\ell-1}.
	\end{equation}
	
\end{definition}

\begin{remark}
	\label{conv}
	In the sequel we will use the convention that $\mathcal {H}^k_\ell (x)\equiv 0$ if $k-2\ell+1<0$.
\end{remark}

\begin{proposition}
\label{gege1}
	Let $n$ be odd and set $h_n=(n-1)/2$. For $1\leq \ell\leq h_n$ and $k\geq 2\ell-1$, we have
	\begin{equation}
		\label{gege}
		\mathcal{H}_{\ell}^k(x)=|x|^{k-2\ell+1}C_{k-2\ell+1}^{(\ell)}\left(\frac{x_0}{|x|}\right),\qquad x\in\mathbb{R}^{n+1}\setminus\{0\},
	\end{equation}
	where $C_m^\ell$ denotes the Gegenbauer polynomial of degree $m$.
\end{proposition}

\begin{proof}
	We set $d:=k-2\ell+1$. By Definition~\ref{harmpoly1}, we have
	\begin{eqnarray*}
		\mathcal{H}_{\ell}^{d+2\ell-1}(x)&=&\sum_{j=0}^{d}\binom{\ell+j-1}{\ell-1}\binom{d+\ell-1-j}{\ell-1}x^{d-j}\bar{x}^{j}\\
		&=&\sum_{j=0}^{d}\binom{\ell+j-1}{j}\binom{\ell+d-j-1}{d-j}x^{d-j}\bar{x}^{j}.
	\end{eqnarray*}
	Fix now $x\in\mathbb{R}^{n+1}\setminus\{0\}$ and let $t\in\mathbb{R}$ be such that
	$$|t|<\frac{1}{|x|}.$$
	Then $|xt|<1$ and $|\bar{x}t|<1$, and hence the binomial expansions give
	$$	(1-xt)^{-\ell}=\sum_{p=0}^{\infty}\binom{\ell+p-1}{p}x^pt^p, \quad (1-\bar{x}t)^{-\ell}=\sum_{q=0}^{\infty}\binom{\ell+q-1}{q}\bar{x}^qt^q.$$
	Since both series converge absolutely, their Cauchy product is
	$$	\frac{1}{(1-xt)^\ell(1-\bar{x}t)^\ell}=\sum_{d=0}^{\infty}\left(\sum_{q=0}^{d}\binom{\ell+d-q-1}{d-q}\binom{\ell+q-1}{q}x^{d-q}\bar{x}^{q}\right)t^d=	\sum_{d=0}^{\infty}\mathcal{H}_{\ell}^{d+2\ell-1}(x)t^d.
	$$
	Since $(1-xt)(1-\bar{x}t)=1-2x_0t+|x|^2t^2$, we have
	\begin{equation}
		\label{gen}
		\sum_{d=0}^{\infty}\mathcal{H}_{\ell}^{d+2\ell-1}(x)t^d=\left(1-2x_0t+|x|^2t^2\right)^{-\ell}.
	\end{equation}
	On the other hand, the generating function of the Gegenbauer polynomials, see \eqref{genera}, by taking
	$$z=\frac{x_0}{|x|}\qquad\text{and}\qquad s=|x|t,$$
	which are admissible because $|s|<1$ and $|z| \leq 1$, we obtain
	$$\left(1-2x_0t+|x|^2t^2\right)^{-\ell}=\sum_{d=0}^{\infty}|x|^d C_d^{(\ell)}\left(\frac{x_0}{|x|}\right)t^d.$$
	Hence, comparing the above formula and \eqref{gen} in a neighborhood of $t=0$ we get
	$$\sum_{d=0}^{\infty}\mathcal{H}_{\ell}^{d+2\ell-1}(x)t^d=\sum_{d=0}^{\infty}|x|^dC_d^{(\ell)}\left(\frac{x_0}{|x|}\right)t^d.$$
	By uniqueness of power series expansions, we get
	$$\mathcal{H}_{\ell}^{d+2\ell-1}(x)=|x|^dC_d^{(\ell)}\left(\frac{x_0}{|x|}\right).$$
	Since $x\in\mathbb{R}^{n+1}\setminus\{0\}$ was arbitrary, the identity holds for every $x$. Recalling that $d=k-2\ell+1$, the result follows.
\end{proof}

\begin{remark}
	A connection between the Gegenbauer polynomials and the Fueter--Sce mapping theorem was established in \cite{ESL}.
\end{remark}

The first basic property of the polynomials $\mathcal{H}_{\ell}^k(x)$ is given by the following result.

\begin{lemma}
	\label{real}
	Let $n$ be an odd number and set $h_n=\frac{n-1}{2}$, $k \geq 2 \ell-1$ and $1 \leq \ell \leq h_n$. Then, for $x \in \mathbb{R}^{n+1}$ the polynomials $\mathcal{H}_{\ell}^k(x)$ are real-valued.
\end{lemma}
\begin{proof}
The statement follows immediately from \eqref{gege}, since Gegenbauer polynomials $C_d^{(\ell)}\left(\frac{x_0}{|x|}\right)$ are real-valued.
\end{proof}
We show that the polynomials $ \mathcal{H}_k^\ell(x)$ have a close connection with the Riesz potential.

\begin{proposition}\label{pos_pow}
	Let $n\in\mathbb N$ be an odd number and assume that $1\leq \ell\leq h_n:=\frac{n-1}2$. Then, for $k \in \mathbb{N}$ and $x \in \mathbb{R}^{n+1}\setminus \{0\}$ we have
	\begin{equation}\label{k_der}
		\partial_{x_0}^{k-1} \left( \frac 1{|x|^{2\ell}} \right)=(k-1)! (-1)^{k+1} \mathcal H_\ell^{k-2+2\ell}(x) |x|^{-2k-2\ell+2}.
	\end{equation}
\end{proposition}

\begin{proof}
	By \eqref{gege} we deduce
	$$\mathcal H_\ell^{k-2+2\ell}(x)=|x|^{k-1}C_{k-1}^{(\ell)}\left(\frac{x_0}{|x|}\right).$$
	Hence, it is enough to prove that
	\begin{equation}\label{gegend}
		\partial_{x_0}^{k-1}\left(|x|^{-2\ell}\right)=(-1)^{k-1}(k-1)!|x|^{-2\ell-k+1}C_{k-1}^{(\ell)}\left(\frac{x_0}{|x|}\right).
	\end{equation}
	Now, we fix $x\in\mathbb R^{n+1}\setminus\{0\}$ and we consider $t \in \mathbb{R}$ such that $|t|<|x|$. Thus we have 
	$$|x-t|^{-2\ell} = \left(|x|^2-2x_0t+t^2\right)^{-\ell}$$
	Using the generating function of the Gegenbauer polynomials, see \eqref{genera}, with
	$$z=\frac{x_0}{|x|}\qquad\text{and}\qquad s=\frac{t}{|x|},$$
	we obtain
	\begin{equation}
		\label{power1}
		|x-t|^{-2\ell}=|x|^{-2 \ell} \left(1-2 \frac{x_0}{|x|^2}t+\frac{t^2}{|x|^2}\right)^{-\ell}
		=\sum_{j=0}^{\infty}|x|^{-2\ell-j}C_j^{(\ell)}\left(\frac{x_0}{|x|}\right)t^j.
	\end{equation}
	On the other hand, the Taylor expansion at $t=0$ gives
	\begin{equation}
		\label{power2}
		|x-t|^{-2\ell}=\sum_{j=0}^{\infty}\frac{(-1)^j}{j!}\partial_{x_0}^j\left(|x|^{-2\ell}\right)t^j.
	\end{equation}
	By comparing \eqref{power1} and \eqref{power2} and the uniqueness of the power series expansion, we get
	$$\partial_{x_0}^j\left(|x|^{-2\ell}\right)=(-1)^j j!|x|^{-2\ell-j}C_j^{(\ell)}\left(\frac{x_0}{|x|}\right).$$
	Taking $j=k-1$ yields \eqref{gegend}. This proves the result
\end{proof}

Now, we study an important property for the coefficients of the polynomials $ \mathcal{H}_\ell^k(x)$.

\begin{proposition}
	\label{summ}
	Let $ n $ be an odd number and define $h_n = \frac{n - 1}{2} $. For $ 1 \leq \ell \leq h_n$ and $ k \geq 2\ell - 1 $, the coefficients of $ \mathcal{H}_\ell^k(x)$ satisfy the following property:
	
	$$ \sum_{j=0}^{k-2\ell+1}\mathcal{T}_j^{k}(\ell)= \binom{k}{2\ell-1}.$$
\end{proposition}
\begin{proof}
	To prove the result we use the Vandermonde identity
	$$
	\sum_{j=0}^{N}\binom{r}{j}\binom{t}{N-j}=\binom{r+t}{N}, \qquad r,t \in \mathbb{R}, \quad N \in \mathbb{N}.
	$$
	Set $N:=k-2\ell+1$. Since
	$\binom{\ell+j-1}{\ell-1}=\binom{\ell+j-1}{j}$ and $\binom{k-\ell-j}{\ell-1}=\binom{\ell-1+N-j}{N-j}$,
	we have
	$$
	\sum_{j=0}^{k-2\ell+1}\mathcal{T}^k_j(\ell)=\sum_{j=0}^{N}\binom{\ell-1+j}{j}\binom{\ell-1+N-j}{N-j}.
	$$
	Using the identity 
	$$\binom{-\ell}{j}:=(-1)^j\binom{\ell-1+j}{j}$$
	 and the Vandermonde identity with
	$r=t=-\ell$ we obtain
	$$
	\sum_{j=0}^{N}\binom{\ell-1+j}{j}\binom{\ell-1+N-j}{N-j}
	=(-1)^N\sum_{j=0}^{N}\binom{-\ell}{j}\binom{-\ell}{N-j}
	=(-1)^N\binom{-2\ell}{N}=\binom{2\ell+N-1}{N}.
	$$
	Since $2\ell+N-1=k$ and $N=k-2\ell+1$, the right hand side equals $\binom{k}{k-2\ell+1}=\binom{k}{2\ell-1}$.
\end{proof}

Now, we want to show an important relation for the polynomials $\mathcal{H}_{\ell}^k(x)$, which is also of independent interest.
\begin{lemma} \label{hrecorsive} Let $n$ be an odd integer and set $h_n=\frac{n-1}{2}$, $x\in\mathbb R^{n+1}$ and $\ell$ be an integer with $1\leq \ell\leq h_n$. For $k$ being an integer such that $k\geq 2\ell -1$, we have
	\begin{equation}
		\mathcal{H}_{\ell+1}^{k+2}(x)-2x_0 \mathcal{H}_{\ell+1}^{k+1}(x)+|x|^2 \mathcal{H}_{\ell+1}^k(x)=\mathcal{H}_\ell^k(x),
	\end{equation}
	where we follow the convention introduced in Remark~\ref{conv}.
	
\end{lemma}

\begin{proof}
	We set $m:=k-2\ell+1\geq 0$. 
	%By \eqref{gege} we have
	%$$\mathcal{H}_\ell^k(x)=|x|^mC_m^\ell\left(\frac{x_0}{|x|}\right).$$
	%Similarly, we get
	%\begin{equation}
	%\mathcal{H}_{\ell+1}^{k+2}(x)=|x|^mC_m^{\ell+1}\left(\frac{x_0}{|x|}\right), \quad \mathcal{H}_{\ell+1}^{k+1}(x)=|x|^{m-1}C_{m-1}^{\ell+1}\left(\frac{x_0}{|x|}\right),\quad 	\mathcal{H}_{\ell+1}^{k}(x)=|x|^{m-2}C_{m-2}^{\ell+1}\left(\frac{x_0}{|x|}\right),
	%\end{equation}
	By the generating function of the Gegenbauer polynomials, see \eqref{genera}, we get
	$$(1-2zs+s^2)\sum_{j=0}^{\infty}C_j^{(\ell+1)}(z)s^j=(1-2zs+s^2)^{-\ell}=\sum_{j=0}^{\infty}C_j^{(\ell)}(z)s^j.$$
	This implies 
	$$\sum_{j=0}^{\infty}C_j^{(\ell+1)}(z)s^j-2z \sum_{j=1}^\infty C_{j-1}^{(\ell+1)}(z)s^j+\sum_{j=2}^{\infty}C_{j-2}^{(\ell+1)}(z)s^j=\sum_{j=0}^{\infty}C_j^{(\ell)}(z)s^j.$$
	Setting  $C_{-1}^{(\ell)}=C_{-2}^{(\ell)}=0$ and comparing the coefficients of $s^m$ gives
	$$C_m^{(\ell+1)}(z)-2zC_{m-1}^{(\ell+1)}(z)+C_{m-2}^{(\ell+1)}(z)=C_m^{(\ell)}(z),$$
	see also \cite[Lemma 2]{Fer}. Taking $t=x_0/|x|$ and usnig \eqref{gege}, we thus obtain
	\begin{align*}
		&\mathcal{H}_{\ell+1}^{k+2}(x)-2x_0\mathcal{H}_{\ell+1}^{k+1}(x)+|x|^2\mathcal{H}_{\ell+1}^{k}(x)\\
		&\quad =|x|^m\left[C_m^{(\ell+1)}\left(\frac{x_0}{|x|}\right)-2\frac{x_0}{|x|}C_{m-1}^{(\ell+1)}\left(\frac{x_0}{|x|}\right)+C_{m-2}^{(\ell+1)}\left(\frac{x_0}{|x|}\right)\right]
		\\
		&\quad =|x|^mC_m^{(\ell)}\left(\frac{x_0}{|x|}\right)=\mathcal{H}_\ell^k(x),
	\end{align*}
	which proves the result.
\end{proof}

Now, we prove a fundamental result for the paper, that asserts the connection between integer powers of the pseudo Cauchy-kernel and the polynomials $\mathcal{H}_\ell^k(x)$.

\begin{theorem}
	\label{exppser}
	Let $n$ be an odd number and $h_n:= \frac{n-1}{2}$. Then for $s$, $x \in \mathbb{R}^{n+1}$ such that $|x|<|s|$ the function $\mathcal{Q}^{-\ell}_{c,s}(x)$, with $1 \leq \ell \leq h_n$, admits the following expansion in series:
	\begin{equation}
		\label{series1}
		\mathcal{Q}_{c,s}^{-\ell}(x)= \sum_{k=2 \ell-1}^{\infty} \mathcal{H}_\ell^k(x)s^{-k-1}.
	\end{equation}
	Moreover, for $\nu \in \mathbb{N}$, we have that
	\begin{equation}
		\label{k12}
		K^2_{\nu,\ell}(s,x)= \sum_{k=2\ell-1}^{\infty} \sum_{j=0}^{\nu} \binom{\nu}{j} (-x_0)^j \mathcal{H}_{\ell}^k(x) s^{\nu-j-k}.
	\end{equation}
\end{theorem}
\begin{proof}
	First we prove the the following series
	\begin{equation}
		\label{seriesu}
		\mathcal{S}(s,x):=\sum_{k=2 \ell-1}^{\infty} \mathcal{H}_\ell^k(x)s^{-k-1},
	\end{equation}
	is convergent. By Proposition \ref{summ} we get
	\begin{eqnarray*}
		|\mathcal{S}(s,x)| &\leq& \sum_{k=2\ell-1}^{\infty}|\mathcal{H}_{\ell}^k(x)| |s|^{-k-1}\\
		&\leq& \sum_{k=2 \ell-1}^{\infty} \left(\sum_{j=0}^{k-2\ell+1} \mathcal{T}_j^k(\ell)\right)|x|^{k-2\ell+1}|s|^{-k-1}\\
		& \leq & \sum_{k=2 \ell-1}^{\infty} \binom{k}{2\ell-1}|x|^{k-2\ell+1}|s|^{-k-1}.
	\end{eqnarray*}
	Thus, by using the ratio test it is clear that the series $\mathcal{S}(s,x)$ is convergent. Now we prove by induction the equality \eqref{series1}. The case $\ell=1$ was proved in \cite[Theorem 4.30]{CDP25}. Now we suppose the equation \eqref{series1} holds for $\ell$ and we want to prove it for $\ell+1$. Thus, we have to show
	\begin{equation}\label{series3}
		\mathcal {Q}_{c,s}^{-\ell-1} (x)= \sum_{k=2 \ell +1}^{\infty} \mathcal{H}_{\ell+1}^k(x)s^{-k-1}
	\end{equation}
	which is equivalent, by the inductive hypothesis, to
	\begin{equation}\label{series2}
		\sum_{k=2 \ell-1}^{\infty} \mathcal{H}_\ell^k(x)s^{-1-k}= \mathcal{Q}_{c,s}(x) \sum_{k=2 \ell +1}^{\infty} \mathcal{H}_{\ell+1}^k(x)s^{-k-1}.
	\end{equation}
	Since $\mathcal{Q}_{c,s}(x)=s^2-2s x_0+|x|^2$ and $s\in\mathbb{R}^{n+1}$ commute with the polynomials $\mathcal{H}_\ell^k(x)$, see Lemma \ref{real}, we have that the right hand side of \eqref{series2} is equal to:
	$$
	\sum_{k=2\ell + 1}^\infty \mathcal {H}^k_{\ell+1}(x) s^{-k+1} -2x_0 \sum_{k=2\ell + 1}^\infty \mathcal {H}^k_{\ell+1}(x) s^{-k} +|x|^2\sum_{k=2\ell + 1}^\infty \mathcal {H}^k_{\ell+1}(x) s^{-k-1}.
	$$
	Now, we change the index in the first sum ($k$ into $k+2$) and in the second sum ($k$ into $k+1$) and we get:
	\[	
	\sum_{k =2\ell-1}^\infty \mathcal {H}^{k+2}_{\ell+1}(x) s^{-k-1} -2x_0 \sum_{k=2\ell }^\infty \mathcal {H}^{k+1}_{\ell+1}(x) s^{-k-1} +|x|^2\sum_{k=2\ell+1}^\infty \mathcal{H}^k_{\ell+1}(x) s^{-k-1}.
	\]
	By Lemma \ref{hrecorsive} applied to the previous expression, the following three equalities holds:
	\[
	\begin{cases}
		\mathcal {H}^{2\ell+1}_{\ell+1}(x)=\mathcal {H}^{2\ell-1}_{\ell}(x) & k=2\ell-1\\
		\mathcal {H}^{2\ell +2}_{\ell+1}(x)  -2x_0 \mathcal {H}^{2\ell +1}_{\ell+1}(x)=\mathcal {H}^{2\ell}_{\ell}(x) & k=2\ell\\
		\mathcal {H}^{k +2}_{\ell+1}(x)  -2x_0 \mathcal {H}^{ k +1}_{\ell+1}(x) +|x|^2 \mathcal {H}^{k}_{\ell+1}(x)=\mathcal {H}^{2k
		}_{\ell}(x) & k\geq 2\ell+1.
	\end{cases}
	\]
	These imply \eqref{series2}. Finally, formula \eqref{expn} follows by using the expansion in series \eqref{exxx} and the binomial theorem applied to \eqref{k1}. Finally, the equality in \eqref{k12} follows by using the binomial theorem and the expansion \eqref{series1} in the definition of the function $K_{\nu, \ell}^2(s,x)$, see \eqref{k2}.
\end{proof}

Now, we recall that the function $\mathcal{Q}_{c,s}^{-\ell}(x)$ can be obtained by applying the operator $D \Delta^{\ell-1}$ to the left slice hyperholomorphic Cauchy kernel, see Theorem \ref{p111} with $\beta=1$ and $m=\ell-1$, see also \cite[Thm. 4.16]{CDP25}.

\begin{corollary}
	\label{app2}
	Let $n$ be an odd integer and set $h_n := \frac{n-1}{2}$. Assume that $s, x \in \mathbb{R}^{n+1}$ with $s \notin [x]$. Then, for every integer $\ell$ such that $1 \leq \ell \leq h_n $, we have
	$$ D \Delta^{\ell-1}S^{-1}_L(s,x)=\sigma_{n,\ell} \mathcal{Q}_{c,s}^{-\ell}(x),$$
	where
	$$\sigma_{n ,\ell}:= 2^{2 \ell-1}(\ell-1)! (-h_n)_{\ell}.
	$$
\end{corollary}

\begin{proposition}
	Let $n$ be an odd number and set $h_n:= \frac{n-1}{2}$ and $1 \leq \ell \leq h_n$. Then, for $x \in \mathbb{R}^{n+1}$ and for every integer $k\geq 2\ell-1$, we have
	\begin{equation}
		\label{harm}
		D \Delta^{\ell-1}x^k=  \sigma_{n ,\ell}\mathcal{H}_\ell^k(x).
	\end{equation}
	Moreover, the polynomials $ \mathcal{H}_\ell^k(x)$ are axially Poly-Harmonic $\mathcal{APH}_{h_n-\ell+1}(U)$
	of order $h_n-\ell+1$.
\end{proposition}
\begin{proof}
	By Corollary \ref{app2} and Proposition \ref{cauchyseries}, for with $s \notin [x]$, we get
	\begin{equation}
		\mathcal{Q}_{c,s}^{-\ell}(x)= \frac{1}{\sigma_{n, \ell}}D \Delta^{\ell-1} S^{-1}_L(s,x)=\frac{1}{\sigma_{n, \ell}} \sum_{k=2 \ell-1}^{\infty}D \Delta^{\ell-1}  x^k s^{-1-k}.
	\end{equation}
	By \eqref{series1} we have
	$$\sum_{k=2 \ell-1}^{\infty}D \Delta^{\ell-1}  x^k s^{-1-k}=\sigma_{n, \ell} \sum_{k=2 \ell-1}^{\infty} \mathcal{H}_\ell^k(x)s^{-k-1}.$$
	Thus \eqref{harm} follows by putting equal the coefficients of the above power series.
	Now, we prove the regularity of the polynomials $\mathcal{H}_\ell^k(x)$.
	
	By Proposition~\ref{axial}, taking $\beta = 1$ and $m = \ell - 1$, we deduce that the operator 
	$
	D \Delta^{\ell-1}
	$
	applied to a slice function yields a function of axial type. Since 
	$
	x^k
	$
	is itself of axial type, relation~\eqref{harm} implies that the polynomials 
	$
	\mathcal{H}_\ell^k(x)
	$
	are also of axial type.
	Finally, in virtue of the Fueter-Sce mapping theorem, see Theorem \ref{FS1}, and \eqref{harm} we have
	$$ \Delta^{h_n-\ell+1}\mathcal{H}_\ell^k(x)=\frac{1}{\sigma_{n ,\ell}} D \Delta^{h_n}x^k=0.$$
	This proves the result.
\end{proof}

To deduce a series expansion of the kernels in \eqref{kk} and \eqref{kk1}, we introduce a new family of polynomials.

\begin{definition}[axially Analytic-Harmonic polynomials of type $(1, h_n-\ell)$]
	Let $n$ be an odd number and set $h_n:=\frac{n-1}{2}$. Then, for $1 \leq \ell \leq  h_n$ and $x \in \mathbb{R}^{n+1}$ we define the polynomials $ \mathcal{P}_\ell^k(x)$ as
	\begin{equation}\label{cliffpoly}
		\mathcal{P}_\ell^k(x)=\sum_{j=0}^{k-2\ell} \binom{\ell+j-1}{\ell-1} \binom{k-\ell-j}{\ell}x^{k-2\ell-j}\bar{x}^j.
	\end{equation}
\end{definition}

\begin{remark}
	If we consider $\ell=h_n$ we get a connection between the polynomials $ \mathcal{P}_{h_n}^k(x)$ and the Clifford-Appell polynomials, see \eqref{ca}, i.e.:
	$$ \mathcal{P}_{h_n}^k(x)= \frac{k!}{(2 h_n)!} \mathcal{Q}_n^k(x).$$
\end{remark}

The polynomials $ \mathcal{P}^{k}_\ell(x) $ are closely related to the polynomials $ \mathcal{H}_\ell^k(x)$, as stated in the following result.

\begin{proposition}
	\label{rell1}
	Let $n$ be an odd number and set $h_n:=\frac{n-1}{2}$. Then,
	for $1 \leq \ell \leq  h_n$ and $x \in \mathbb{R}^{n+1}$ we have
	\begin{equation}
		\label{rell}
		\mathcal{P}^k_\ell(x)= \mathcal{H}_{\ell+1}^{k+1}(x)-\bar{x} \mathcal{H}_{\ell+1}^k(x).
	\end{equation}
\end{proposition}
\begin{proof}
	By using the definition of the polynomials $ \mathcal{H}_{\ell}^k(x)$, see \eqref{harmpoly}, we have
	\begin{eqnarray}
		\label{star2}
		\mathcal{H}_{\ell+1}^{k+1}(x)-\bar{x} \mathcal{H}_{\ell+1}^k(x)&=& \sum_{j=0}^{k-2 \ell} C_j^{k+1}(\ell+1) x^{k-2\ell-j} \bar{x}^j- \sum_{j=0}^{k-2\ell-1}C_j^{k}(\ell+1)x^{k-2\ell-1-j} \bar{x}^{j+1}\\
		\nonumber
		&=& \sum_{j=1}^{k-2 \ell} \left[C_j^{k+1}(\ell+1)-C_{j-1}^{k}(\ell+1)\right]x^{k-2\ell-j} \bar{x}^j+C_0^{k+1}(\ell+1) x^{k-2\ell}.
	\end{eqnarray}
	By using the definition of the coefficients $C_j^k(\ell)$ and the Stifel identity we get
	\begin{equation}
		\label{ide}
		C_j^{k+1}(\ell+1)-C_{j-1}^{k}(\ell+1)= \binom{\ell+j-1}{\ell-1} \binom{k-\ell-j}{\ell}.
	\end{equation}
	Finally, by plugging \eqref{ide} into \eqref{star2} we get
	\begin{eqnarray*}
		\mathcal{H}_{\ell+1}^{k+1}(x)-\bar{x} \mathcal{H}_{\ell+1}^k(x)&=&\sum_{j=1}^{k-2 \ell}
		\binom{\ell+j-1}{\ell-1} \binom{k-\ell-j}{\ell}x^{k-2\ell-j} \bar{x}^j+\binom{k-\ell}{\ell} x^{k-2\ell}\\
		&=&\sum_{j=0}^{k-2 \ell}
		\binom{\ell+j-1}{\ell-1} \binom{k-\ell-j}{\ell}x^{k-2\ell-j} \bar{x}^j\\
		&=& \mathcal{P}_\ell^k(x).
	\end{eqnarray*}
	This proves the result.
\end{proof}
\begin{remark}
	\label{Jpol}
	By virtue of the connection between the polynomials $\mathcal{H}_\ell^k(x)$ and the Gegenbauer polynomials (see Proposition \ref{gege1}), a connection between the polynomials $\mathcal{P}_\ell^k(x)$ and the Gegenbauer polynomials exists as well.
\end{remark}
Now, we have all the tools to provide a series expansion for $(s-\bar{x}) \mathcal{Q}_{c,s}^{-\ell}(x)$.
\begin{theorem}
\label{blockE}
	Let $n$ be an odd number and set $h_n:= \frac{n-1}{2}$. Then, for $s$, $x \in \mathbb{R}^{n+1}$ such that $|x|<|s|$  and for $1 \leq \ell \leq h_n$, we have	
	\begin{equation}
		\label{exxx}
		(s-\bar{x})\mathcal{Q}_{c,s}^{-\ell}(x)= \sum_{k=2(\ell-1)}^{\infty} \mathcal{P}^k_{\ell-1}(x) s^{-1-k}.
	\end{equation}	
	Moreover, for $ \nu \in \mathbb{N}$, we have that
	\begin{equation}
		\label{expn}
		K_{\nu, \ell}^{1,L}(s,x)= \sum_{k=2 (\ell-1)}^{\infty} \sum_{j=0}^{\nu} \binom{\nu}{j} (-x_0)^j \mathcal{P}_{\ell-1}^k(x) s^{\nu-j-k-1}.
	\end{equation}
	
\end{theorem}
\begin{proof}
	We start proving the convergence of the series in \eqref{exxx}. By \eqref{rell}, Proposition \ref{summ} and Remark \ref{conv} we have
	\begin{eqnarray*}
		\sum_{k=2 (\ell-1)}^{\infty} | \mathcal{P}_{\ell-1}^k(x) s^{-1-k}| & \leq & \sum_{k=2 (\ell-1)}^{\infty} | \mathcal{H}_{\ell}^{k+1}(x)||s|^{-1-k}+|x| \sum_{k= 2 (\ell -1)}^{\infty} | \mathcal{H}_{\ell}^k(x)| |s|^{-1-k}\\
		&\leq& \sum_{k=2 (\ell-1)}^{\infty}\left[ \sum_{j=0}^{k-2\ell+2}C_j^{k+1}(\ell)+ \sum_{j=0}^{k-2 \ell+1} C_j^{k}(\ell)\right]|x|^{k-2\ell}|s|^{-1-k}\\
		& = & \sum_{k=2 (\ell-1)}^{\infty} \left[\binom{k+1}{2\ell-1}+\binom{k}{2 \ell-1}\right]|x|^{k-2\ell+2}|s|^{-1-k}.
	\end{eqnarray*}
	Since we are supposing that $|x|<|s|$, by the ratio test we get that the above series is convergent.
	Now, we prove the equality in \eqref{exxx}.
	By using the expansion of $ \mathcal{Q}_{c,s}^{-\ell}(x)$, obtained in \eqref{series1}, we have
	\begin{eqnarray*}
		(s-\bar{x}) \mathcal{Q}_{c,s}^{-\ell}(x)&=& \mathcal{Q}_{c,s}^{-\ell}(x)s-\bar{x} \mathcal{Q}_{c,s}^{-\ell}(x)\\
		&=& \sum_{k=2 \ell-1}^{\infty} \mathcal{H}_{\ell}^{k}(x)s^{-k}-\bar{x} \sum_{k=2 \ell-1}^{\infty} \mathcal{H}_{\ell}^k(x) s^{-1-k}\\
		&=& \sum_{k=2 (\ell-1)}^{\infty} \mathcal{H}_{\ell}^{k+1}(x)s^{-1-k}-\bar{x}\sum_{k=2(\ell-1)}^{\infty} \mathcal{H}_{\ell}^k(x) s^{-1-k},
	\end{eqnarray*}
	where, in the second summation, we use the fact that $ \mathcal{H}_{\ell}^{2\ell-1}(x)=0$, see Remark \ref{conv}. By \eqref{rell} we get
	$$ (s-\bar{x}) \mathcal{Q}_{c,s}^{-\ell}(x)= \sum_{k=2 (\ell-1)}^{\infty} \left[\mathcal{H}_{\ell}^{k+1}(x)-\bar{x} \mathcal{H}_{\ell}^k(x)\right]s^{-1-k}= \sum_{k= 2(\ell-1)}^{\infty} \mathcal{P}^k_{\ell-1}(x) s^{-1-k}.$$
	This proves formula \eqref{exxx}.
\end{proof}

\begin{proposition}
	\label{lapapp}
	Let $n$ be an odd number, set $h_n:= \frac{n-1}{2}$ and suppose $1 \leq \ell \leq h_n$. Then, for $x \in \mathbb{R}^{n+1}$, we have
	\begin{equation}
		\label{app3}
		\Delta^\ell x^k=\gamma_{n,\ell} \mathcal{P}^{k}_\ell(x), \qquad k \geq 2 \ell,
	\end{equation}
	where $\gamma_{n ,\ell}$ is given in  Theorem \ref{app1}. Moreover, the polynomials $\mathcal{P}^{k}_\ell(x)$ are axially Analytic-Harmonic
	$\mathcal{AAH}_{1, h_n-\ell}(\mathbb{R}^{n+1})$ of type $(1, h_n-\ell)$.
\end{proposition}
\begin{proof}
	By Theorem \ref{app1} and Proposition \ref{cauchyseries} we have
	$$ (s-\bar{x}) \mathcal{Q}_{c,s}^{-\ell-1}(x)= \frac{1}{ \gamma_{n, \ell}}\Delta^{\ell}S^{-1}_L(s,x)=\frac{1}{ \gamma_{n, \ell}} \sum_{k=2 \ell}^{\infty} \Delta^{\ell} x^k s^{-1-k}.$$
	By \eqref{exxx} we have
	\begin{equation}
		\label{power}
		\sum_{k=2\ell}^{\infty} \Delta^{\ell} x^k s^{-1-k}=\gamma_{n, \ell} \sum_{k=2 \ell}^{\infty} \mathcal{P}_\ell^k(x)s^{-1-k}.
	\end{equation}
	The equality in \eqref{app3} then follows by equating the coefficients in the power series expansion of \eqref{power}. Now, by Proposition \ref{axial}, taking $\beta=0$, and \eqref{app3} we deduce that the polynomials $\mathcal{P}_\ell^k(x)$ are of axial form. Finally, since $x^k$ is a slice hyperholomorphic, by the Fueter-Sce mapping theorem (see Theorem \ref{FS1}) we get
	$$ D \Delta^{h_n-\ell}\mathcal{P}_\ell^k(x)= \frac{1}{\gamma_{n, \ell}}  D \Delta^{h_n} x^k=0.$$
\end{proof}

Now, we have developed all the tools to provide an expansion in series of $S^{-1}_{L,D^\beta \Delta^m}(s,x)$ and $S^{-1}_{L,\overline{D}^\beta \Delta^m}(s,x)$. The results below follows by substituting the expansions of $ K^{1,L}_{\nu, \ell}(s,x) $ and $ K^{2}_{\nu, \ell}(s,x) $ (see \eqref{expn} and \eqref{k12}, respectively) into the expression \eqref{do1}, \eqref{de1}, \eqref{bar1}, and \eqref{bar2}.

\begin{proposition}
	\label{expser}
	Let $n$ be an odd number and set $h_n:=\frac{n-1}{2}$. Let  $\beta$, $m \in \mathbb{N}_0$ such that $m +\beta\leq h_n$. We assume that $s$, $x \in \mathbb{R}^{n+1}$ such that $|x|<|s|$.
	Then,
	if $\beta=2k_1+1$, where $k_1 \in \mathbb{N}$, we have
	\begin{eqnarray}
		\nonumber
		S^{-1}_{L,D^\beta \Delta^m}(s,x)&=& c_{\beta,m, h_n}\left( \sum_{j=0}^{k_1-1} \sum_{k=2m+2j+2k_1+2}^{\infty} \sum_{\alpha=0}^{2j+1} a^{1}_{j,k_1,m} \binom{2j+1}{\alpha} (-x_0)^\alpha \mathcal{P}^k_{m+j+k_1+1}(x)s^{2j-\alpha-k} \right.\\
		\label{Dser}
		&&\left. -\sum_{j=0}^{k_1} \sum_{k=2m+2j+1+2k_1}^{\infty} \sum_{\alpha=0}^{2j} b_{j,k_1,m}^1\binom{2j}{\alpha} (-x_0)^\alpha \mathcal{H}^k_{m+j+1+k_1}(x) s^{2j-\alpha-k} \right).
	\end{eqnarray}
	If $\beta$ is even, i.e. $\beta=2k_2$, with $k_2 \in \mathbb{N}$ we have
	\begin{eqnarray}
		\nonumber
		S^{-1}_{L,D^\beta \Delta^m}(s,x)&=& c_{\beta,m, h_n}\left( \sum_{j=0}^{k_2-1} \sum_{k=2m+2j+2k_2}^{\infty} \sum_{\alpha=0}^{2j} a^{1}_{j,k_2,m} \binom{2j}{\alpha} (-x_0)^\alpha \mathcal{P}^k_{m+j+k_2}(x)s^{2j-\alpha-k-1} \right.\\
		\label{Dser1}
		&&\left.\! \! \! \! \! \! \! \!	\! \!\! \!\! \!	\! \!\! \!	 -\sum_{j=0}^{k_2-1} \sum_{k=2m+2j+1+2k_2}^{\infty} \sum_{\alpha=0}^{2j+1} b_{j,k_2,m}^2 \binom{2j+1}{\alpha} (-x_0)^\alpha \mathcal{H}^k_{m+j+1+k_2}(x) s^{2j-\alpha-k+1} \right)
	\end{eqnarray}
	where
	$$c_{\beta,m, h_n}:=\frac{2^\beta (h_n-m)\gamma_{n,m}}{m!}, \qquad \gamma_{n,m}:=4^m m! (-h_n)_m$$ and the coefficients $a^{1}_{j,k_1,m}$, $b_{j,k_1,m}^1$ are defined in \eqref{coeff} and the coefficients $a^{2}_{j,k_2,m}$, $b^2_{j,k_2,m}$ are defined in \eqref{coef}.
	
\end{proposition}

\begin{proposition}
	\label{expser1}
	Let $n$ be an odd number and set $h_n:=\frac{n-1}{2}$. Let  $\beta$, $m \in \mathbb{N}_0$ such that $m +\beta\leq h_n$.
	We assume that $s$, $x \in \mathbb{R}^{n+1}$ such that $|x|<|s|$. Then, if $\beta=2k_1+1$, where $k_1 \in \mathbb{N}$, we have
	\begin{eqnarray}
		\nonumber
		S^{-1}_{L,\overline{D}^\beta \Delta^m}(s,x)&=& \mathbf{c}_{\beta,m, h_n}\left( \sum_{j=0}^{k_1} \sum_{k=2m+2j+2k_1+2}^{\infty} \sum_{\alpha=0}^{2j+1} \mathbf{a}^{1}_{j,k_1,m} \binom{2j+1}{\alpha} (-x_0)^\alpha \mathcal{P}^k_{m+j+k_1+1}(x)s^{2j-\alpha-k} \right.\\
		\label{Dbars}
		&&\left. +\sum_{j=0}^{k_1} \sum_{k=2m+2j+1+2k_1}^{\infty} \mathbf{b}^1_{j,k_1,m} \sum_{\alpha=0}^{2j} \binom{2j}{\alpha} (-x_0)^\alpha \mathcal{H}^k_{m+j+k_1+1}(x) s^{2j-\alpha-k} \right).
	\end{eqnarray}
	If $\beta$ is even, i.e. $\beta=2k_2$, with $k_2 \in \mathbb{N}$ we have
	\begin{eqnarray}
		\nonumber
		S^{-1}_{L,\overline{D}^\beta \Delta^m}(s,x)&& = \mathbf{c}_{\beta,m, h_n}\left( \sum_{j=0}^{k_2} \sum_{k=2m+2j+2k_2}^{\infty} \sum_{\alpha=0}^{2j} \mathbf{a}^{2}_{j,k_2,m} \binom{2j}{\alpha} (-x_0)^\alpha \mathcal{P}^k_{m+j+k_2}(x)s^{2j-\alpha-k-1} \right.\\
		\label{Dbars1}
		&&\left. \! \! \! \! \! \! \! \!\! \!\! \!	\! \!\! \!	+\sum_{j=0}^{k_2-1} \sum_{k=2m+2j+1+2k_2}^{\infty}\mathbf{b}^2_{j,k_2,m} \sum_{\alpha=0}^{2j+1} \binom{2j+1}{\alpha} (-x_0)^\alpha \mathcal{H}^k_{m+j+1+k_2}(x) s^{2j-\alpha-k+1} \right),
	\end{eqnarray}
	where $$\mathbf{c}_{\beta,m, h_n}:=2^\beta 4^m (-h_n)_m,$$ the coefficients $\mathbf{a}^{1}_{j,k_2,m}$ and $\mathbf{b}^1_{j,k_2,m}$  are defined in \eqref{c1}, the coefficients $\mathbf{a}^{2}_{j,k_2,m}$ and $\mathbf{b}^2_{j,k_2,m}$ are defined in \eqref{c2}.
\end{proposition}

\begin{remark}
	The series expansions given in Proposition~\ref{expser} and Proposition~\ref{expser1} are quite general and encompass two notable special cases: the Poly-Harmonic and Poly-Analytic kernels.
	
	The expansion of the Poly-Harmonic kernel can be recovered from equation~\eqref{Dser} by setting $\beta = 1 $. In this case, we obtain
	$$
	S^{-1}_{L, D \Delta^m}(s, x) = -2(h_n - m)\gamma_{n, m} \sum_{k = 2m + 1}^{\infty} \mathcal{H}^k_{m+1}(x) s^{-1 - k}.
	$$
	
	On the other hand, the expansion of the Poly-Analytic kernel can be recovered from equations~\eqref{Dbars} and~\eqref{Dbars1} by taking $ \beta = h_n - m $. Specifically, the series expansion in the Poly-Analytic case is given by
	$$
	S^{-1}_{L, \overline{D}^{h_n - m} \Delta^m}(s, x) = 2^{2h_n} h_n! (-h_n)_m (s - x_0)^{h_n - m} \sum_{k = 2h_n}^{\infty} \mathcal{P}^k_{h_n}(x) s^{-1 - k}.
	$$
	
\end{remark}

\begin{remark}
	By Proposition~\ref{summ} and Remark~\ref{Jpol}, the series expansions of $ S^{-1}_{L,D^\beta \Delta^m}(s,x) $ and $ S^{-1}_{L,\overline{D}^\beta \Delta^m}(s,x) $ can be expressed in terms of Gegenbauer polynomials.
\end{remark}

\section{Functional calculi of the Dirac fine structures}\label{FUNC-CAL_DIRAC}	

In this section we aim to establish the functional calculi for the Dirac fine structures. This functional calculi are more general then the one presented in \cite{CDP25, CDPS, Polyf1, Polyf2}. In order to introduce these functional calculi we need to recall the basic notions of the
spectral theory based on the $S$-spectrum and the $S$-functional calculus.

\subsection{The $S$-functional calculus for $n$-tuples of operators}

We denote by $V$ the Banach space over $ \mathbb{R}$ with norm given by $\|.\|$. The tensor product $V_n=V \otimes \mathbb{R}_n$ is a two-sided Banach module over $ \mathbb{R}_n$. An element in the space $V_n$ is of the type $\sum_{A} v_A \otimes e_A$, where $e_A$ has been introduced at the beginning of Section 2. The multiplications (right and left) of an element $v \in V_n$ by a scalar $a \in \mathbb{R}_n$ are defined as
	
	 $$ va= \sum_{A} v_A \otimes (e_Aa) \quad \hbox{and} \quad av=\sum_{A} v_A \otimes (ae_A).$$
In the sequel for short we will write $\sum_{A} v_A e_A$ instead of $\sum_{A} v_A \otimes e_A$. We denote by $\mathcal{B}(V)$ the space of bounded $ \mathbb{R}$-homomorphism of the Banach space $V$ into itself endowed with the natural norm denoted by $\| .\|_{\mathcal{B}(V)}$. If we assume that $ T_A \in \mathcal{B}(V)$, we can define the operator as follows:

$$T= \sum_{A} T_A e_A,$$
and its action on $v= \sum_{B} v_B e_B$ as
$$T(v)=\sum_{A,B} T_A(v_B)e_A e_B$$
The set of all such right bounded linear operators is denoted by $ \mathcal{B}(V_n)$ and a norm is defined by
$$ \| T\|_{\mathcal{B}(V_n)}=\sum_{A} \| T_A\|_{\mathcal{B}(V)}.$$

The set of bounded operators, in paravector form, $T= \sum_{\mu=0}^{n}e_{\mu} T_{\mu}$, where $T_\mu \in \mathcal{B}(V)$ for $\mu=0,1,...,n$, will be denoted by $ \mathcal{B}^{0,1}(V_n)$. The conjugate of the operator $T$ is given by $\overline{T}=T_0-\sum_{\mu=1}^{n} e_{\mu}T_{\mu}$. The subset of operators in $\mathcal{B}^{0,1}(V_n)$ whose components commute among themselves will be denoted by $ \mathcal{BC}^{0,1}(V_n)$. If we consider $T \in \mathcal{BC}^{0,1}(V_n)$ the operator $T \bar{T}$ is well defined and is given by $T\overline{T}=\overline{T}T= \sum_{\mu=0}^{n} T_\mu^2$ and $T+\bar{T}=2T_0$.

Now, we have all the tools to recall the notion of spectrum and resolvent operators within the Clifford framework. In \cite{ColomboSabadiniStruppa2011} the authors have proved that for $T \in \mathcal{B}^{0,1}(V_n)$, $s \in \mathbb{R}^{n+1}$ such that $\|T\|<|s|$, then
\begin{equation}
\label{basicS}
\sum_{m=0}^{\infty} T^m s^{-1-m}=-\mathcal{Q}_s(T)^{-1}(T-\bar{s}\mathcal{I}) \quad \hbox{and} \quad \sum_{m=0}^{\infty}  s^{-1-m}T^m=-(T-\bar{s} \mathcal{I})\mathcal{Q}_s(T)^{-1},
\end{equation}
where
\begin{equation}
\label{ncommQ}
\mathcal{Q}_s(T):=T^2-2s_0T+|s|^2 \mathcal{I}.
\end{equation}

We observe that the right hand sides of \eqref{basicS} are the Cauchy kernels $S^{-1}_L(s,x)$ and $S^{-1}_R(s,x)$ written in form I, where we have formally replaced the paravector $x \in \mathbb{R}^{n+1}$ by the paravector $T \in \mathcal{BC}^{0,1}(V_n)$.

\begin{remark}
It is remarkable that the series in \eqref{basicS} for $\|T\|<|s|$ have a closed form also when the components of the operator given by $T_\mu \in \mathcal{B}(V)$ do not commute among themselves
\end{remark}
The following definitions arise naturally, from \eqref{basicS}, see \cite{ColomboSabadiniStruppa2011}.

\begin{definition}[The $S$-resolvent set, the $S$-spectrum and the $S$-resolvent operators]
Let $T \in \mathcal{B}^{0,1}(V_n)$. The $S$-resolvent set of $T$ is defined as
$$ \rho_S(T)=\{s \in \mathbb{R}^{n+1}\, : \, \mathcal{Q}_s^{-1}(T) \in \mathcal{B}(V_n)\},$$
and the $S$-spectrum of $T$ is given by
$$ \sigma_S(T):= \mathbb{R}^{n+1}\setminus \rho_S(T).$$
Furthermore, for $s \in \rho_S(T)$ we define the left and the right $S$-resolvent operators as
$$ S^{-1}_L(s,T)=-\mathcal{Q}_{s}^{-1}(T)(T-\bar{s} \mathcal{I}), \quad \left( \hbox{resp.} \quad S^{-1}_R(s,T)=-(T-\bar{s} \mathcal{I})\mathcal{Q}_s^{-1}(T)\right).$$
\end{definition}

The Cauchy kernel series written in the form II (see Definition \ref{Ckernel}) is crucial to define a functional calculus for bounded operators with components that commute among themselves.
\begin{theorem}
Let $T \in \mathcal{BC}^{0,1}(V_n)$ and $s \in \mathbb{R}^{n+1}$ be such that $\|T\|<|s|$. Then we have
\begin{equation}
\label{commser}
\sum_{m=0}^{\infty}T^m s^{-1-m}=(s\mathcal{I}-\bar{T})\mathcal{Q}_{c,s}^{-1}(T), \quad \hbox{and} \quad \sum_{m=0}^{\infty}s^{-1-m} T^m =\mathcal{Q}_{c,s}^{-1}(T)(s\mathcal{I}-\bar{T}),
\end{equation}
where
\begin{equation}
\label{commQ}
\mathcal{Q}_{c,s}(T)= s^2 \mathcal{I}-s(T+\bar{T})+T \bar{T}.
\end{equation}
\end{theorem}

The above result suggests that the notion of spectrum and of resolvent sets of $T$ for paravector operators with commuting components.

\begin{definition}[$F$-resolvent set and $F$-spectrum]
Let $T \in \mathcal{BC}^{0,1}(V_n)$. We define the $F$-resolvent set of $T$ as
$$ \rho_F(T)= \{s \in \mathbb{R}^{n+1} \, : \, \mathcal{Q}_{c,s}(T)^{-1} \in \mathcal{B}(V_n)\},$$
and the $F$-spectrum of $T$ as
$$ \sigma_F(T)=\mathbb{R}^{n+1}\setminus \rho_F(T).$$
\end{definition}

Based on the strategy in \cite[Thm. 4.5.6]{CGK} and in \cite{CDP25} it was proved the following relation between the $F$-spectrum and the $S$-spectrum.

\begin{theorem}
Let $T \in \mathcal{BC}^{0,1}(V_n)$. Then, the operator $ \mathcal{Q}_{c,s}(T)$ (defined in \eqref{commQ})
 is invertible if and only if the operator $\mathcal{Q}_s(T)$ (defined in \eqref{ncommQ}) is invertible and so
$$\sigma_S(T)=\sigma_F(T).$$
\end{theorem}
By \eqref{commser}, we can define the left (resp. right) $S$-resolvent for paravector operators whose components commute among themselves, as
$$ S^{-1}_L(s,T)=(s \mathcal{I}-\bar{T}) \mathcal{Q}_{c,s}(T)^{-1}, \qquad \left( \, \hbox{resp.} \quad S^{-1}_R(s,T)=\mathcal{Q}_{c,s}(T)^{-1}(s \mathcal{I}-\bar{T})\right).$$
By mimicking the definition of slice hyperholomorphic function, see Definition \ref{sh}, we can provide the definition of slice hyperholomorphicity operator-valued.

\begin{definition}
	Let $U \subseteq \mathbb{R}^{n+1}$ be an axially symmetric open set. An operator valued function $K:U \to \mathcal{B}(V_n)$ is called left (resp. right) slice hyperholomorphic, if there exist operator valued functions $A$, $B : \mathcal{U} \to \mathcal{B}(V_n) $ with the set $ \mathcal{U}$ as in \eqref{twente}, such that for every $(u,v) \in \mathcal{U}$:
	\begin{itemize}
		\item[i)] The operator $K$ admit for every $I \in \mathbb{S}$ the representation:
		$$K(u+Iv)=A(u,v)+IB(u,v), \quad \left(\hbox{resp.} \, \, K(u+Iv)=A(u,v)+B(u,v)I\right).$$
		\item[ii)] The operators $A$, $B$ satisfy the even-odd conditions:
		$$A(u,-v)=A(u,v), \quad \hbox{and} \quad B(u,-v)=-B(u,v).$$
		\item[iii)] The operators $A$, $B$ satisfy the Cauchy-Riemann equations
		$$ \frac{\partial}{\partial u}A(u,v)=\frac{\partial}{\partial v} B(u,v), \quad \hbox{and} \quad \frac{\partial}{\partial u}A(u,v)=-\frac{\partial}{\partial v} B(u,v)$$
		with the derivatives are understood in the natural operator norm convergence sense.
	\end{itemize}
	Moreover, if the operators $A$, $B$ are two-sided linear we call the operator $K(s)$ intrinsic.
\end{definition}

In the variable $s$ the resolvent operators satisfy the following regularity, see \cite[Lemma 3.1.11]{CGK}.

\begin{lemma}
Let $n \in \mathbb{N}$ and $T \in \mathcal{BC}^{0,1}(V_n)$. Then, the left (resp. right) $S$-resolvent operator is a $\mathcal{B}(V_n)$-valued right (resp. left) slice hyperholomorphic function of the variable $s$ on $\rho_S(T)$.
\end{lemma}

\begin{definition}
Let $T \in \mathcal{BC}^{0,1}(V_n)$ and let $U \subset \mathbb{R}^{n+1}$ be a bounded slice Cauchy domain. We denote by $ \mathcal{SH}^L_{\sigma_S(T)}(U)$, $ \mathcal{SH}^R_{\sigma_S(T)}(U)$ and $ \mathcal{N}_{\sigma_S(T)}(U)$ the set of all left, right and intrinsic slice hyperholomorphic function $f$ with $\sigma_S(T) \subset U $ and $\overline{U} \subset \mathcal{D}(f)$, where $ \mathcal{D}(f)$ is the domain of the function $f$.
\end{definition}
Inspired by the Cauchy formula for slice hyperholomorphic functions (see~\eqref{cauchynuovo}), we introduce the $S$-functional calculus.

\begin{definition}[$S$-functional calculus]
Let $T \in \mathcal{BC}^{0,1}(V_n)$, $U \subset \mathbb{R}^{n+1}$ be a bounded slice Cauchy domain, let $I \in \mathbb{S}$ and set $ds_I=ds(-I)$. For any $f \in \mathcal{SH}^L_{\sigma_S(T)}(U)$ (resp. $f \in \mathcal{SH}_{\sigma_S(T)}^R(U)$) we define
\begin{equation}
\label{Sfun}
f(T)= \frac{1}{2 \pi} \int_{\partial(U \cap \mathbb{C}_I)} S^{-1}_L(s,T) ds_I f(s), \quad \left( \hbox{resp.} \quad f(T)= \frac{1}{2 \pi} \int_{\partial(U \cap \mathbb{C}_I)} f(s)ds_I S^{-1}_R(s,T)\right).
\end{equation}
\end{definition}

\begin{remark}
Since the integrals in \eqref{Sfun} depend neither on $U$ nor on the imaginary unit $I \in \mathbb{S}$ we have that the $S$-functional calculus is well-defined.
\end{remark}

\subsection{The $D$ and $\overline{D}$ functional calculi on the $S$-spectrum}

The goal of this section is to develop functional calculi based on the $S$-spectrum corresponding to each possible factorization of the second map in the Fueter–Sce theorem, using operators of the form $ D^\beta \Delta^m $ and $\overline{D}^\beta \Delta^m $, with $m$, $ \beta \in \mathbb{N}$ such that $m + \beta \leq h_n $. In order to introduce the resolvent operators associated with the Dirac fine structures, we first present some preliminary results. We begin by introducing two novel bounded operators with commuting components:

\begin{definition}[Poly-Harmonic polynomials operators]
	Let $n$ be an odd number and $h_n=\frac{n-1}{2}$. Then for $1 \leq \ell \leq h_n$ and $T \in \mathcal{BC}^{0,1}(V_n)$ we define the Clifford Poly-Harmonic operators as
\begin{equation}
	\label{harmope}
	\mathcal{H}_{\ell}^k(T):= \sum_{j=0}^{k-2\ell+1} \binom{\ell+j-1}{\ell-1} \binom{k-\ell-j}{\ell-1} T^{k-2\ell+1-j}\bar{T}^{j}, \qquad k \geq 2 \ell-1.
\end{equation}
\end{definition}

\begin{remark}
\label{conv1}
This family of operators formally arises by replacing the paravector $ x $ with the operator $T $, whose components commute, in~\eqref{harmpoly}.
 The Poly-Harmonic polynomials operators are zero if $k-2 \ell+1<0$.
\end{remark}

\begin{definition}[Analytic-Harmonic polynomials operators (of type $(1, h_n-\ell)$)]
Let $n$ be an odd number and $h_n=\frac{n-1}{2}$. Then for $1 \leq \ell \leq h_n$ and $T \in \mathcal{BC}^{0,1}(V_n)$ we define the Analytic-Harmonic polynomials operators (of type $(1, h_n-\ell)$) as
\begin{equation}
	\label{cliffope}
\mathcal{P}^k_\ell(T)=\sum_{j=0}^{k-2\ell} \binom{\ell+j-1}{\ell-1} \binom{k-\ell-j}{\ell}T^{k-2\ell-j}\bar{T}^j, \qquad k \geq 2 \ell.
\end{equation}
\end{definition}

\begin{remark}
Formally, this family of operators is derived by replacing the paravector $x$ in~\eqref{cliffpoly} with a commuting operator $T$. By using similar arguments in Proposition \ref{rell1} we can prove the following relation between the Poly-Harmonic and Analytic-Harmonic polynomials operators:
\begin{equation}
\label{relope}
\mathcal{P}^k_\ell(T):= \mathcal{H}_{\ell+1}^{k+1}(T)-\bar{T} \mathcal{H}_{\ell+1}^k(T), \qquad k \geq 2 \ell.
\end{equation}
\end{remark}

\begin{proposition}
\label{Qint}
Let $n$ be an odd number and $h_n:=\frac{n-1}{2}$, with $1 \leq \ell \leq h_n$. We assume that $T \in \mathcal{BC}^{0,1}(V_n)$ and $s \in \mathbb{R}^{n+1}$ being such that $\|T\| < |s|$. Then we have
\begin{equation}
\mathcal{Q}_{c,s}^{-\ell}(T)= \sum_{k=2\ell-1}^{\infty} \mathcal{H}_{\ell}^k(T) s^{-k-1}.
\end{equation}

Moreover, the operator $\mathcal{Q}_{c,s}^{-\ell}(T)$ is a $ \mathcal{B}(V_n)$-intrinsic slice hyperholomorphic.

\end{proposition}
\begin{proof}
The result follows by similar arguments used in Theorem \ref{exppser}. Finally, since $\mathcal{Q}_{c,s}^{-1}(T) $ is an intrinsic slice hyperholomorphic function with values in $ \mathcal{B}(V_n) $, and the product of intrinsic slice hyperholomorphic functions is again intrinsic slice hyperholomorphic, the result follows.

\end{proof}

\begin{proposition}
\label{Qope}
Let $n$ be an odd number and $h_n:=\frac{n-1}{2}$, with $1 \leq \ell \leq h_n$. We assume that $T \in \mathcal{BC}^{0,1}(V_n)$ and $s \in \mathbb{R}^{n+1}$ is such that $\|T\| < |s|$. Then for $\nu \in \mathbb{N}$ we have
\begin{equation}
	\label{res01}
 \sum_{k=2\ell-1}^{\infty} \sum_{j=0}^{\nu} \binom{\nu}{j} (-T_0)^j \mathcal{H}_{\ell}^k(T) s^{\nu-j-k-1}=(s \mathcal{I}-T_0)^\nu \mathcal{Q}_{c,s}^{-\ell}(T),
\end{equation}
\begin{equation}
\label{res2}
\sum_{k=2 \ell-2}^{\infty} \sum_{j=0}^{\nu} \binom{\nu}{j} (-T_0)^j \mathcal{P}_{\ell-1}^k(T) s^{\nu-j-k-1}=(s\mathcal{I}-\bar{T}) (s \mathcal{I}-T_0)^{\nu} \mathcal{Q}_{c,s}^{-\ell}(T),
\end{equation}
where the series converges in the operator norm.
\end{proposition}
\begin{proof}
Since the operator $T$ is bounded and has commuting components, we can use arguments similar to those used in Theorem~\ref{blockE}, with the paravector $x$ replaced by $T$.
\end{proof}

The left-hand sides of equations \eqref{res01} and \eqref{res2} are well-defined only under the condition $ \|T\| < |s| $, whereas their right-hand sides are defined for all $ s \in \mathbb{R}^{n+1} \setminus \sigma_S(T)$. This observation motivates the following definition.

 \begin{definition}[$K^1$ and $K^2$ resolvent operators]
 Let $n$ be an odd number and set $h_n= \frac{n-1}{2}$. We assume $1 \leq \ell \leq h_n$ and $\nu \in \mathbb{N}$. For $s \in \rho_S(T)$ and $T \in \mathcal{BC}^{0,1}(V_n)$ we define the $K_1$ and $K_2$ resolvent operators as
 \begin{equation}\nonumber
 	K_{\nu,\ell}^{1,L}(s,T)=(s\mathcal{I}-\bar{T}) (s\mathcal{I}-T_0)^{\nu} \mathcal{Q}_{c,s}^{-\ell}(T),
 \end{equation}
 and
 \begin{equation}\nonumber
 	K_{\nu,\ell}^{2}(s,T)= (s \mathcal{I}-T_0)^{\nu} \mathcal{Q}_{c,s}^{-\ell}(T).
 \end{equation}
 \end{definition}

Inspired by the kernel of the integral representations of the $ D $ and $ \overline{D}$-fine structures  (see Theorems \ref{Dinte} and \ref{Dbarinte}, respectively), we introduce the notions of the corresponding resolvent operators. These can be expressed as linear combinations of the \( K^1 \) and \( K^2 \) resolvent operators.

\begin{definition}[$D$-resolvent operators]
	\label{Dresope}
	Let $n$ be an odd number and set $h_n:=(n-1)/2$. Let $m$, $\beta \in \mathbb{N}_0$ be such that $\beta+m\leq h_n$. We set
$$
c_{\beta, m, h_n}:=\frac{2^{\beta} (h_n-m)! \gamma_{n,m}}{m!}, \qquad \gamma_{n, m}=4^m m! (-h_n)_m
$$
 We assume $T \in \mathcal{BC}^{0,1}(V_n)$ and $s \in \rho_S(T)$. If $\beta=2k_1+1$, where $k_1\in\mathbb N$, the $D$-resolvent operators are defined as
	\begin{equation}
		\label{do1__res}
	S^{-1}_{L,D^{\beta} \Delta^m}(s,T)= c_{\beta, m, h_n}\left(  \sum_{j=0}^{k_1-1} a^1_{j,k_1,m} K^{1,L}_{2j+1, m+j+2+k_1}(s,T)  -\sum_{j=0}^{k_1} b^1_{j,k_1,m}	K_{2j,m+1+k_1+j}^{2}(s,T)  \right),
	\end{equation}
	and if $\beta=2k_2$, where $k_2\in\mathbb N$, we have
	\begin{equation}\label{de1_res}
		S^{-1}_{L, D^{\beta}  \Delta^m }(s,T)= c_{\beta, m, h_n} \left(  \sum_{j=1}^{k_2-1} a^2_{j,k_2,m} K^{1,L}_{2j, m+j+1+k_2}(s,T) -\sum_{j=0}^{k_2-1} b^2_{j,k_2,m} K_{2j+1, m+1+j+k_2}^2(s,T)  \right),
	\end{equation}
	where the coefficients $a^1_{j,k_1,m}$, $b^1_{j,k_1,m}$, $a^2_{j,k_2,m}$ and $b^2_{j,k_2,m}$  are defined in \eqref{coeff} and \eqref{coef}.
\end{definition}

\begin{definition}[$\overline D$-resolvent operators]
	\label{Dbaroperes}
	Let $n$ be an odd number and set $h_n:=(n-1)/2$. Let $m$, $\beta \in \mathbb{N}_0$ be such that $\beta+m\leq h_n$.
We set
$$
\mathbf{c}_{\beta,m ,h_n}:=2^{\beta} 4^m  (-h_n)_{m}.
$$
 We assume $T \in \mathcal{BC}^{0,1}(V_n)$ and $s \in \rho_S(T)$. If $\beta=2k_1+1$, where $k_1\in\mathbb N$, we define the $\overline{D}$-resolvent operators as
	
		\begin{equation}
		S^{-1}_{L, \overline {D}^{\beta} \Delta^m}(s,T) =  \mathbf{c}_{\beta,m ,h_n} \left(  \sum_{j=0}^{k_1} \mathbf{a}^1_{j,k_1,m} K^{1,L}_{2j+1, m+j+2+k_1}(s,T)  +\sum_{j=0}^{k_1} \mathbf{b}^1_{j,k_1,m} K^2_{2j, m+j+1+k_1}(s,T) \right),
	\end{equation}
	and	if $\beta=2k_2$, where $k_2\in\mathbb N$, we have
	\begin{equation}
		S^{-1}_{L, \overline D^{\beta} \Delta^m} (s,T) = \mathbf{c}_{\beta,m ,h_n} \left(  \sum_{j=0}^{k_2} \mathbf{a}^2_{j,k_2,m} K^{1,L}_{2j, m+j+1+k_2}(s,T) +\sum_{j=0}^{k_2-1} \mathbf{b}^2_{j,k_2,m} K^2_{2j+1,m+1+k_2+j}(s,T) \right),
	\end{equation}
	
where the coefficients $\mathbf{a}^1_{j,k_1,m}$, $\mathbf{a}^2_{j,k_2,m}$, $\mathbf{b}^1_{j,k_1,m}$ and $\mathbf{b}^2_{j,k_2,m}$ are defined in \eqref{c1} and \eqref{c2}.

\end{definition}	

\begin{remark}
If $\beta = 1$ in the $D$-resolvent operators, we recover the \emph{Poly-Harmonic-resolvent operator}, given by
$$
S^{-1}_{L, D \Delta^m}(s,T) = -2(h_n-m) \gamma_{n,m}\mathcal{Q}_{c,s}^{-m-1}(T),
$$
see \cite[Def.~5.3]{CDP25}. If instead $\beta = h_n - m$ in the $\overline{D}$-resolvent operators, we obtain the \emph{resolvent operator of the Poly-Analytic functional calculus}, given by
$$
S^{-1}_{L, \overline{D}^{h_n - m} \Delta^m}(s,T) = 4^{h_n} h_n! (-h_n)_m \, (s\mathcal{I} - \bar{T})(s\mathcal{I} - T_0)^{h_n - m} \mathcal{Q}_{c,s}^{-h_n - 1}(T),
$$
see \cite[Def.~9.1]{CDP25}. For other values of $\beta$ and $n = 5$, the $D$- and $\overline{D}$-resolvent operators were derived in \cite{Fivedim}.

\end{remark}

\begin{remark}
Using the series expansions of the $ K^{1,L}_{\nu, \ell}$ and$K^2_{\nu, \ell}$ resolvent operators given in \eqref{res01} and \eqref{res2}, one can also derive the corresponding series expansions for the $ D $ and $\overline{D} $-resolvent operators.
\end{remark}

To construct new functional calculi based on the resolvent operators defined in Definition~\ref{Dresope} and Definition~\ref{Dbaroperes}, it is necessary to establish the slice hyperholomorphicity of the $D$ and $\overline{D}$-resolvent operators with respect to the variable $s$.

\begin{lemma}
	\label{regope}
	Let $n$ be an odd number, set $h_n:=(n-1)/2$ and $T \in \mathcal{BC}^{0,1}(V_n)$. The resolvent operators $S^{-1}_{L, D^\beta \Delta^m} (s,T)$ and $S^{-1}_{L,\overline{D}^\beta \Delta^m} (s,T)$, with $m,\,\beta \in \mathbb{N}_0$ be such that $\beta+m\leq h_n$, are a $ \mathcal{B}(V_n)$-valued right slice hyperholomorphic in the variable $s$ in $ \rho_S(T)$.
\end{lemma}
\begin{proof}
It is sufficient to show that
$$
K^{1,L}_{\nu,\ell}(s,T) = (s\mathcal{I} - \bar{T})(s\mathcal{I} - T_0)^\nu \mathcal{Q}_{c,s}^{-\ell}(T)
\quad \text{and} \quad
K^2_{\nu, \ell}(s,T) = (s\mathcal{I} - T_0)^\nu \mathcal{Q}_{c,s}^{-\ell}(T),
$$
for $\nu \in \mathbb{N}_0 $ and $ \ell \leq h_n $, are $\mathcal{B}(V_n)$-valued right slice hyperholomorphic functions in the variable $s$. From the facts that both $ (s\mathcal{I} - T_0)^{\nu} $ and $ \mathcal{Q}_{c,s}^{-\ell}(T) $ are $ \mathcal{B}(V_n)$-intrinsic slice hyperholomorphic in $ s$ (see Proposition~\ref{Qint}), we get that $K^2_{\nu,\ell}(s,T) $ is $ \mathcal{B}(V_n)$-valued and intrinsic slice hyperholomorphic in $ s$. Moreover, since $ K^{1,L}_{\nu, \ell}(s,T) $ can be written as
$$
K^{1,L}_{\nu, \ell}(s,T) = K^2_{\nu, \ell}(s,T) \, s - \bar{T} \, K^2_{\nu, \ell}(s,T),
$$
it is evident that $ K^{1,L}_{\nu, \ell}(s,T)$ is a $ \mathcal{B}(V_n)$-valued right slice hyperholomorphic function in the variable $s$.

\end{proof}

\begin{definition}[$D$-functional calculi on the $S$-spectrum]
		\label{D_funct_cal}
Let $n$ be an odd number and set $h_n:=(n-1)/2$. Let $m$ $\beta \in \mathbb{N}_0$ be such that $\beta+m\leq h_n$. Let $T \in \mathcal{BC}^{0,1}(V_n)$ be such that its components  $T_i$, $i=0,...,n-1$ have real spectra and $T_n=0$. Let $U$ be a bounded slice Cauchy domain, not necessarily connected, and set $ds_I=ds(-I)$ for $I\in \mathbb{S}$. For any $f \in \mathcal{SH}^L_{\sigma_S(T)}(U)$ we set
$$
f_{\beta, m}(x)=D^\beta \Delta^{m}f(x).
$$
 We define $D$-functional calculus for the operator $T$ as
	\begin{equation}
	\label{I1}
	f_{\beta, m}(T):=\frac{1}{2\pi} \int_{\partial(U \cap \mathbb{C}_I)}  S^{-1}_{L,D^\beta\Delta^m} (s,T)ds_I f(s),
\end{equation}
where the $D$-resolvent operators are defined in Definition \ref{Dresope}.
\end{definition}

\begin{definition}[$\overline {D}$-functional calculi on the $S$-spectrum]
Let $n $ be an odd number and set $ h_n := \frac{n-1}{2}$. Let $m$ $\beta \in \mathbb{N}_0$ be such that $\beta+m\leq h_n$. Let $T \in \mathcal{BC}^{0,1}(V_n)$ be an operator whose components $T_i$, $i=0,...,n-1$ have real spectra and $T_n=0$. Let $U$ be a bounded slice Cauchy domain, not necessarily connected, and set $ds_I := ds(-I)$ for $I \in \mathbb{S} $.
For any function $f \in \mathcal{SH}^L_{\sigma_S(T)}(U) $, define
$$
f_{\overline{\beta}, m}(x) := \overline{D}^\beta \Delta^m f(x).
$$
We then define the $ \overline{D}$-functional calculi for the operator $T$ as
\begin{equation}
\label{T1_bar}
f_{\overline{\beta}, m}(T) := \frac{1}{2\pi} \int_{\partial(U \cap \mathbb{C}_I)} S^{-1}_{L,\overline{D}^\beta\Delta^m}(s,T) \, ds_I \, f(s),
\end{equation}
where the $\overline{D}$-resolvent operators are defined in Definition \ref{Dbaroperes}.
\end{definition}

\begin{definition}
We denote all the $D$ and $\overline{D}$-functional calculi as Dirac functional calculi.
\end{definition}

\begin{remark}
Although the condition that $T_i$, for $i = 0, \dots, n-1$, have real spectra and $T_n=0$ may appear artificial, it is in fact crucial for proving the independence from the kernels of $\overline{D}^\beta \Delta^m$ and $D^\beta \Delta^m $, when $U$ is not connected, see also Remark~\ref{MCzero}.

\end{remark}

\begin{remark}
The $ D $-functional calculus includes, as a particular case when $ \beta = 1 $, the Poly-Harmonic functional calculus developed in \cite{CDP25}. On the other hand, the $\overline{D} $-functional calculus contains, as a special case when $\beta = h_n - m $, the Poly-Analytic functional calculus introduced in \cite{CDP25}.

\end{remark}

Now, we prove some basic properties of the Dirac functional calculi.

\begin{theorem}
	\label{well1}
	The Dirac functional calculi based on the $S$-spectrum, i.e. the integrals \eqref{I1} and \eqref{T1_bar}, depend neither on the imaginary units $I \in \mathbb{S}$ nor on the slice Cauchy domain $U$.
\end{theorem}
\begin{proof}
	We prove the result for $D$-functional calculi but a similar proof also holds for the $\overline{D}$-functional calculi. We first show that the integral \eqref{I1} does not depend on the slice domain $U$. Let $U'$ be another bounded slice Cauchy domain with $\sigma_S(T)\subset U'$ and $\overline{U'}\subset \mathcal{D}(f)$, and let us assume for the moment $\overline{U'}\subset U$. Then $O:=U\setminus U'$ is again a bounded slice Cauchy domain and we have $\overline O\subset\rho_S(T)$ and $\overline O\subset \mathcal{D}(f)$. Hence the function $f(s)$ is left slice hyperholomorphic and the $D$-resolvent operators $S^{-1}_{L, D^\beta \Delta^m}(s,T)$ are $\mathcal{B}(V_n)$ right slice hyperholomorphic in the variable $s$ on $\overline O$.
	The Cauchy integral theorem (see Theorem \ref{Cif}) therefore implies
	\begin{eqnarray*}
		0 & =& \frac 1{2\pi} \int_{\partial (O\cap \mathbb C_I)} S^{-1}_{L, D^\beta \Delta^m}(s,T) \, ds_I\, f(s)\\
		&=& \frac 1{2\pi} \int_{\partial (U\cap \mathbb C_I)} S^{-1}_{L, D^\beta \Delta^m}(s,T) \, ds_I\, f(s)-\frac 1{2\pi} \int_{\partial (U'\cap \mathbb C_I)} S^{-1}_{L, D^\beta \Delta^m}(s,T) \, ds_I\, f(s).
	\end{eqnarray*}
	If $\overline {U'} \not\subset U$, then $O:=U'\cap U$ is an axially symmetric that contains $\sigma_S(T)$. As \cite[Remark 3.2.4]{CGK}, we can hence find a third slice Cauchy domain $U''$ with $\sigma_S(T)\subset U''$ and $\overline{U''}\subset O=U\cap U'$. The above arguments show that the integrals over the boundaries of all three sets agree.
	
	In order to show the independence of the imaginary units, we choose two units $I, J\in\mathbb S$ and the two slice Cauchy domains $U_x, U_s\subset \mathcal{D}(f)$ with $\sigma_S(T)\subset U_x$ and $\overline{U_x} \subset U_s$ (the subscripts $s$ and $x$ are chosen in order to indicate the variable of integration in the following computation). The set $U^c_x=\mathbb R^{n+1}\setminus U_x$ is then an unbounded axially symmetric slice Cauchy domain with $\overline{U^c_x}\subset \rho_S(T)$. The $D$-resolvent operators are $\mathcal{B}(V_n)$-valued right slice hyperholomorphic on $\rho_S(T)$, see Lemma \ref{regope}. Moreover at infinity we have
	$$
	\lim_{s\to\infty} S^{-1}_{L,\, D^{\beta}\Delta^m}(s,T)=0.
	$$
	The slice hyperholomorphic Cauchy formula (see \eqref{cauchynuovo}) implies that
	$$
	 S^{-1}_{L, D^\beta \Delta^m}(s,T)=\frac 1{2\pi}\int_{\partial (U^c_x\cap\mathbb C_I)} S^{-1}_{L, D^\beta \Delta^m}(x,T)\, dx_I S^{-1}_R(x,s)
	$$
	for every $s\in U^c_x$. By \cite[Cor. 2.1.26]{CGK} we have $S^{-1}_R(x,s)=-S^{-1}_L(s,x)$, and since
$$\partial(U^c_x\cap\mathbb C_J)=-\partial(U_x\cap \mathbb C_J)$$ we have
	\[
	\begin{split}
		f_{\beta,\, m}(T)&=\frac 1{2\pi} \int_{\partial(U_s\cap \mathbb C_J)} S^{-1}_{L, D^\beta \Delta^m}(s,T)\, ds_J\, f(s)\\
		&= \frac 1{(2\pi)^2}\int_{\partial(U_s\cap\mathbb C_J)} \left( \int_{\partial (U^c_x\cap\mathbb C_I)} S^{-1}_{L, D^\beta \Delta^m}(x,T) dx_I S^{-1}_R(x,s) \right) \,ds_J f(s)\\
		&= \frac 1{(2\pi)^2}\int_{\partial(U_x\cap\mathbb C_I)} S^{-1}_{L, D^\beta \Delta^m}(x,T) dx_I \left( \int_{\partial (U_s\cap\mathbb C_J)}  S^{-1}_L(s,x)  \,ds_J f(s) \right)\\
		&=\frac 1{2\pi} \int_{\partial( U_x\cap \mathbb C_I)} S^{-1}_{L, D^\beta \Delta^m}(x,T)\, dx_I\, f(x),
	\end{split}
	\]
	where the last identity follows again from the slice hyperholomorphic Cauchy formula because we chose $\overline{U_x}\subset U_s$.
\end{proof}

To demonstrate that the Dirac-functional calculi are independent of the specific kernels of the operators from which they are derived, the following results play a crucial role.

\begin{lemma}
	\label{blockzero}
	Let $n$ be an odd number and set $h_n:=(n-1)/2$. Let $\nu \in \mathbb{N}$ and $\ell \leq h_n$ such that $2 \ell-\nu-3 \geq 0$. We assume that the operators $T_{i}$, for $i=0,...,n-1$, have real spectrum and $T_n=0$. Let $U$ be a bounded slice Cauchy domain such that $(\partial U) \cap \sigma_S(T)= \emptyset$. Thus, for every $I \in \mathbb{S}$ we have
	\begin{equation}
		\label{zer1}
		\int_{\partial(U \cap \mathbb{C}_I)} K^{2}_{\nu,\ell}(s,T) ds_I s^\gamma=0, \quad \hbox{if} \quad 0 \leq \gamma \leq 2\ell-2-\nu.
	\end{equation}
	and
	\begin{equation}
		\label{zer2}
		\int_{\partial(U \cap \mathbb{C}_I)} K^{1,L}_{\nu,\ell}(s,T) ds_I s^\gamma=0, \quad \hbox{if} \quad  0\leq \gamma  \leq 2\ell-3-\nu.
	\end{equation}
\end{lemma}	
\begin{proof}
	To show this result we need to recall the following facts:
	\begin{equation}
		\label{zero1}
		\int_{\partial(U \cap \mathbb{C}_I)} \mathcal{Q}_{c,s}^{-\ell}(T) ds_Is^\alpha=0, \quad \hbox{if} \quad 0\leq \alpha \leq 2\ell-2,
	\end{equation}
	and
	\begin{equation}
		\label{zero2}
		\int_{\partial(U \cap \mathbb{C}_I)} (s\mathcal{I}-\bar{T})\mathcal{Q}_{c,s}^{-\ell}(T) ds_Is^\alpha=0, \quad \hbox{if} \quad 0\leq \alpha \leq 2\ell-3,
	\end{equation}
	see \cite[Thm. 5.12]{CDP25} and \cite[Thm. 7.8]{CDP25}, respectively. We start proving \eqref{zer1}. By the binomial theorem we have
	\begin{equation}
		\label{z1}
		\int_{\partial(U \cap \mathbb{C}_I)} K^{2}_{\nu,\ell}(s,T) ds_I s^\gamma= \sum_{t=0}^{\nu} \binom{\nu}{t}T_0^{\nu-t}\int_{\partial(U \cap \mathbb{C}_I)} \mathcal{Q}_{c,s}^{-\ell}(T)ds_Is^{\gamma+t}.
	\end{equation}
	The integral of the right-hand side of \eqref{z1} is zero by \eqref{zero1} and the fact that
	$ \gamma +t \leq \gamma +\nu \leq 2 \ell-2.$
	Now, we prove \eqref{zer2}. By the binomial theorem we have
	\begin{equation}
		\label{z2}
		\int_{\partial(U \cap \mathbb{C}_I)} K^{1,L}_{\nu,\ell}(s,T) ds_I s^\gamma=\sum_{t=0}^{\nu} \binom{\nu}{t}T_0^{\nu-t}\int_{\partial(U \cap \mathbb{C}_I)} (s\mathcal{I}-\bar{T})\mathcal{Q}_{c,s}^{-\ell}(T)ds_Is^{\gamma+t}.
	\end{equation}
	The integral of the right-hand side of \eqref{z2} is zero by \eqref{zero2} and the fact that
	$ \gamma +t \leq \gamma +\nu \leq 2 \ell-3.$
	
\end{proof}

\begin{remark}
\label{MCzero}
The proofs of \eqref{zero1} and \eqref{zero2} were originally established for operators of the form $ T = T_1 e_1 + \dots + T_n e_n$, but they can also be done for operators of the form $T = T_0 + T_1 e_1 + \dots + T_{n-1} e_{n-1}$ (where the components have real spectra) by employing analogous arguments. A key ingredient in establishing this result is the use of the monogenic functional calculus; see~\cite{JBOOK, JM}. This calculus requires, as a fundamental assumption, that $T_n=0$ and $T_i$, for $ i = 0, \ldots, n-1 $, have real spectrum.
\end{remark}

\begin{proposition}
	\label{zerothm}
Let $n$ be an odd number and set $h_n:=(n-1)/2$. Let $m$ $\beta \in \mathbb{N}_0$ be such that such that $\beta+m\leq h_n$. We assume that the operators $T_{i}$, for $i=0,...,n-1$, have real spectrum. Let $U$ be a bounded slice Cauchy domain such that $(\partial U) \cap \sigma_S(T)= \emptyset$. Thus, for every $I \in \mathbb{S}$ we have
\begin{equation}
\label{intzero}
\int_{\partial (U \cap \mathbb{C}_I)} S^{-1}_{L, D^{\beta}\Delta^m}(s,T) ds_I s^{\alpha}=0, \qquad \hbox{if} \quad 0 \leq \alpha \leq \beta+2m-1.
\end{equation}
	Thus, for every $I \in \mathbb{S}$, for $0\leq \beta\leq h_n$ and $\beta +m\leq h_n$, we have
\begin{equation}
\label{intzero2}
 \int_{\partial (U \cap \mathbb{C}_I)} S^{-1}_{L, \overline{D}^{\beta}\Delta^m}(s,T) ds_I s^{\alpha}=0, \qquad \hbox{if} \quad 0 \leq \alpha \leq \beta+2m-1.
\end{equation}
\end{proposition}
\begin{proof}
	The result follows by Lemma \ref{blockzero}, Definition \ref{Dresope}, Definition \ref{Dbaroperes} and the linearity of the integrals.
\end{proof}

We are now equipped with all the necessary tools to establish the following key result, which, together with Theorem~\ref{well1}, guarantees that the Dirac functional calculi are well defined.
 
\begin{proposition}[Independence of the representative]
	\label{indep_repr}
Let $n$ be an odd number and set $h_n:=(n-1)/2$. Let $m,\,\beta \in \mathbb{N}_0$ be such that $\beta+m\leq h_n$. We assume $T \in \mathcal{BC}^{0,1}(V_n)$ be such $T_{i}$, $i=0,...,n-1$ have real spectra.
	 Let $U$ be a slice Cauchy domain with $\sigma_S(T) \subset B_{r}(0) \subset U$, with $\|T\|<r$. Let $f$, $g \in \mathcal{SH}^L_{\sigma_S(T)}(U)$ and such that
	$ D^\beta \Delta^{m}_{n+1} f(x)=D^\beta \Delta^{m}_{n+1} g(x)$  or $ \overline{D}^\beta \Delta^{m}_{n+1} f(x)=\overline {D}^\beta \Delta^{m}_{n+1} g(x)$, then we have $f_{\beta,\, m}(T)=g_{\beta,m}(T)$ or $f_{\overline{\beta},\, m}(T)=g_{\overline{\beta},m}(T)$, respectively.
\end{proposition}
\begin{proof}
	We prove the results for the $D$-functional calculi but a similar proof also holds for the $\overline{D}$-functional calculi. We assume that $U$ is connected. By the definition of the $D$-functional calculi, see Definition \ref{D_funct_cal}, we have
	$$ f_{\beta,m}(T)-g_{\beta,m}(T)= \frac{1}{2\pi} \int_{\partial(U \cap \mathbb{C}_I)} S^{-1}_{L, D^{\beta}\Delta^m}(s,T) ds_I(f(s)-g(s)).$$
	By Lemma \ref{regope}, the resolvent operator $S^{-1}_{L,D^{\beta}\Delta^m}(s,T)$ is slice hyperholomorphic in $s$. By hypothesis, $f-g$ belongs to the kernel of $D^\beta\Delta^m$; hence, by Lemma \ref{kernel}, $f-g$ is a polynomial. Therefore, we can replace the domain of integration by $B_r(0)\cap\mathbb{C}_I$, where $r>0$ is chosen such that $|T|<r$ and $\overline{U}\subset B_r(0)$. Thus, we can write
	$$f_{\beta,m}(T)-g_{\beta,m}(T)=\frac{1}{2\pi}  \sum_{\nu=0}^{\beta+2m-1} \int_{\partial(B_r(0) \cap \mathbb{C}_I)}S^{-1}_{L, D^{\beta}\Delta^m}(s,T) \, ds_I\, s^{\nu} \alpha_{\nu}, $$
	where $\{\alpha_{\nu}\}_{0 \leq \nu \leq \beta+2m-1} \subseteq \mathbb{R}_n$. Since $ \nu \leq \beta+2m-1$, by Proposition \ref{zerothm} we know that the previous integral is zero, and so we get
	$$f_{\beta,m}(T)=g_{\beta,m}(T).$$
	In  the case the set $U$ is not connected we can write $U=\bigcup_{\ell=1}^n U_{\ell}$ where $U_{\ell}$ are the connected components of the set $U$. So by Lemma \ref{kernel} we have
	$$
	f(s)-g(s)=\sum_{\tau=1}^{n} \sum_{\nu=0}^{\beta+2 m-1} \chi_{U_{\tau}} (s) s^{\nu}\alpha_{\nu,\tau}.
	$$
	Hence by the definition of the $D$-functional calculi, see Definition \ref{D_funct_cal}, we have
	$$ f_{\beta, m}(T)-g_{\beta, m}(T)= \frac{1}{2\pi} \sum_{\tau=1}^{n} \sum_{\nu=0}^{\beta+2m -1} \int_{\partial(U_{\tau} \cap \mathbb{C}_I)} S^{-1}_{L, D^\beta\Delta^m}(s,T)ds_I s^{\nu}\alpha_{\nu,\tau}.$$
	Finally, by Proposition \ref{zerothm} we get that $f_{\beta,m}(T)=g_{\beta,m}(T)$.
\end{proof}

An easy algebraic property of the Dirac functional calculi is stated in the following result:

\begin{lemma}
	
	Let $n$ be an odd number and set $h_n := \frac{n-1}{2}$. Suppose $m,\,\beta \in \mathbb{N}_0$ satisfy $\beta + m \leq h_n$. Assume $T \in \mathcal{BC}^1(V_n)$ is such that each component $T_i$, for $i = 0, \ldots, n-1$, has real spectrum.
Define
$$
f_{\beta,m}(x) := D^\beta \Delta^{m}f(x),\ \ \ \   g_{\beta,m}(x) := D^\beta \Delta^{m}g(x),
$$
 and
 $$
 f_{\overline{\beta},m}(x) := \overline{D}^\beta \Delta^{m}f(x),\ \ \ g_{\overline{\beta},m}(x) := \overline{D}^\beta \Delta^{m}g(x),
 $$
 where $f, g \in \mathcal{SH}^L_{\sigma_S(T)}(U)$ and $a \in \mathbb{R}_n$. Then we have
	$$ (f_{\beta,m}a+g_{\beta,m})(T)=f_{\beta,m}(T)a+g_{\beta,m}(T)$$
	and
	$$
	(f_{\overline \beta,m}a+g_{\overline\beta,m})(T)=f_{\overline{\beta},m}(T)a+g_{\overline{\beta},m}(T).
	$$
\end{lemma}
\begin{proof}
	The result follows from the linearity of the integrals \eqref{I1} and \eqref{T1_bar}.
\end{proof}

The generalization of the holomorphic functional calculus to sectorial operators leads
	to the $H^\infty$-functional calculus that, in the complex setting, was introduced in the paper \cite{McI1}, see also the books
\cite{Haase, HYTONBOOK1, HYTONBOOK2}. Moreover, the boundedness  of the $H^\infty$ functional calculus depends on suitable quadratic
estimates and this calculus has several applications to boundary value problems, see  \cite{MC10,MC97,MC06}.
The functional calculi introduced in this paper admit an $H^\infty$-extension, in the spirit of the $S$-functional calculus \cite{ACQS2016,MANSCH25} as extended for the quaternionic fine structures of Dirac type treated in \cite{CPS1,MPS23}.
We remark that the $H^\infty$-functional calculus exists also for the
monogenic functional calculus (see \cite{JM}) and it was introduced by A. McIntosh and his collaborators, see the books \cite{JBOOK,TAOBOOK} for more details.

\section*{Declarations and statements}

{\bf Data availability}. The research in this paper does not imply use of data.

{\bf Conflict of interest}. The authors declare that there is no conflict of interest.

{\bf Acknowledgments}.
F. Colombo,  A. De Martino and S. Pinton are supported by MUR grant Dipartimento di Eccellenza 2023-2027. The authors are very grateful to the anonymous referee for reading the paper in great detail and for the insightful suggestions.

\end{document}